%% file: main.tex
\documentclass[11pt,a4paper,dvipsnames]{article}
\input{preamble.tex}
\title{\sffamily\bfseries Minimax~Adaptive~Control~for~Finite~Sets~of~Positive\\Linear~Systems\\[2ex]}
\author{
    Fethi Bencherki$^{\star}$ \and  
    Anders Rantzer$^{\star}$
    \\[2ex]
}

\date{%
    $^{\star}$Department of Automatic Control, Lund University, Lund, Sweden \\
    \texttt{\{\href{mailto:fethi.bencherki@control.lth.se}{fethi.bencherki}, \href{mailto:anders.rantzer@control.lth.se}{anders.rantzer}\}@control.lth.se} 
}
\usepackage{wrapfig}
\usepackage{float}
\usepackage{tikz}
\usepackage{graphicx}
\usepackage{amsmath,amssymb}
\usepackage{algorithm}
\usepackage{algpseudocode}
\usetikzlibrary{chains,shapes.multipart}
\usetikzlibrary{shapes,calc}
\usetikzlibrary{automata,positioning}
\usetikzlibrary{positioning, shapes, arrows, calc}
\tikzstyle{block} = [draw, rectangle, line width=0.7mm]
\usetikzlibrary{arrows.meta,positioning,calc,shapes.symbols}
\usepackage{amsthm}
\newtheorem{problem}{Problem}
\newenvironment{proofsketch}
{\begin{proof}[Proof sketch]}
{\end{proof}}
\begin{document}

\maketitle

\begin{abstract}
    \input{tex_input/abstract.tex}

\end{abstract}


\input{tex_input/introduction.tex}

\input{tex_input/contributions.tex}  
\input{tex_input/outline.tex}  
\input{tex_input/prelim.tex}
\input{tex_input/problem_form.tex}

\input{tex_input/reformulation.tex}
\input{tex_input/main_result1.tex}
\input{tex_input/main_result2.tex}

\input{tex_input/conclusions.tex}
\input{tex_input/acknowledgments.tex}  
\begin{appendix}
\input{tex_input/append_prelim}
\input{tex_input/append_main}

\end{appendix}
\begingroup
\sloppy
\printbibliography
\endgroup
\end{document}

%% file: preamble.tex
\usepackage[top=25mm, bottom=25mm, right=25mm, left=25mm]{geometry}
\usepackage[T1]{fontenc}
\usepackage[utf8]{inputenc}

\usepackage[english]{babel}      

\usepackage{lmodern}           
\usepackage{microtype}         
\usepackage{mathrsfs}          
\usepackage{amsthm, amsmath, amssymb, mathtools} 
\usepackage{bm, bbm}           
\usepackage{stmaryrd}          
\usepackage{accents}
\numberwithin{equation}{section}

\usepackage{cases,xfrac,gensymb} 
\usepackage{tocloft}             
\usepackage{listings}            
\usepackage{lipsum}              
\usepackage{titling}             
\usepackage{proof}               
\let\oldunderbar\underbar
\usepackage{sectsty}             
\let\underbar\oldunderbar
\usepackage{parskip}             
\usepackage{booktabs}            
\usepackage{enumerate}           
\usepackage{autobreak}           
\usepackage{xcolor}              
\usepackage{csquotes}            
\usepackage{adjustbox}

\usepackage[bf,sf]{titlesec}       
\usepackage[labelfont={bf,sf}]{caption} 
\usepackage{subcaption}            

\usepackage[
  maxnames=20,      
  giveninits=true,  
  isbn=false,        
]{biblatex}
\usepackage{graphicx}            
\usepackage{float}               
\usepackage{tikz}                
\usetikzlibrary{external}

\usepackage{pgfplots}
\pgfplotsset{compat=1.18}
\usepgfplotslibrary{colormaps}
\usepackage{pgfplotstable}
\usepgfplotslibrary{fillbetween}

\let\mathbf\bm
\DeclareMathAlphabet\mathbfcal{OMS}{cmsy}{b}{n} 

\input{commands.tex}

\makeatletter
\def\th@plain{
  \thm@headfont{\normalfont\sffamily\bfseries}%
  \itshape
}
\def\th@definition{
  \thm@headfont{\normalfont\sffamily\bfseries}%
  \thm@notefont{\normalfont\sffamily\bfseries}%
}
\newtheoremstyle{myStyle1}
  {0.3cm}
  {0.3cm}
  {\itshape}
  {}
  {}
  {:}
  {.5em}
  {%
    \thmname   {\normalfont\sffamily\bfseries #1 }%
    \thmnumber {\normalfont\sffamily\bfseries #2}%
    \thmnote   {\normalfont\sffamily\bfseries \ (#3)}%
  }

\newtheoremstyle{myStyle2}
  {0.3cm}   
  {0.3cm}   
  {}        
  {}        
  {}
  {:}       
  {.5em}    
  {%
    \thmname   {\normalfont\sffamily\bfseries #1 }%
    \thmnumber {\normalfont\sffamily\bfseries #2}%
    \thmnote   {\normalfont\sffamily\bfseries \ (#3)}%
  }
\makeatother

\theoremstyle{myStyle1}
\newtheorem{theorem}{Theorem}
\newtheorem{lemma}{Lemma}

\newtheorem{corollary}{Corollary}
\newtheorem{assumption}{Assumption}

\theoremstyle{myStyle2}
\newtheorem{definition}{Definition}
\newtheorem{remark}{Remark}
\newtheorem{example}{Example}

\providecommand{\keywords}[1]{\textbf{\textsf{Keywords.}} #1}

\definecolor{Red}{rgb}{1, 0, 0}
\definecolor{BlueIMT}{RGB}{0,42,72}
\definecolor{ForestGreen}{rgb}{0.13, 0.55, 0.13}
\definecolor{Gray}{rgb}{0.66, 0.66, 0.66}
\definecolor{Green}{rgb}{0.0, 0.5, 0.0}
\definecolor{MidnightBlue}{rgb}{0.098, 0.098, 0.439}
\definecolor{Orange}{rgb}{0.93, 0.53, 0.18}

\definecolor{color1}{RGB}{68,119,170}   
\definecolor{color2}{RGB}{102,204,238}  
\definecolor{color3}{RGB}{34,136,51}    
\definecolor{color4}{RGB}{17,119,51}    
\definecolor{color5}{RGB}{204,187,68}   
\definecolor{color6}{RGB}{221,204,119}  
\definecolor{color7}{RGB}{204,102,119}  
\definecolor{color8}{RGB}{136,34,85}    
\definecolor{color9}{RGB}{170,51,119}   
\definecolor{color10}{RGB}{102,102,102} 
\definecolor{color11}{RGB}{50,50,50} 

\usepackage{hyperref} 
\hypersetup{%
    hidelinks,
    hypertexnames = true,
    plainpages    = false,
    colorlinks = true,
    urlcolor   = MidnightBlue,
    linkcolor  = MidnightBlue,
    citecolor  = ForestGreen,
}
\usepackage[capitalize,nameinlink]{cleveref}

\newcommand{\overbar}[1]{\mkern 1.5mu\overline{\mkern-1.5mu#1\mkern-1.5mu}\mkern 1.5mu}
\newboolean{showcomments}
\setboolean{showcomments}{true}
\newcommand{\fethi}[1]{\ifthenelse{\boolean{showcomments}}
	{\textcolor{blue}{(Fethi says: #1)}}{}}

\usepackage{comment}

%% file: commands.tex
\newcommand{\abs}[1]{\left\lvert#1\right\rvert}


%% file: tex_input/abstract.tex
We present a minimax adaptive control framework for discrete-time positive linear systems with parametric uncertainty and adversarial disturbances. The uncertainty in the system dynamics is assumed to lie in a finite set of possible plants. We formulate the problem as a dynamic game between the controller, which minimizes the cost, and an adversary, which selects both the disturbances and the plant dynamics to maximize the cost. An equivalent reformulation of the original game transforms the problem into a standard minimax two-player zero-sum dynamic game. This enables the problem to be addressed via minimax dynamic programming. We provide an explicit solution to the Bellman inequality, yielding stabilizing, positivity preserving policies without requiring an initially stabilizing controller. The resulting controller enjoys robustness guarantees in the form of bounded \(\ell_1\)-gain from disturbances to errors. Once the uncertain parameters have been sufficiently estimated, the controller behaves like a standard \(\mathcal H_\infty\)-type optimal controller for positive linear systems. The theoretical findings are supported by numerical experiments illustrating the resulting adaptive controller in action.

%% file: tex_input/introduction.tex
\section{Introduction}
Technological advances have led to engineering systems and networks of increasing scale, interconnections, and dynamical complexity. In many applications, the underlying physics are only partially understood, and the resulting mathematical models are therefore only approximate. As a consequence, uncertainties introduced at the modeling stage must be accounted for in controller design. This has motivated a substantial literature on robust and adaptive control, leading to the seminal work on self-tuning regulators~\cite{aastrom1973self} and followed by important contributions on stochastic adaptive control and convergence analysis~\cite{goodwin1981discrete,narendra1997adaptive,guo1995convergence}. Comprehensive treatments of the topic can be found in the classic textbooks~\cite{astrom1994adaptive,goodwin2014adaptive} and in more recent surveys that connect adaptive control with modern learning-based control and finite-sample analysis~\cite{annaswamy2021historical,annaswamy2023adaptive,matni2019self,tsiamis2023statistical}.

\medskip

A central theme in adaptive control is the interplay between regulation and learning. A controller must maintain good closed loop performance using the information currently available, while also gathering informative data that improves its internal model of the plant over time.
 This tradeoff has been studied extensively since the early self-tuning regulator literature, and it remains central to modern adaptive control, learning-based synthesis, and finite-sample analysis~\cite{aastrom1973self,goodwin1981discrete,narendra1997adaptive,guo1995convergence,annaswamy2023adaptive}. More recent work has sharpened this perspective by clarifying how finite horizon guarantees, regret bounds, and robustness considerations shape adaptive control design~\cite{tsiamis2023statistical,annaswamy2023adaptive}.

\medskip

The minimax adaptive control framework, first introduced in~\cite{didinsky1994minimax}, addresses parametric uncertainties and disturbances in a worst-case sense. Cast in a game-theoretic framework~\cite{bacsar2008h}, the controller is the minimizing player and the disturbances act as the maximizing player. This perspective has recently been revisited with an emphasis on discrete-time minimax optimization~\cite{rantzer2021minimax}. One major advantage is that minimax adaptive controllers do not require an initially stabilizing policy to provide provable performance guarantees. This makes the framework especially attractive in settings with large modeling errors or poorly modeled dynamics. Related work has used this formulation to derive explicit $\ell_2$-gain guarantees and robust simultaneous stabilization results for finite sets of LTI plants~\cite{bencherki2023robust,cederberg2022synthesis,renganathan2023online}.

\medskip

Positive systems are a structured class of dynamical models in which nonnegative initial conditions and positivity-preserving inputs guarantee nonnegative states and outputs. This makes them a natural modeling framework for queueing, routing, and flow networks~\cite{tassiulas1996throughput,arneson2016routing}; traffic network models~\cite{coogan2015traffic}; epidemic and compartmental dynamics~\cite{coogan2015traffic}; and more general monotone processes on nonnegative state spaces~\cite{rantzer2015scalable,sadraddini2019formal}. Their structure can also be exploited to obtain scalable analysis and synthesis methods, since stability and controller design is conducive to scalability~\cite{rantzer2015scalable,arneson2016routing,sadraddini2019formal}.
In particular, positive systems admit linear Lyapunov functions, linear programming based controller synthesis, and distributed control architectures that are computationally efficient than general quadratic or semidefinite formulations~\cite{berman1994nonnegative,farina2011positive,rantzer2021scalable,Rantzer2022,li2024exact}. 
Recent work has also developed data-driven and adaptive methods for positive systems, including direct stabilization from persistently exciting data and linear programming/copositive Lyapunov synthesis~\cite{Shafai2022,Miller2023}, data informativity conditions for distributed positive stabilization~\cite{Iwata2024}, data-driven identification and positive subspace system identification from input-output data~\cite{WangShafai2024a,WangShafai2024b}, dynamic output feedback synthesis for externally positive systems~\cite{MakdahPasqualetti2023}, and robust adaptive online control based on data-driven Bellman/Q-learning equations~\cite{BencherkiRantzer2024a,bencherki2025adaptive}.

\medskip

Motivated by these developments, we propose a framework for minimax adaptive control of discrete-time positive linear systems. The proposed controllers learn the system parameters online from incoming data while preserving positivity and robustness against adversarial disturbances. We do not impose statistical assumptions on the uncertain parameters or disturbances; instead, the model is deterministic, and the unknown quantities are selected by an adversary that seeks to maximize the cost function. By solving a minimax Bellman inequality, we derive explicit adaptive control laws that stabilize the system without requiring an initially stabilizing policy. We establish these results for the case in which the unknown system parameters belong to a finite set, and we show that the closed loop remains robust to unmodeled dynamics and mild modeling errors while remaining computationally tractable and scalable for large networked systems. See Fig.~\ref{fig:prob} for an illustration of the problem considered in the paper.
\begin{figure}[htb!]
    \centering
\begin{tikzpicture}[
    >=latex,
    block/.style={
        draw,
        line width=1.4pt,
        rectangle,
        inner xsep=8pt,
        inner ysep=6pt,
        align=center,
        font=\small
    },
    controller/.style={
        draw=blue!70!black,
        line width=1.2pt,
        rectangle,
        inner xsep=8pt,
        inner ysep=4pt,
        align=center,
        font=\small,
        fill=blue!10
    },
    adversary/.style={
        draw=red!60!black,
        line width=1.2pt,
        fill=red!15,
        cloud,
        cloud puffs=11,
        cloud puff arc=120,
        cloud ignores aspect,
        inner xsep=12pt,
        inner ysep=8pt,
        align=center,
        font=\small,
        text=red!50!black
    },
    redarrow/.style={->, line width=1.2pt, red!60!black},
    bluearrow/.style={->, line width=1.2pt, blue!70!black},
    blackarrow/.style={->, line width=1.2pt}
]

\node[block] (plant) at (0,0) {Positive\\Linear system};
\node[controller, text=blue!70!black] (ctrl) at (0,-1.55)
{Minimax Adaptive\\Controller};
\node[adversary] (adv) at (2.5,1.7) {Adversary};

\draw[blackarrow] (plant.west) -- ++(-2.4,0)
    node[midway, above] {Errors};

\draw[blackarrow] ($(plant.west)+(-1.5,0)$) |- (ctrl.west)
    node[pos=0.35, left] {State Measurement};

\draw[bluearrow] (ctrl.east) -- ++(0.85,0) |- ($(plant.east)+(0,-0.35)$);
\node[blue!70!black] at (4,-0.9) {Control action};

\draw[redarrow] (4.2,0) -- (plant.east)
    node[pos=0, above] {Disturbances};

\draw[redarrow] (0,1.5) -- (plant.north)
    node[midway, left] {System Parameters};

\end{tikzpicture}
\caption{Feedback interconnection of a discrete-time positive linear system with uncertain parameters, external disturbances, unmodelled dynamics, and an adaptive controller. The goal is to design the controller so as to minimize the worst case induced \(\ell_1\)-gain from disturbances to errors, while preserving positivity and accounting for parametric uncertainty.}
    \label{fig:prob}
\end{figure}
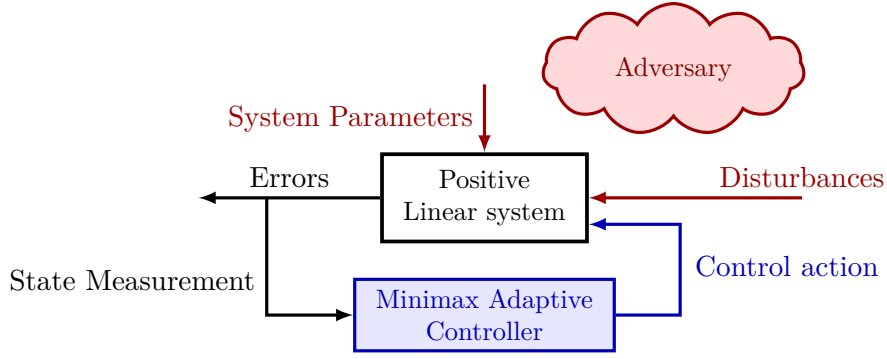

%% file: tex_input/contributions.tex
\subsection{Contributions}
The main contributions of this work are:
\begin{enumerate}
\item[\textit{(C1).}] \emph{Minimax adaptive dual control formulation for positive linear systems.} We formulate a minimax dynamic game for positive systems between an adversary, which selects the unknown dynamics and disturbance sequence in an adversarial manner, and a controller, which selects the control inputs. This formulation compels the controller to learn the true dynamics in order to stabilize the system, thereby giving rise to an exploration--exploitation tradeoff and a natural dual-control setting.

\item[\textit{(C2).}] \emph{Extension of \cite{gurpegui2025minimax} to positivity preserving disturbances.}
We show that, in the known system parameters case, our problem extends the setting of~\cite{gurpegui2025minimax} by allowing sign-indefinite disturbances, rather than only nonnegative disturbances as previously done, provided that they preserve state positivity. Despite this enlarged disturbance class, the resulting model based problem still admits a solution via linear programming, as in~\cite{gurpegui2025minimax}. This extension in the admissible disturbance class is essential for the success of the adaptive minimax approach proposed in this paper.

\item[\textit{(C3).}] \emph{Reformulation of the original problem.} We reformulate the original proposed problem in terms of the next state and model-specific nonnegative data histories, thereby obtaining a standard minimax zero-sum dynamic game between an adversary, which selects the positive next states, and a controller, which responds through its choice of inputs.

We establish the equivalence between the original problem and its reformulated counterpart, and show that solving one yields a solution to the other through minimax dynamic programming. This allows us to shift our focus to the reformulated problem.

\item[\textit{(C4).}] \emph{Explicit solutions to the Bellman inequality.}
We derive closed form adaptive policies by providing an explicit solution to
the associated Bellman inequality for systems with unknown input direction.
For such systems, we also quantify the minimum \(\ell_1\)-gain achievable in
the scalar case under the certainty equivalence policy, which provides an upper
bound on the optimal \(\ell_1\)-gain. We then generalize these results to a
finite set of linear time-invariant plants.

\item[\textit{(C5).}] \emph{Motivating applications.} We illustrate how the resulting adaptive policies can be deployed in practice by considering two representative positive linear system applications: a two reservoir network with uncertain transfer direction and a four mode multiclass job fluid queueing network inspired from~\cite{bertsimas2023optimal}.

\end{enumerate}

%% file: tex_input/outline.tex
\subsection{Organisation}
After introducing notation in Section~\ref{notation}, Section~\ref{pbf} presents the problem formulation. Section~\ref{reform} derives a data-history reformulation of the original problem, which enables its conversion into a standard positive minimax dynamic game formulation, addressed via minimax dynamic programming in Section~\ref{sec:dp-positive}. Section~\ref{bellman-ineq} presents our main results in the form of explicit solutions to the Bellman inequality, first for the special case of input sign uncertainty and then for the general case of finite set of positive LTI plants, accompanied by numerical examples. Finally, Section~\ref{conclus} concludes the paper and discusses directions for future research. Preliminary lemmas, their proofs, and the proofs of the main results are deferred to Appendices~\ref{appenA} and~\ref{appenB}. Appendix~\ref{model-based} addresses the model based counterpart of the problem
considered here.

%% file: tex_input/prelim.tex
\subsection{Notation and Preliminaries}\label{notation}
We let \(\mathbb{R}^n_+\) denote the nonnegative orthant in dimension \(n\). The vector \(\mathbf{1}\) denotes the all-ones vector, with dimension understood from context. For any vector \(x\), \(|x|\) denotes the elementwise absolute value. For a vector \(x\), \(\max x\) and \(\min x\) denote its largest and smallest entries, respectively. For a set \(\mathcal M\), \(|\mathcal M|\) denotes its cardinality. If $z_t^{(A,B)} \in \mathbb{R}_+$ is defined for every pair $(A,B)\in\mathcal M$, we write
$
Z_t \coloneqq \{z_t^{(A,B)}\}_{(A,B)\in\mathcal M}
$
for the indexed collection of nonnegative real numbers $z_t^{(A,B)}$, one for each $(A,B)\in\mathcal M$. The \(\ell_1\)-norm of a sequence \(w=\{w_t\}_{t=0}^\infty\) is given by
$
\|w\|_{1}
\coloneqq
\sum_{t=0}^\infty \mathbf{1}^\top |w_t|
=
\sum_{t=0}^\infty \|w_t\|_{1}.
$ Inequalities between vectors and matrices are understood elementwise.

%% file: tex_input/problem_form.tex
\section{Problem formulation}\label{pbf}
In this paper, the focus is the positive linear system class presented next.
\subsection{Positive system class}\label{pb_class}
We consider the discrete-time positive linear systems of the form
\begin{align*}
    x_{t+1} = A x_t + B u_t + w_t,\qquad x_0 \in \mathbb{R}^n_+
\end{align*}
where the dynamics are parameterized by \(A \in \mathbb{R}_+^{n \times n}\) and \(B \in \mathbb{R}^{n \times m}\). We constrain the input \(u_t\) to lie in a positivity preserving admissible input set \(\mathcal U(x_t)\) that depends on the current state $x_t \in \mathbb R_+^n$, where
\[
  \mathcal U(x)
  \subseteq
  \left\{
  u \in \mathbb{R}^m :
  Ax+Bu \geq 0,
  \ \text{for all } x\in\mathbb{R}_+^n
  \right\}.
\]
Thus, at each time \(t \in \mathbb{N}\), the control action satisfies \(u_t \in \mathcal U(x_t)\). Concrete choices of \(\mathcal U(x_t)\) are discussed and specified in section~\ref{sec31}.
The disturbance $w$ is restricted to a set \(\mathcal W(x,u)\) that preserves state positivity pointwise in time. Specifically, we define
\[
  \mathcal W(x,u)
  \coloneqq
  \left\{
  w \in \mathbb{R}^n :
  w \geq -(Ax+Bu),
  \ \text{for all } x\in\mathbb{R}_+^n,
  \ \text{for all } u\in\mathcal U(x)
  \right\}.
\]
These sets are such that, for every \(t \in \mathbb N\), the next state \(x_{t+1}\) remains nonnegative for all nonnegative \(x_t\), \(u_t \in \mathcal U(x_t)\) and \(w_t \in \mathcal W(x_t,u_t)\). 
In particular, \(w_t\) can be negatively large in magnitude as long as it does not render the state negative.
We assume that the true, unknown positive system belongs to a compact model class \(\mathcal M\) consisting of admissible system matrices \((A,B)\). To this end, we define the set $\mathcal{M}$ next.
\begin{definition}[Model set \(\mathcal{M}\)]\label{Def1}
Let
$
    \mathcal M \subseteq \mathbb{R}_+^{n \times n} \times \mathbb{R}^{n \times m}
$
be a compact set of admissible system parameters $(A,B)\in \mathcal{M}$.
\end{definition}
The set $\mathcal{M}$ captures the parametric uncertainty in the dynamics. 
\begin{assumption}\label{assump1}
The true unknown system parameters \((A_\star,B_\star)\) lie in
\(\mathcal M\), i.e., \((A_\star,B_\star)\in\mathcal M\).
\end{assumption}
In the next subsection, the minimax dual control problem considered in this paper is introduced.
\subsection{A minimax dual control problem for positive system}
We consider the following linear minimax dynamic game problem for the class of positive systems introduced above.
\begin{problem}\label{prob:l1-robust-control}
Find a causal policy \(\mu\) solving
\[
\begin{aligned}
 J_\ast(x_0)=\inf_{\mu}\ 
\sup_{\substack{(A,B)\in\mathcal M\\
 w \ge -(Ax+Bu)\\T\in\mathbb N}}
&\sum_{t=0}^{T}
\bigl( s^\top x_t + r^\top u_t - \gamma^\top |w_t| \bigr) \\
\text{s.t.}\quad
& x_{t+1}=Ax_t+Bu_t+w_t,\qquad x_0\in\mathbb R^n_+,\\
& u_t=\mu_t\bigl(x_0,\ldots,x_t;\,u_0,\ldots,u_{t-1}\bigr),\qquad \forall t\in\mathbb N,\\
& u_t\in\mathcal U(x_t),\qquad \forall t\in\mathbb N .
\end{aligned}
\]
\end{problem}
Note that in Problem~\ref{prob:l1-robust-control}, we restrict the policy map \(\mu\) to causal control policies
which are required to generate positivity preserving inputs \(u_t \in \mathcal U(x_t)\). Additionally, we require the worst case stage cost to be nonnegative, as stated in the following assumption.
\begin{assumption}[Positivity of the stage cost]\label{ass:stage cost-pos}
The stage cost weights satisfy \(s\in\mathbb{R}^n_+\), \(r\in\mathbb{R}^m\), and \(\gamma\in\mathbb{R}^n_+\). Moreover, for every \((A,B)\in\mathcal M\), \(t\in\mathbb N\), \(x_t\in\mathbb{R}^n_+\), and \(u_t\in\mathcal U(x_t)\), we assume
\[
   \max_{w_t \in \mathcal{W}(x_t,u_t)}
   \Bigl(s^\top x_t+r^\top u_t-\gamma^\top |w_t|\Bigr)
   \ge 0 .
\]
\end{assumption}
\begin{remark}
Since \(u_t\in\mathcal U(x_t)\) guarantees \(Ax_t+Bu_t\ge 0\), we have \(0\in\mathcal W(x_t,u_t)\). Therefore, Assumption~\ref{ass:stage cost-pos} simplifies to
$
s^\top x_t+r^\top u_t\ge 0.
$
\end{remark}
Unlike standard minimax dynamic games for positive linear systems
(cf.~\cite{bacsar2008h,gurpegui2023minimax}), in Problem~\ref{prob:l1-robust-control} the maximizing adversary selects not only the disturbance sequence \(w\), but also the unknown dynamics \((A,B) \in \mathcal{M}\). In the absence of uncertainty in \((A,B) \), the problem reduces to an \(\mathcal{H}_\infty\)-style optimal control problem for positive linear systems. When \((A,B)\) is unknown, however, the controller must learn the true dynamics in order to stabilize the system. This naturally leads to solutions in the form of adaptive dual controllers that balance learning and control, and can outperform linear state-feedback designs under model uncertainty~\cite{rantzer2021minimax}.

Next, we motivate the study of problems in the form of Problem~\ref{prob:l1-robust-control} by relating them to established problems in the optimal control of positive linear systems. This connection also suggests suitable choices for the admissible input set \(\mathcal U(x)\). 
\subsection{Connections to existing positive system formulations and classes}\label{sec31}

When the model is known, i.e., when the uncertainty set reduces to
$
    \mathcal M=\{(A_\star,B_\star)\},
$
Problem~\ref{prob:l1-robust-control} covers, as special cases, several classes of positive linear systems considered in the model-based optimal control over positive systems literature; see \cite{gurpegui2025minimax,gurpegui2023minimax,ohlin2024heuristic,ohlin2024optimal}. These special cases correspond to the following choices of the admissible input set \(\mathcal U(x)\):
\begin{itemize}
\item[\textit{(i)}] 
The works~\cite{gurpegui2026minimax,gurpegui2025minimax,gurpegui2023minimax} consider the admissible input set
\[
\mathcal U(x)
\coloneqq \bigl\{\, u \in \mathbb{R}^m : |u| \le Ex \,\bigr\},
\qquad E \in \mathbb{R}_{+}^{m\times n}.
\]
Together with the elementwise conditions
$
A \ge |B|E,
s > E^\top |r|.
$
This ensures positivity of the closed-loop system and nonnegativity of the stage cost.
Unlike \cite{gurpegui2025minimax}, which permits only nonnegative disturbances \(w_t\ge 0\), we allow disturbances of either sign, provided that they preserve positivity of the next state. This is encoded by the elementwise constraint
$
w_t \in \mathcal{W}(x_t,u_t),
\ \text{for all}\ t \in \mathbb{N}.
$
Define
\[
\mathcal K(E)
\coloneqq
\Bigl\{
K \in \mathbb{R}^{m\times n}
:
K=\Lambda E,\;
\Lambda=\operatorname{diag}(\lambda_1,\dots,\lambda_m),\;
\lambda_i\in\{-1,+1\},\ \forall i\in\{1,\dots,m\}
\Bigr\}.
\]
\(\mathcal K(E)\) is the set of sign-saturated feedback gains associated with the constraint \(|u|\le Ex\). In the model-based setting, the works~\cite{gurpegui2025minimax,gurpegui2023minimax}
show that it is sufficient, without loss of optimality, to restrict attention
to static state-feedback policies of the form
\[
u_t = Kx_t,
\qquad K\in\mathcal K(E).
\]
Under this construction, it is shown in~\cite{gurpegui2025minimax} that the problem reduces to a linear program whenever the game is feasible. 
\item[\textit{(ii)}] 
In the works \cite{ohlin2024heuristic,ohlin2024optimal}, the input is partitioned as \(u=(u_1,\dots,u_n)\), where \(u_i\in\mathbb{R}^{m_i}\) and \(\sum_{i=1}^n m_i=m\).
Let
\[
E^\top =
\begin{bmatrix}
E_1 & \cdots & E_n
\end{bmatrix}
\in \mathbb{R}_+^{n\times n}.
\]
Define
\[
\mathcal K(E)
\coloneqq \Bigl\{\, K \in \mathbb{R}^{m\times n} :
\forall i\in\{1,\dots,n\},\
\mathbf{1}^\top K_i = E_i^\top \ \text{or}\ \mathbf{1}^\top K_i = 0 \Bigr\},
\]
where \(K_i\) denotes the \(i\)-th block row of \(K\), corresponding to \(u_i\in\mathbb{R}^{m_i}\).
The admissible input set is of the form 
\[
\mathcal U(x)
\coloneqq \Bigl\{\, u \in \mathbb{R}^{m}_+ :
\exists\,K\in\mathcal K(E)\ \text{such that}\ u=Kx \Bigr\}.
\]
In this setup, it is also assumed that \(A+BK\ge 0\) for all \(K\in\mathcal K(E)\).
Finally, to maintain nonnegativity of the stage cost \(s^\top x + r^\top u\), it suffices to assume \(s,r\ge 0\), since inputs are nonnegative by design. This class of positive systems also leads to a linear program and has been shown to be equivalent to stochastic shortest path problems, making it an appealing candidate application for future work.
\end{itemize}
In Appendix~\ref{model-based}, we derive the model-based solution of Problem~\ref{prob:l1-robust-control} in the nominal case where
$
    \mathcal M=\{(A_\star,B_\star)\},
$
and for the choice of \(\mathcal U(x)\) in \textit{(i)}. This is useful for three reasons:
\begin{enumerate}
\item[\textit{(1)}] From Section~\ref{bellman-ineq} onward, we focus on input sets of the form
\textit{(i)}, although analogous results can be derived for \textit{(ii)} in
the same spirit.
    
\item[\textit{(2)}] Unlike \cite{gurpegui2025minimax}, our model-based formulation allows a broader class of disturbances: rather than requiring disturbances to be nonnegative, we only require them to be positivity preserving. We therefore derive the solution to this more relaxed model-based problem in Appendix~\ref{model-based}.
    \item[\textit{(3)}] Understanding the nominal model-based solution plays an important role in deriving the adaptive, model-free solution.
\end{enumerate}

\begin{remark}
The extension of the admissible disturbance set from nonnegative disturbances, \(w\ge 0\), to positivity preserving disturbances, \(w\in \mathcal{W}(x,u)\), plays a crucial role in enabling the success of our approach. This will become clear in the next section.
\end{remark}

%% file: tex_input/reformulation.tex
\section{Reformulated problem via the nonnegative historical data}\label{reform}
In this section, we take a first step toward solving problem~\ref{prob:l1-robust-control} by reformulating it as a standard minimax zero-sum dynamic game through the introduction of so called historical data variables. This reformulation, inspired by~\cite{rantzer2021minimax}, enables the use of standard minimax dynamic-programming theory to address the problem.

The results in this section and in Section~\ref{sec:dp-positive} hold for the
general class of input sets \(\mathcal{U}(x)\) introduced in
Section~\ref{pb_class}, and are stated in that generality. Recall the original minimax problem with unknown dynamics, formulated as Problem~\ref{prob:l1-robust-control}.
\paragraph{Reparameterization via the next state.}
Introduce the next state sequence \(v\), with
$
v_t \coloneqq x_{t+1}.
$
Then,
$
w_t = v_t - (A x_t + B u_t),
$
and the constraint \(w \ge -(Ax+Bu)\) is equivalent to \(v \ge 0\), independently of the adversary’s choice of \((A,B)\). Substituting \(w_t = v_t - (A x_t + B u_t)\) yields the equivalent problem
\begin{equation}\label{eq:robust-rewritten}
\begin{aligned}
\inf_{\mu}\;
\sup_{\substack{(A,B)\in\mathcal M\\[1pt]  v \,\ge\, 0\\[1pt] T\in\mathbb N}}
&\ \sum_{t=0}^{T} \Bigl( s^\top x_t + r^\top u_t - \gamma^\top \lvert v_t - A x_t - B u_t\rvert \Bigr) \\
\text{s.t.}\quad
& x_{t+1}=v_t,\qquad x_0 \in \mathbb{R}_+^n, \\
& u_t=\mu_t\!\big(x_0,\ldots,x_t;\,u_0,\ldots,u_{t-1}\big),\qquad
u_t\in\mathcal U(x_t).
\end{aligned}
\end{equation}
Since the feasibility condition \(v \ge 0\) does not depend on \((A,B)\), maximizing jointly over \((A,B)\in\mathcal M\) and \(v \ge 0\) can be carried out in any order.
Hence, problem~\eqref{eq:robust-rewritten} can be written equivalently as
\begin{equation}\label{eq:robust-rewritten-commuted}
\begin{aligned}
\inf_{\mu}\;
\sup_{ v \,\ge\, 0}
\sup_{\substack{(A,B)\in\mathcal M\\ T\in\mathbb N}}
&\ \sum_{t=0}^{T} \Bigl( s^\top x_t + r^\top u_t - \gamma^\top \lvert v_t - A x_t - B u_t\rvert \Bigr) \\
\text{s.t.}\quad
& x_{t+1}=v_t,\qquad x_0 \in \mathbb{R}_+^n, \\
& u_t=\mu_t\!\bigl(x_0,\ldots,x_t;\,u_0,\ldots,u_{t-1}\bigr),\qquad
u_t\in\mathcal U(x_t).
\end{aligned}
\end{equation}
If one instead assumes nonnegative disturbances, \(w\geq 0\), then the next
state constraint becomes
$
v\geq Ax+Bu,
$
which depends on the unknown true dynamics \((A_\star,B_\star)\). Imposing
this constraint for all pairs \((A,B)\in\mathcal M\) would introduce
conservatism in the adversary's choice, since positivity only needs to hold
for the true pair selected by the adversary. This motivates our preference for
the alternative formulation
$
w\geq -(Ax+Bu),
$
or, equivalently, \(v\geq 0\), which should hold irrespective of which pair
\((A,B)\) is selected by the nonnegative adversary $v$.
\paragraph{Model-specific historical data.}
For each model \((A,B)\in\mathcal M\), define the scalar nonnegative cumulative history by
\begin{align}\label{hist2}
z_{t+1}^{(A,B)}
\coloneqq
\sum_{\tau=0}^{t}
\gamma^\top \bigl| v_{\tau} - A x_{\tau} - B u_{\tau} \bigr|.
\end{align}
Equivalently, this quantity evolves according to the recursion
\begin{align}\label{hist1}
z_{t+1}^{(A,B)}
=
z_t^{(A,B)}
+
\gamma^\top \bigl| v_t - A x_t - B u_t \bigr|,
\qquad
z_0^{(A,B)} = 0.
\end{align}
We then define the historical data collection, indexed by $\mathcal M$, as
\begin{align}\label{hist3}
Z_t \coloneqq \{z_t^{(A,B)}\}_{(A,B)\in\mathcal M},
\end{align}
which collects the history variables associated with all dynamics pairs
$(A,B)\in\mathcal M$.

\paragraph{Minimax game with historical data.}
Using the definitions in~\eqref{hist2}--\eqref{hist3}, the reformulated problem~\eqref{eq:robust-rewritten-commuted} can be equivalently written as follows.

\begin{problem}\label{prob:robust-rewritten-min-history}
Find a causal policy \(\eta\) solving
\[
\begin{aligned}
\inf_{\eta}\;
\sup_{\substack{v \ge 0\\ T \in \mathbb{N}}}
\quad
&
\sum_{t=0}^{T} \bigl(s^\top x_t + r^\top u_t\bigr)
-
\min_{(A,B)\in\mathcal M} z_{T+1}^{(A,B)}
\\[1mm]
\text{s.t.}\quad
&
x_{t+1} = v_t,
\qquad
x_0 \in \mathbb{R}_+^n,
\\
&
z_{t+1}^{(A,B)}
=
z_t^{(A,B)}
+
\gamma^\top \bigl|v_t - A x_t - B u_t\bigr|,
\qquad
z_0^{(A,B)} = 0,
\quad
\forall\,(A,B)\in\mathcal M,
\\
&
Z_t \coloneqq \{z_t^{(A,B)}\}_{(A,B)\in\mathcal M},
\qquad
Z_0 \coloneqq \{0\}_{(A,B)\in\mathcal M},
\\
&
u_t = \eta(x_t,Z_t) \in \mathcal{U}(x_t),
\qquad
\forall\,t\in\mathbb N.
\end{aligned}
\]
\end{problem}
\begin{remark}
Note that constructing the history vector \(Z_t\) does not require knowledge of the true dynamics \((A,B)\), since its construction exhausts all possible pairs in the model set \(\mathcal M\).
\end{remark}
\begin{remark}
In the reformulated problem~\ref{prob:robust-rewritten-min-history}, the unknown true parameters \((A,B)\) appear only in terminal cost and not in the extended dynamics \((x_t,Z_t)\). Consequently, the control law \(\eta(\cdot)\) need not depend on the full state history; it depends only on the current extended state \((x_t,Z_t)\). The problem is cast as a standard zero-sum minimax dynamic game in which the adversary chooses the positive next state \(v_t\ge 0\) to maximize the cost and
the controller chooses the input \(u_t\in\mathcal U(x_t)\) to minimize it.
\end{remark}
\begin{remark}
The reformulation technique used here follows a principle similar to that in~\cite{rantzer2021minimax}, with one important difference. In~\cite{rantzer2021minimax}, the history variable is defined in terms of a global covariance historical-data matrix that is independent of the model set \(\mathcal M\). By contrast, in our formulation, the history variables are defined locally for each possible pair \((A,B)\in\mathcal M\). This distinction is specific to our \(\ell_1\) setup.
\end{remark}
We next address problem~\ref{prob:robust-rewritten-min-history} using minimax dynamic programming.

%% file: tex_input/main_result1.tex
\section{Minimax dynamic programming for positive systems}\label{sec:dp-positive}
\paragraph{Dynamic programming operators.}
For a value function $V(x,Z)$, define the operators $\mathcal F$ and $\mathcal F_u$ by
\begin{align}\label{Foper}
\mathcal F V(x,Z)
\coloneqq
\min_{u\in\mathcal U(x)}
\underbrace{
\max_{v\ge0}
\left\{
s^\top x+r^\top u+V(v,Z_+)
\right\}
}_{\eqqcolon\,\mathcal F_uV(x,Z)},
\end{align}
where $Z_+$ denotes the updated nonnegtaive historical data collection
\[
Z_+
\coloneqq
\{z_+^{(A,B)}\}_{(A,B)\in\mathcal M},
\]
whose entries are updated via
\[
z_+^{(A,B)}
=
z^{(A,B)}
+
\gamma^\top\bigl|v-Ax-Bu\bigr|,
\qquad
\forall\,(A,B)\in\mathcal M.
\]
\paragraph{Initialization and value iteration.}
We initialize value iteration with the terminal cost function
\begin{align}\label{V0}
V_0(x,Z)
=
-\min_{(A,B)\in\mathcal M} z^{(A,B)}.
\end{align}
The value iteration is then carried recursively by
\begin{align}\label{Vk}
V_{k+1}(x,Z)
=
\mathcal F V_k(x,Z)
=
\min_{u\in\mathcal U(x)}
\max_{v\ge 0}
\left\{
s^\top x+r^\top u+V_k(v,Z_+)
\right\}.
\end{align}
This defines the sequence of value functions $\{V_k\}_{k\in\mathbb N}$.

Given all the ingredients introduced above, we are now able to state one of the main results of the paper. This result establishes a crucial equivalence between the original Problem~\ref{prob:l1-robust-control} and its reformulation as a standard dynamic game, Problem~\ref{prob:robust-rewritten-min-history}.
\begin{theorem}\label{thm:l1-dp}
Given \(\gamma\in\mathbb{R}_+^n\) and a parameter set \(\mathcal M\), as in
Definition~\ref{Def1}, satisfying Assumption~\ref{assump1}, define
\(\mathcal F\) by~\eqref{Foper} and the sequence
\(\{V_k\}_{k\in\mathbb N}\) by~\eqref{V0}--\eqref{Vk}, with
\(z^{(A,B)}\) constructed for each \((A,B)\in\mathcal M\). Suppose that
Assumption~\ref{ass:stage cost-pos} holds. Then, the following statements hold.
\begin{enumerate}
    \item[\textit{(i)}] Problems~\ref{prob:l1-robust-control} and~\ref{prob:robust-rewritten-min-history} have the same value. This value is finite if and only if the sequence \(\{V_k(x,0)\}_{k=0}^\infty\) is bounded above for every \(x\in\mathbb R_+^n\). In that case, the limit \(V_\ast \coloneqq \lim_{k\to\infty} V_k\) exists and satisfies
    $
    J_\ast(x_0)=V_\ast(x_0,0).
    $

    \item[\textit{(ii)}] Let \(\eta^\ast(x,Z)\) be the minimizing \(u\) in \(\mathcal F V_\ast(x,Z)\). Then, \(\eta^\ast\) is optimal control law for problem~\ref{prob:robust-rewritten-min-history}, and the control law
\begin{align*}
\mu_t^\ast(x_0,\ldots,x_t,u_0,\ldots,u_{t-1})
\coloneqq
\eta^\ast\!\Biggl(
x_t,\,
\left\{
\sum_{\tau=0}^{t-1}
\gamma^\top
\bigl|x_{\tau+1}-Ax_\tau-Bu_\tau\bigr|
\right\}_{(A,B)\in\mathcal M}
\Biggr).
\end{align*}
    is optimal for the original minimax problem~\ref{prob:l1-robust-control}.

    \item[\textit{(iii)}] If there exists a function \(\overbar V(x,Z)\) such that \(\overbar V\ge V_0\) and
\[
\mathcal F_{\overbar\eta}\,\overbar V(x,Z)\le \overbar V(x,Z)
\qquad
\text{for all } (x,Z) \in \mathbb R_+^n \times \mathbb R^{\abs{\mathcal{M}}},
\]
where
$
\mathcal F_{\overbar\eta}V(x,Z)\coloneqq \mathcal F_uV(x,Z)\big|_{u=\overbar\eta(x,Z)},
$
then the policy \(\overbar\mu\) defined by 
\begin{align*}
u_t
=
\overline{\mu}_t(x_0,\ldots,x_t,u_0,\ldots,u_{t-1})
\coloneqq
\overline{\eta}\!\Biggl(
x_t,\,
\left\{
\sum_{\tau=0}^{t-1}
\gamma^\top
\bigl|x_{\tau+1}-Ax_\tau-Bu_\tau\bigr|
\right\}_{(A,B)\in\mathcal M}
\Biggr).
\end{align*}
 satisfies the cost upper bound
$
J_{\overbar\mu}(x_0)\le \overbar V(x_0,0).
$
\end{enumerate}
\end{theorem}
\begin{proof}
    See Appendix~\ref{AppendB}.
\end{proof}
\begin{remark}
Theorem~\ref{thm:l1-dp} has three main implications:
\begin{enumerate}
\item[\textit{(1)}] The original problem~\ref{prob:l1-robust-control} and the reformulated problem~\ref{prob:robust-rewritten-min-history} have the same optimal value, as item~\textit{(i)} above shows. 

\item[\textit{(2)}] Any policy that solves the reformulated problem optimally, or suboptimally, gives a policy for the original problem with the same optimality, or suboptimality, guarantee, as items~\textit{(ii)}-\textit{(iii)} above state. 

\item[\textit{(3)}] Any solution of the Bellman inequality associated with the reformulated problem provides an upper bound for the value of the original problem, achieved by the corresponding induced policy, as item~\textit{(iii)} states. 
\end{enumerate}
\end{remark}
Having established the equivalence between the original problem, Problem~\ref{prob:l1-robust-control}, and the reformulated problem, Problem~\ref{prob:robust-rewritten-min-history}, we now focus on the reformulated problem. To this end, we relax the Bellman equation and instead consider the corresponding Bellman inequality, as stated in Theorem~\ref{thm:l1-dp}\textit{(iii)}, for which a solution will be derived in the sequel.

%% file: tex_input/main_result2.tex
\section{Explicit solutions to the Bellman inequality} \label{bellman-ineq}
To illustrate how the results of Theorem~\ref{thm:l1-dp} can be applied, we now focus on the setting described in item \emph{(i)} of Subsection~\ref{sec31}. Similar results for the setting in item \emph{(ii)} can be obtained by following the same approach to the one developed here. Accordingly, we let
\[
    \mathcal U(x)\coloneqq\{\,u\in\mathbb{R}^m: |u|\le Ex\,\},
    \qquad \forall x\in\mathbb{R}^n_+,
\]
where the matrix $E$ is governed by the following assumption.
\begin{assumption}\label{ass:input-set}
We assume that there exists \(E\in\mathbb{R}_+^{m\times n}\) such that:
\begin{enumerate}
\item[\textit{(i)}] For all \((A,B)\in\mathcal M\), \(A\ge |B|E\).

\item[\textit{(ii)}] The stage-cost parameters satisfy \(s\ge E^\top |r|\). 
\end{enumerate}
\end{assumption}
The two conditions in Assumption~\ref{ass:input-set} have the following interpretations.
\begin{enumerate}
    \item[\textit{(i)}]
    This condition ensures positivity of the closed-loop system for all admissible dynamics \((A,B)\in\mathcal M\) including the true pair $(A_\star,B_\star)$. Indeed, for every \(t\in\mathbb N\), \(x_t\in\mathbb R^n_+\), and \(u_t\) satisfying \(|u_t|\le Ex_t\), we have
    \[
        Ax_t+Bu_t
        \ge Ax_t-|B||u_t|
        \ge (A-|B|E) x_t
        \ge 0
    \]
for all $(A,B) \in \mathcal{M}$.
    \item[\textit{(ii)}]
    This condition implies Assumption~\ref{ass:stage cost-pos}. Whenever \(Ax_t+Bu_t\ge 0\), and since \(\gamma\in\mathbb R^n_+\), we have, for every \(t\in\mathbb N\), \(x_t\in\mathbb R^n_+\), and \(u_t\) satisfying \(|u_t|\le Ex_t\),
    \begin{multline*}
       \max_{w_t \ge -(A x_t + B u_t)}
       \Bigl(s^\top x_t+r^\top u_t-\gamma^\top |w_t|\Bigr)
       =
       s^\top x_t+r^\top u_t  
       \ge s^\top x_t-|r|^\top |u_t|  
      \ge \bigl(s-E^\top |r|\bigr)^\top x_t  
       \ge 0 .
    \end{multline*}
    Hence, the worst-case stage cost is nonnegative.
\end{enumerate}
We begin by providing a solution to the Bellman inequality for the special case of input sign uncertainty, and then generalize it to the case of a finite set of LTI plants afterwards.
\subsection{Positive systems with unknown input direction}\label{sec7}
We consider positive linear systems with unknown input direction. For scalar unstable systems, it can be shown that it is not possible to synthesize a static state-feedback law that stabilizes both input directions simultaneously: stabilizing one direction necessarily destabilizes the other; see \cite{vinnicombe2004examples} for further details.
In this setting, the uncertainty enters only through the sign of the input matrix, while the state matrix is common across the models. Specifically, we consider the two model unknown sign class
\begin{align}\label{model-class1}
   \mathcal M \coloneqq \{(A,B),(A,-B)\} \subset \mathbb R_+^{n\times n}\times \mathbb R^{n\times m},
\end{align}
for which the historical data collection reduces to
$Z_t=\{z_t^{(+)},\,z_t^{(-)}\}$, where
$z_t^{(+)}\coloneqq z_t^{(A,B)}$ and
$z_t^{(-)}\coloneqq z_t^{(A,-B)}$ where the model-specific histories are updated as
\[
z_+^{(\sigma)}
=
z^{(\sigma)}
+\gamma^\top\bigl|\,v-Ax-\sigma Bu\,\bigr|,
\qquad
z_0^{(\sigma)}=0,\ \ \sigma\in\{+,-\}.
\]
Recall also the input constraint \(\abs{u}\le Ex\). Without a loss of optimality, Such inputs can be expressed as \(u=Kx\), with \(K\) belonging to
\[
\mathcal K(E)
\coloneqq
\Bigl\{\, K \in \mathbb{R}^{m\times n} :
K=\Lambda E,\ 
\Lambda=\operatorname{diag}(\lambda_1,\dots,\lambda_m),\
\lambda_i\in[-1,1]
\Bigr\}.
\]
The next result gives an explicit expression for an adaptive controller satisfying a prescribed \(\ell_1\)-gain bound $\gamma$ whenever the pair $(A,B)$ is stabilizable. 
\begin{theorem}\label{thm:l1-unknown-sign}
Let \(K_+, K_- \in \mathcal{K}(E)\), and consider the model class
\(\mathcal{M}\) in~\eqref{model-class1}. Let \((s,r,E)\) be as in
Assumption~\ref{ass:input-set}. Suppose there exist \(p_+\), \(p_-\), and \(h\)
satisfying \(0\leq\max\{p_+,p_-\}\leq h\leq \gamma\) such that the following
inequalities hold
\[
\begin{aligned}
\textit{(i)}\;\;
&p_\sigma
\ \ge\
s+K_\sigma^\top r+(A+\sigma BK_\sigma)^\top p_\sigma,
\qquad \text{for all}\ \sigma\in\{+,-\},
\\[4pt]
\textit{(ii)}\;\;
&h
\ \ge\
s+K_k^\top r+(A+\sigma BK_k)^\top p_\sigma,
\qquad \text{for all}\ k,\sigma\in\{+,-\}\ \text{with}\ k\neq \sigma,
\\[4pt]
\textit{(iii)}\;\;
&h^\top x
\ \ge\
\bigl(s^\top+r^\top K_k+h^\top A\bigr)x
+(h-\gamma)^\top |BK_kx|,
\qquad k\in\{+,-\},\ \text{for all}\ x\in\mathbb R_+^n.
\end{aligned}
\]
Then, the Bellman inequality
$
\mathcal F\,\overbar V(x,Z)\leq \overbar V(x,Z) 
$ holds for
\begin{align*}
    \overbar V(x,Z)
    =
    \max\Bigl\{
    p_+^\top x-z^{(+)},
    \,p_-^\top x-z^{(-)},
    \,h^\top x-\tfrac12\!\left(z^{(+)}+z^{(-)}\right)
    \Bigr\} 
\end{align*}
Consequently, the certainty equivalence policy
\begin{equation}\label{eq:cert-unknown-sign}
u_t= K_{k_t}x_t,
\qquad
k_t=\arg\min_{\sigma\in\{+,-\}} z_t^{(\sigma)},
\end{equation}
satisfies the cost bound
$
J_\mu(x_0)\leq h^\top x_0.
$
If, in addition, \(m=1\) (the single input case), then inequality \textit{(iii)}
above simplifies to the state independent inequality
\[
\textit{(iii')}\;\;
h^\top
\ge
s^\top + r^\top K_k + h^\top A + (h-\gamma)^\top |B|E,
\qquad k\in\{+,-\}.
\]
\end{theorem}
\begin{proof}
See Appendix~\ref{thm:l1-unknown-sign-append}.
\end{proof}
The next remark provides a recipe for the solving the vector inequalities in Theorem~\ref{thm:l1-unknown-sign}.
\begin{remark}[Recipe for solving \emph{(i)}--\emph{(iii)}]\label{here}
The inequalities in Theorem~\ref{thm:l1-unknown-sign} can be solved in the
following sequence:
\begin{enumerate}
\item[\emph{(i)}] First, for each \(\sigma\in\{+,-\}\), one may solve the corresponding equality
\[
p_\sigma = s + A^\top p_\sigma - E^\top\!\left|\,r+\sigma B^\top p_\sigma\,\right|,
\]
which gives the optimal cost achieved by the model-based controller for the
known sign \(\sigma\). The corresponding optimal model-based gain is
\[
K_\sigma^\star
=
-\operatorname{diag}\!\bigl(\operatorname{sgn}(r+\sigma B^\top p_\sigma)\bigr)\,E.
\]
\item[\emph{(ii)}] Next, with \(p_+\) and \(p_-\) fixed, choose \(h\) large enough
so that the inequalities hold. Since there are only two mixed cases,
namely \((\sigma,k)=(+,-)\) and \((\sigma,k)=(-,+)\), the smallest feasible
choice is the componentwise maximum
\[
h
=
\max\!\Bigl\{
s+r^\top K_-+p_+^\top(A+BK_-),\;
s+r^\top K_++p_-^\top(A-BK_+)
\Bigr\}.
\]
In particular, this yields \(h\ge p_+\) and \(h\ge p_-\), since the
right-hand sides are obtained by applying the controller designed for the wrong,
and potentially destabilizing, input direction.
\item[\emph{(iii)}] Finally, with \(p_+\), \(p_-\), and \(h\) fixed, choose
\(\gamma \ge h\) sufficiently large so that the last inequality holds.
For fixed \(p_+\), \(p_-\), and \(h\), these inequalities are linear in
\(\gamma\), so a feasible value of \(\gamma\) can be readily obtained.
\end{enumerate}
In the single input case \(m=1\) with \(K_k=\pm E\), condition \emph{(iii)}
reduces to \emph{(iii')}, so the last step becomes a componentwise linear
inequality in \(\gamma\), independent of \(x\).
\end{remark}
Next, we provide an example showcasing the results obtained in Theorem~\ref{thm:l1-unknown-sign}.
\begin{example}[Two-reservoir network with uncertain transfer direction]\label{ex:unknown-sign}
Consider the two-reservoir water network shown in Figure~\ref{fig:water-network-unknown-sign}. The states \(x_{1,t}\) and \(x_{2,t}\) represent the nonnegative stored water volumes in each reservoir. The reservoirs are connected through the surrounding network and through a controllable interconnection pipe equipped with a pump. The dynamics are linearized to
$
x_{t+1}=Ax_t\pm Bu_t+w_t,
$
with
\[
A=\begin{bmatrix}\frac12 & \frac{9}{10}\\[2pt]\frac25 & \frac{7}{10}\end{bmatrix},\qquad
B=\begin{bmatrix}\frac15\\[2pt]-\frac25\end{bmatrix},\qquad
E=\begin{bmatrix}1 & 1\end{bmatrix},\qquad
s=\begin{bmatrix}1\\[2pt]1\end{bmatrix},\qquad
r=\frac15.
\]
Here, the matrix \(A\) models passive storage and hydraulic coupling. Each reservoir retains part of its own water volume over one time step, captured by the diagonal elements of $A$, while the off diagonal entries account for network mediated redistribution between the reservoirs. The input vector \(B\) has one positive and one negative entry because a transfer command moves water from one reservoir to the other, meaning one stored volume increases while the other decreases. The uncertainty \(\pm B\) represents the fact that the effective direction of the interconnection pump is unknown, for instance because of reversed installation or an uncertain sign convention in the supervisory controller. The disturbance \(w_t\) represents exogenous inflows and withdrawals at the reservoirs, such as rainfall, upstream supply variations, consumer demand, or unmeasured leakage, while preserving nonnegativity of the next state $x_{t+1}$ consistent with the physical requirement that stored water volumes cannot be negative. We have $K_+=E,\ K_-=-E$
and 
Assumption~\ref{ass:input-set} holds since
\[
A=
\begin{bmatrix}
\frac12 & \frac{9}{10}\\[2pt]
\frac25 & \frac{7}{10}
\end{bmatrix}
\ge
\begin{bmatrix}
\frac15 & \frac15\\[2pt]
\frac25 & \frac25
\end{bmatrix}
=|B|E,
\qquad
s=
\begin{bmatrix}
1\\[2pt]
1
\end{bmatrix}
>
\begin{bmatrix}
\frac15\\[2pt]
\frac15
\end{bmatrix}
=E^\top|r|.
\]
Moreover,
$
\operatorname{eig}(A)=\frac{3\pm\sqrt{73}}{10},
$
so the open loop system is unstable. In particular, the true pump direction must be identified in order to stabilize the reservoir volumes. We now apply the recipe of Remark~\ref{here} to extract the adaptive policy. Solving \emph{(i)} gives the smallest positive
\[
p_+=
\begin{bmatrix}
4 & 8
\end{bmatrix}^{\top},
\qquad
p_-=
\begin{bmatrix}
\frac{8}{3} & \frac{16}{3}
\end{bmatrix}^{\top}.
\]
The corresponding model-based optimal gains are then
\[
K_+^\star=-\operatorname{sgn}(r+B^\top p_+)E=E,
\qquad
K_-^\star=-\operatorname{sgn}(r-B^\top p_-)E=-E,
\]
so they coincide with the gains chosen above. Next, the inequalities \emph{(ii)} become
\[
h \ge s+ K_-^\top r+p_+^\top(A+BK_-)
=
\begin{bmatrix}
\frac{42}{5} & \frac{62}{5}
\end{bmatrix}^{\top},
\qquad
h \ge s+ K_+^\top r+p_-^\top(A-BK_+)
=
\begin{bmatrix}
\frac{94}{15} & \frac{134}{15}
\end{bmatrix}^{\top}.
\]
Hence the smallest feasible choice is the componentwise maximum
\[
h=
\max\!\left\{
\begin{bmatrix}
\frac{42}{5} & \frac{62}{5}
\end{bmatrix}^{\top},
\begin{bmatrix}
\frac{94}{15} & \frac{134}{15}
\end{bmatrix}^{\top}
\right\}
=
\begin{bmatrix}
\frac{42}{5} & \frac{62}{5}
\end{bmatrix}^{\top}.
\]
Finally, using the single input condition \emph{(iii')}, one obtains
\[
\gamma_1+2\gamma_2\ge \tfrac{292}{5},
\qquad
\gamma\ge h.
\]
Taking \(\gamma_1\) at its smallest admissible value \(\gamma_1=h_1=\frac{42}{5}\), the smallest corresponding \(\gamma_2\) is
$
\gamma_2=25,
$
so one minimal choice is
$
\gamma=
\begin{bmatrix}
\frac{42}{5} & 25
\end{bmatrix}^{\top}.
$
Simulations with the control law~\eqref{eq:cert-unknown-sign} are shown in Figure~\ref{fig:unknown-sign-six-panels}. In this interpretation, the certainty-equivalence controller initially increases control activity in order to learn the true pump direction; once the history variables identify the correct sign, the controller collapses to the model-based law for the true network dynamics. For more details on the implementation, refer
to the code\footnote{Code available at: \href{https://github.com/Fethi-Bencherki/minimax-adaptive-control-of-positive-systems}{https://github.com/Fethi-Bencherki/minimax-adaptive-control-of-positive-systems}}.
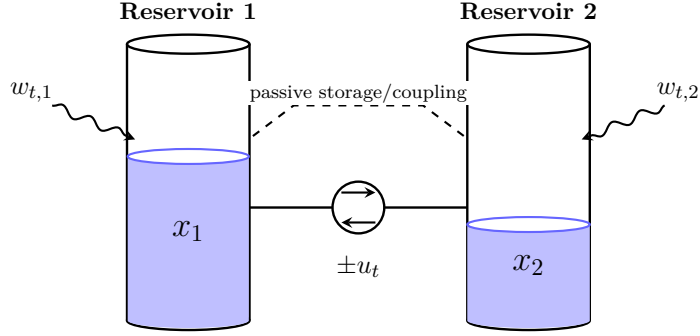
\begin{figure}[!tbh]
\centering
\begin{tikzpicture}[>=stealth,thick,scale=0.9,transform shape]

\def\tankW{0.9}      
\def\tankE{0.14}     
\def\tankH{4.2}


\draw[line width=0.9pt]
(0,\tankE) ellipse [x radius=\tankW,y radius=\tankE];

\draw[line width=0.9pt]
(-\tankW,\tankE)--(-\tankW,\tankH);

\draw[line width=0.9pt]
(\tankW,\tankE)--(\tankW,\tankH);

\draw[line width=0.9pt]
(0,\tankH) ellipse [x radius=\tankW,y radius=\tankE];

\fill[blue!25]
(-0.89,\tankE)
--
(-0.89,2.55)
arc[start angle=180,end angle=360,
    x radius=.89,
    y radius=.11]
--
(0.89,\tankE)
arc[start angle=360,end angle=180,
    x radius=.89,
    y radius=.11];

\draw[blue!60]
(0,2.55) ellipse [x radius=.89,y radius=.11];

\node[font=\small\bfseries] at (0,4.7) {Reservoir 1};
\node at (0,1.45) {\Large $x_1$};


\begin{scope}[xshift=5cm]


\draw[line width=0.9pt]
(0,\tankE) ellipse [x radius=\tankW,y radius=\tankE];

\draw[line width=0.9pt]
(-\tankW,\tankE) -- (-\tankW,\tankH);

\draw[line width=0.9pt]
(\tankW,\tankE) -- (\tankW,\tankH);

\draw[line width=0.9pt]
(0,\tankH) ellipse [x radius=\tankW,y radius=\tankE];

\fill[blue!25]
(-0.90,\tankE)
--
(-0.90,1.55)
arc[start angle=180,end angle=360,
    x radius=.90,
    y radius=.11]
--
(0.90,\tankE)
arc[start angle=360,end angle=180,
    x radius=.90,
    y radius=.11];

\draw[blue!60]
(0,1.55) ellipse [x radius=.90,y radius=.11];

\node[font=\small\bfseries] at (0,4.7) {Reservoir 2};
\node at (0,0.95) {\Large $x_2$};

\end{scope}


\draw[line width=0.9pt] (0.90,1.80) -- (2.12,1.80);

\draw[line width=0.9pt] (2.88,1.80) -- (4.10,1.80);

\draw[fill=white,line width=0.9pt] (2.50,1.80) circle (0.38);

\draw[->] (2.25,2.02) -- (2.75,2.02);
\draw[->] (2.75,1.58) -- (2.25,1.58);

\node[below=6pt] at (2.50,1.42) {$\pm u_t$};


\draw[dashed]
(0.95,2.85) --
(1.55,3.30) --
(3.45,3.30) --
(4.05,2.85);

\node[
fill=white,
font=\scriptsize,
inner sep=1pt
]
at (2.5,3.48)
{passive storage/coupling};


\draw[
decorate,
decoration={
snake,
amplitude=0.6mm,
segment length=4mm
},
->
]
(-2.0,3.3)--(-0.78,2.8);

\node at (-2.3,3.5) {$w_{t,1}$};

\draw[
decorate,
decoration={
snake,
amplitude=0.6mm,
segment length=4mm
},
->
]
(7,3.3)--(5.78,2.8);

\node at (7.2,3.5) {$w_{t,2}$};

\end{tikzpicture}
\caption{Two-reservoir water network with uncertain transfer direction. The states \(x_1\) and \(x_2\) denote stored water volumes, \(u_t\) is the pump command, and \(w_t\) represents exogenous inflows and withdrawals. The matrix \(A\) captures passive storage and hydraulic coupling, while the uncertainty \(\pm B\) means that a positive control input may transfer water in either direction.}
\label{fig:water-network-unknown-sign}
\end{figure}
\begin{figure}[!tbh]
    \centering
    \begin{minipage}[t]{0.45\linewidth}
        \centering
        \input{plots/water_tank/state_x1.tex}
    \end{minipage}
    \hfill
    \begin{minipage}[t]{0.45\linewidth}
        \centering
        \input{plots/water_tank/state_x2.tex}
    \end{minipage}
    \vspace{0.5em}
    \begin{minipage}[t]{0.45\linewidth}
        \centering
        \input{plots/water_tank/control_input.tex}
    \end{minipage}
    \hfill
    \begin{minipage}[t]{0.45\linewidth}
        \centering
        \input{plots/water_tank/disturbance.tex}
    \end{minipage}
    \vspace{0.5em}
    \begin{minipage}[t]{0.45\linewidth}
        \centering
        \input{plots/water_tank/history_variables.tex}
    \end{minipage}
    \hfill
    \begin{minipage}[t]{0.45\linewidth}
        \centering
        \input{plots/water_tank/sign_estimate.tex}
    \end{minipage}
    \label{fig:ce-unknown-sign}
\caption{Closed loop simulation for Example~\ref{ex:unknown-sign}, comparing the adaptive controller for the unknown input direction with the optimal controller that knows the true sign. The panels show, from top left to bottom right, \(x_1\), \(x_2\), the control input \(u_t\), the positivity-preserving disturbance \(w_t\), the history variables \(z^{(+)}\) and \(z^{(-)}\), and the true and estimated signs. The true sign is fixed but unknown to the adaptive controller, while disturbances may be negative as long as they preserve state positivity. The adaptive controller initially uses probing to identify the input direction; once the sign is learned, its behavior coincides with the known dynamics controller.}
    \label{fig:unknown-sign-six-panels}
\end{figure}
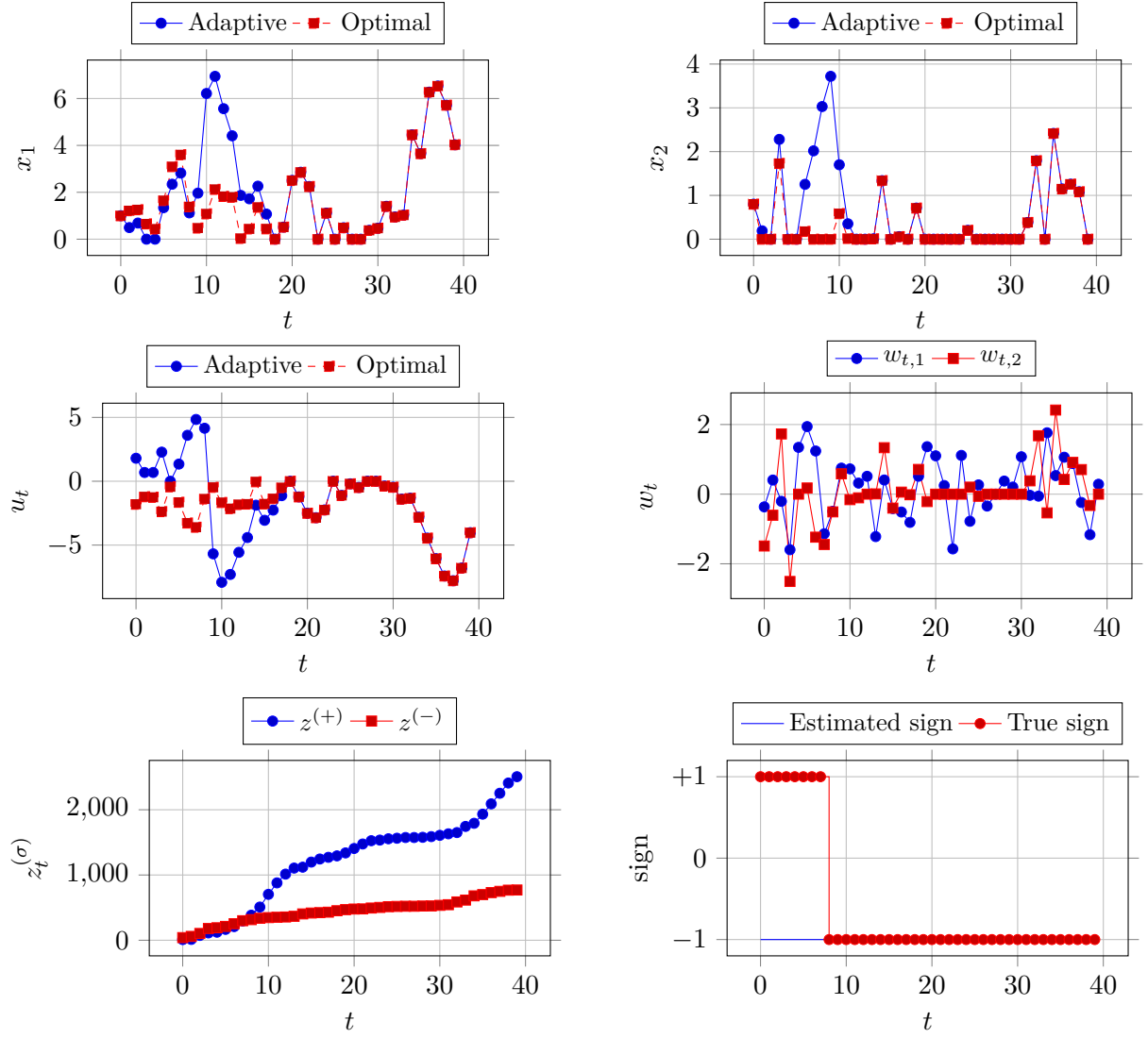
\end{example}
Next, we derive an upper bound on the minimum achievable gain \(\gamma\) for the certainty-equivalence (\textup{CE}) policy in~\eqref{eq:cert-unknown-sign} over the model class \(\mathcal{M}\) in the scalar case with \(r=0\), that is, with no input cost. We denote this gain by \(\gamma^{\mathrm{CE}}(a)\).
\begin{corollary}[Scalar CE gain and cost upper bounds]\label{thm:l1-scalar}
Consider the scalar case \(n=m=1\) with model class
$
\mathcal M=\{(a,b),(a,-b)\},
$
where \(a>0\) and \(b>0\). Let \(s=1\) and \(r=0\), and set
\(E=\frac{a}{b}\), \(K_+=-E\), and \(K_-=E\).
Hence, Theorem~\ref{thm:l1-unknown-sign} applies to the CE controller
\begin{align}\label{CE-scalar}
    u_t=-k_t \tfrac{a}{b}x_t,\qquad
k_t=\arg\min_{\sigma\in\{+,-\}} z_t^{(\sigma)}.
\end{align}
Then, the conditions in Theorem~\ref{thm:l1-unknown-sign} \textit{(i)}-\textit{(iii)} reduces to the simpler scalar inequalities
\[
\begin{aligned}
\text{\emph{\textit{(i)}}}\;\;
 p\ge 1,
\qquad
\text{\emph{\textit{(ii)}}}\;\;
 h\ge 1+2ap,
\qquad
\text{\emph{\textit{(iii)}}}\;\;
 h\ge 1+2ah-\gamma a .
\end{aligned}
\]
Moreover, the smallest value of \(\gamma\) for which there exist
\(p,h\in[0,\gamma]\) satisfying \textit{(i)}--\textit{(iii)} is
\begin{equation}\label{gain-bound}
\gamma^{\mathrm{CE}}(a)
=
\max\{1+2a,\,4a\}
=
\begin{cases}
1+2a, & 0<a\leq \tfrac12,\\[4pt]
4a,   & a\geq \tfrac12.
\end{cases}
\end{equation}
The corresponding minimizing values are \(p^\star=1\) and
\(h^\star=1+2a\), and the certainty equivalence policy proposed here certifies the cost
bound
$
J_\mu(x_0)\leq (1+2a)x_0.
$
\end{corollary}
\begin{proof}
See Appendix~\ref{thm:l1-scalar-appe}.
\end{proof}
\begin{remark}
In the absence of input cost, $r=0$, Corollary~\ref{thm:l1-scalar} shows that
$\gamma^{\mathrm{CE}}(a)$ is an upper bound on the best $\ell_1$-gain certified by
Theorem~\ref{thm:l1-unknown-sign} from the disturbance $w$ to the state $x$
under the scalar CE policy in~\eqref{CE-scalar}. Equivalently, the certified performance bound
can be written explicitly as
$$
\sum_{t=0}^{\infty} x_t
\leq
(1+2a)x_0
+
\max\{1+2a,4a\}
\sum_{t=0}^{\infty}\lvert w_t\rvert.
$$
The existence conditions for the stabilizing certainty-equivalence policy proposed in
Theorem~\ref{thm:l1-unknown-sign} are only sufficient and therefore may not capture
the best $\ell_1$-gain achievable by the CE policy. One may conjecture that the CE policy
achieves the optimal $\ell_1$-gain bound, analogous to the conjecture posed for
the quadratic setting in~\cite{vinnicombe2004examples}. The derivation of such optimal $\ell_1$-gain bound is deferred to future work.  
\end{remark}
We next generalize Theorem~\ref{thm:l1-unknown-sign} to a finite set of positive LTI models \((A,B)\in\mathcal M\).
\subsection{Explicit solution of the Bellman inequality in the general case}\label{sec8}
We now turn our focus to finding an explicit expression for an adaptive controller satisfying a prescribed bound on the \(\ell_{1}\)-gain for finite model sets of the form
\begin{align}\label{modelsetM}
    \mathcal M = \{(A_1,B_1),\ldots,(A_M,B_M)\}
    \subset \mathbb{R}_+^{n\times n}\times\mathbb{R}^{n\times m},
\end{align}
for which the historical data collection reduces to
$Z_t=\{z_t^{(1)},\ldots,z_t^{(M)}\}$, where
$z_t^{(i)}\coloneqq z_t^{(A_i,B_i)}$ for \(i=1,\ldots,M\).
In this case, problem~\ref{prob:robust-rewritten-min-history} becomes
\begin{equation}
\label{eq:symmetric-minimax-min-history}
\begin{aligned}
J_*(x_0)
=
&\;
\inf_{\eta}\;
\sup_{\substack{v\ge0\\ T\in\mathbb N}}
\left(
\sum_{t=0}^{T}\bigl(s^\top x_t+r^\top u_t\bigr)
-
\min_{i\in\{1,\ldots,M\}} z^{(i)}_{T+1}
\right)
\\[1mm]
\text{s.t.}\quad
&
x_{t+1}=v_t,
\qquad
x_0\in\mathbb R_+^n,
\\
&
z^{(i)}_{t+1}
=
z^{(i)}_{t}
+
\gamma^\top\bigl|v_t-A_i x_t-B_i u_t\bigr|,
\qquad
z^{(i)}_{0}=0,
\quad
i=1,\ldots,M,
\\
&
Z_t
\coloneqq
\{z_t^{(1)},\ldots,z_t^{(M)}\},
\qquad
Z_0
=
\{0,\ldots,0\},
\\
&
u_t
=
\eta_t(x_t,Z_t),
\qquad
|u_t|
\le
Ex_t.
\end{aligned}
\end{equation}
A solution to the Bellman inequality corresponding to
Problem~\ref{eq:symmetric-minimax-min-history} is characterized in the next
theorem.
\begin{theorem}\label{thm1}
Let \(K_1,\ldots,K_M\in\mathcal{K}(E)\subset\mathbb{R}^{m\times n}\), and consider the model set \(\mathcal{M}\) in~\eqref{modelsetM}. Let \((s,r,E)\) be as in Assumption~\ref{ass:input-set}.
Suppose there exist vectors \(p_{ij}\in\mathbb{R}^n_+\), for
\(i,j\in\{1,\ldots,M\}\), such that
$
0 \le p_{ij}=p_{ji}\le \gamma ,
$
and such that, for all \(x\in\mathbb{R}^n_+\) and all
\(i,j,k\in\{1,\ldots,M\}\), excluding the case \(i\neq j=k\), the inequality
\begin{multline}\label{ineq-cond-jk}
p_{jk}^\top x
\ge
s^\top x
+r^\top K_k x
+p_{ij}^\top \max\!\bigl\{A_i x + B_i K_k x, A_j x + B_j K_k x\bigr\} 
-\frac{1}{2}\gamma^\top
\bigl|(A_i-A_j)x+(B_i-B_j)K_k x\bigr|
\end{multline}
holds.
Define the linear value functions
\[
    V^{ij}(x,Z)
    \coloneqq
    p_{ij}^\top x-\frac{1}{2}\bigl(z^{(i)}+z^{(j)}\bigr),
    \qquad
    \overbar V(x,Z)
    \coloneqq
    \max_{i,j} V^{ij}(x,Z).
\]
Then, $\overbar V$ solves the Bellman inequality, i.e.
$
    \mathcal{F}\overbar V(x,Z)\le \overbar V(x,Z)
$ holds. Consequently, the certainty equivalence control law
\[
u_t = K_{k_t}x_t,
\qquad
k_t = \arg\min_{i\in\{1,\ldots,M\}} z_t^{(i)}
\]
satisfies the performance bound
$
J_{\mu}(x_0)\le \max_{i,j} p_{ij}^\top x_0 .
$
\end{theorem}
\begin{proof}
    See Appendix~\ref{appenBsth}.
\end{proof}
\begin{remark}
The model index $k_t\in\{1,\ldots,M\}$ selected by the certainty equivalence policy is given by
\[
k_t
=\arg\min_{i \in\{1,\ldots,M\}} \sum_{\tau=0}^{t-1}
\gamma^\top\!\bigl|x_{\tau+1}-A_i x_\tau-B_i u_\tau\bigr|
=\arg\min_{i \in\{1,\ldots,M\}} \sum_{\tau=0}^{t-1}
\left\|\operatorname{diag}(\gamma)\bigl(x_{\tau+1}-A_i x_\tau-B_i u_\tau\bigr)\right\|_1.
\]
The selected model is the one that minimizes the cumulative \(\gamma\)-weighted \(\ell_1\)-norm of the disturbance prediction error up to time \(t\).
\end{remark}
\begin{remark}
The inequalities \eqref{ineq-cond-jk} admit an interesting interpretation.

\begin{enumerate}
\item[\emph{(i)}]
If \(i=j=k\), then \eqref{ineq-cond-jk} reduces to
\[
p_{ii}^\top x
\ge
s^\top x+r^\top K_i x+p_{ii}^\top (A_i+B_iK_i)x,
\]
which is the standard linear algebraic inequality for the plant \((A_i,B_i)\) with parameter \(p_{ii}\). The quantity \(p_{ii}^\top x\) provides an upper bound on the closed loop infinite horizon cost when model \(i\) is known.
\item[\emph{(ii)}]
If \(i=j\neq k\), then \eqref{ineq-cond-jk} becomes
\[
p_{ik}^\top x
\ge
s^\top x+r^\top K_k x+p_{ii}^\top (A_i+B_iK_k)x,
\]
showing that the additional cost incurred by applying a controller not matched to the true model is upper bounded by the gap
$
p_{ik}^\top x-p_{ii}^\top x .
$
\item[\emph{(iii)}]
If \(i\neq j\), then the negative term
$
-\frac12 \gamma^\top
\bigl|(A_i-A_j)x+(B_i-B_j)K_kx\bigr|
$
captures the benefit of learning: even when a mismatched controller is used, improved model knowledge reduces the cost to go.
\end{enumerate}
\end{remark}
\begin{example}[Recipe for solving~\eqref{ineq-cond-jk}]
One possible way of solving the inequalities in~\eqref{ineq-cond-jk} for the model set
\[
\mathcal M=\{(A_1,B_1),\ldots,(A_M,B_M)\}
\]
can be carried out as follows.
\begin{enumerate}
\item[\emph{(i)}] First, consider the cases in which all indices are the same, i.e.,
\(i=j=k\). For each model \((A_i,B_i)\), \(i\in\{1,\ldots,M\}\), we solve the corresponding
model-based fixed-point equation
\begin{equation}\label{eq:diag-fp}
p_{ii}
=
s+A_i^\top p_{ii}
-
E^\top\bigl|r+B_i^\top p_{ii}\bigr|,
\qquad
\gamma\geq p_{ii}.
\end{equation}
This yields the Cost vectors \(p_{11},\ldots,p_{MM}\). The associated model-based optimal
gains are
\begin{equation}\label{eq:gain-opt}
K_i^\star
=
-\operatorname{diag}\!\bigl(\operatorname{sgn}(r+B_i^\top p_{ii})\bigr)E,
\qquad i=1,\ldots,M.
\end{equation}
This step fixes the diagonal Cost vectors \(p_{11},\ldots,p_{MM}\) and the corresponding
model-based gains \(K_1^\star,\ldots,K_M^\star\).

\item[\emph{(ii)}] Next, consider the cases in which \(i=j\neq k\). For each unordered pair
\(\{j,k\}\) with \(j,k\in\{1,\ldots,M\}\) and \(j\neq k\), define the symmetric mixed cost vectors
$
p_{jk}\coloneqq p_{kj}.
$
We then choose \(p_{jk}\) large enough to dominate the two right hand sides corresponding to
the cases \((i,j,k)=(j,j,k)\) and \((i,j,k)=(k,k,j)\). A feasible componentwise choice is
\begin{equation}\label{eq:mixed-pjk}
p_{jk}
=
\max\!\Bigl\{
s+K_k^\top r+\bigl(A_j+B_jK_k\bigr)^\top p_{jj},
\;
s+K_j^\top r+\bigl(A_k+B_kK_j\bigr)^\top p_{kk}
\Bigr\},
\qquad j\neq k,
\end{equation}
where the maximum is taken componentwise. These inequalities correspond to applying the
controller designed for one model to the other model. In other words, they capture the cost
increase that may occur when the wrong controller is applied. This step fixes the symmetric
mixed cost vectors \(p_{jk}=p_{kj}\) for all \(j\neq k\), corresponding to the cases
\(i=j\neq k\).

\item[\emph{(iii)}] Finally, consider the cases with \(i=k\neq j\). With all \(p_{ij}\) fixed,
choose \(\gamma\in\mathbb R_+^n\) sufficiently large so that
\(0\leq p_{ij}\leq \gamma,\ i,j\in\{1,\ldots,M\}\)
and the remaining inequalities in~\eqref{ineq-cond-jk} hold. At this stage, the remaining
admissible mixed index cases are precisely those with \(i=k\neq j\), while the cases
\(i\neq j=k\) are excluded by assumption. For each such triple, the corresponding inequality
is linear in \(\gamma\). Since there are only finitely many index combinations, a feasible
value of \(\gamma\) can be obtained by choosing it sufficiently large so that all remaining
inequalities hold.
\end{enumerate}
We summarize these steps in Algorithm~\ref{alg:recipe-ineq-cond-jk}.
\end{example}

\begin{algorithm}[t]
\caption{Recipe for solving~\eqref{ineq-cond-jk}.}
\label{alg:recipe-ineq-cond-jk}
\begin{algorithmic}[1]
\State \textbf{Input:} Model set \(\mathcal M=\{(A_1,B_1),\ldots,(A_M,B_M)\}\), input matrix \(E\), cost vectors \(s,r\).
\State \textbf{Output:} Cost vectors \(p_{ij}\), gains \(K_i^\star\), and gain bound \(\gamma\).

\State \textbf{Step 1:} Solve the diagonal cases \(i=j=k\) using~\eqref{eq:diag-fp}, and compute \(K_i^\star\) using~\eqref{eq:gain-opt}.
\State \textbf{Step 2:} For each \(j\neq k\), set \(p_{jk}=p_{kj}\) according to~\eqref{eq:mixed-pjk}.
\State \textbf{Step 3:} Choose \(\gamma\) sufficiently large so that \(0\le p_{ij}\le\gamma\) and the remaining cases \(i=k\neq j\) in~\eqref{ineq-cond-jk} hold.
\State \textbf{Return:} \(\{p_{ij}\}\), \(\{K_i^\star\}\), and \(\gamma\).
\end{algorithmic}
\end{algorithm}
\begin{remark}
The recipe provided in Algorithm~\ref{alg:recipe-ineq-cond-jk} may not yield the tightest cost
vectors \(p_{ij}\) and \(\ell_1\)-gain bound $\gamma$. It can be refined to produce
better solutions by using techniques similar to those in
\cite{cederberg2022synthesis}, adapted to our positive system, linear cost
setting.
\end{remark}
We illustrate the theoretical results obtained in this section with the following numerical example.
\begin{example}[Four-mode multiclass job queueing network]\label{ex:mfqnet-four-mode}
We construct a four mode example of multiclass job fluid queueing network (MFQNET) inspired by the model in \cite{bertsimas2023optimal}. 
We consider a queueing network with three job classes and three associated queues. Jobs of classes 1 and 2 wait in two upstream queues and are served by servers \(S_1\) and \(S_2\), respectively, while class 3 corresponds to a common downstream queue served by \(S_3\). Accordingly, the state is \(x=[\,x_{1}\,\,x_{2}\,\,x_{3}\,]^\top \in \mathbb{R}_+^3\), where \(x_{1}\), \(x_{2}\), and \(x_{3}\) denote the workloads in queues 1, 2, and 3, respectively. Over one sampling period, a portion of the workload in queue 1 may remain there awaiting service, while another portion may be served and subsequently routed to queue 2 or directly to queue 3; similarly, a portion of the workload in queue 2 may remain there awaiting service, while another portion may be served and subsequently routed to queue 1 or directly to queue 3. See Figure~\ref{fig:mfqnet-four-mode} for a schematic representation.
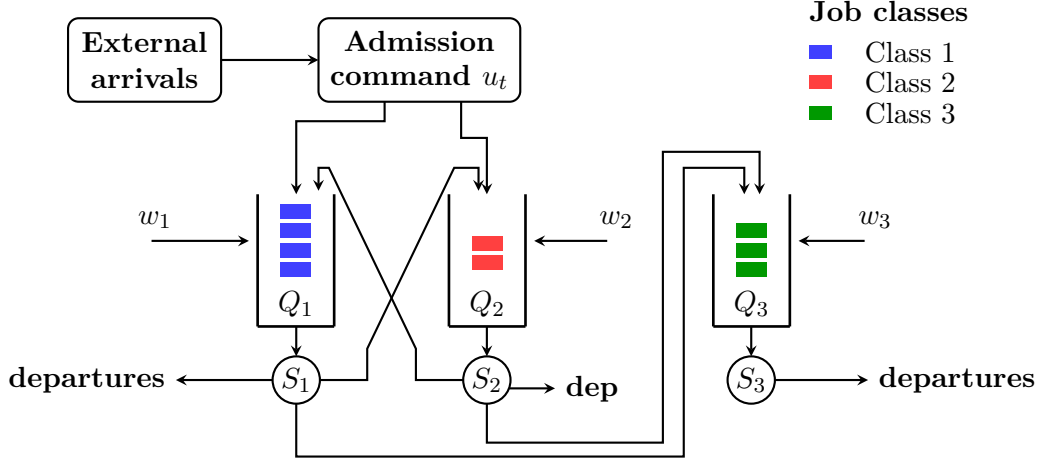
\begin{figure}
\centering
\begin{tikzpicture}[>=stealth,thick,scale=0.92]

    \draw[rounded corners=4pt] (0.4,5.0) rectangle (2.6,6.2);
    \node[align=center,font=\bfseries] at (1.5,5.6) {External\\arrivals};

\draw[rounded corners=4pt] (4.0,5.00) rectangle (6.9,6.20);
\node[align=center,font=\bfseries] at (5.45,5.60) {Admission\\[-1pt]command \(u_t\)};

\draw[->] (2.6,5.6) -- (4.0,5.6);

    \node[font=\bfseries] at (12.20,6.30) {Job classes};

    \fill[blue!75] (11.05,5.62) rectangle (11.37,5.82);
    \node[anchor=west] at (11.70,5.72) {Class 1};

    \fill[red!75] (11.05,5.18) rectangle (11.37,5.38);
    \node[anchor=west] at (11.70,5.28) {Class 2};

    \fill[green!60!black] (11.05,4.74) rectangle (11.37,4.94);
    \node[anchor=west] at (11.70,4.84) {Class 3};


\draw[line width=1pt] (3.125,1.77) -- (3.125,3.67);
\draw[line width=1pt] (3.125,1.77) -- (4.225,1.77);
\draw[line width=1pt] (4.225,1.77) -- (4.225,3.67);
\node[font=\bfseries] at (3.675,2.09) {\(Q_1\)};

\draw[line width=1pt] (5.875,1.77) -- (5.875,3.67);
\draw[line width=1pt] (5.875,1.77) -- (6.975,1.77);
\draw[line width=1pt] (6.975,1.77) -- (6.975,3.67);
\node[font=\bfseries] at (6.425,2.09) {\(Q_2\)};

\draw[line width=1pt] (9.675,1.77) -- (9.675,3.67);
\draw[line width=1pt] (9.675,1.77) -- (10.775,1.77);
\draw[line width=1pt] (10.775,1.77) -- (10.775,3.67);
\node[font=\bfseries] at (10.225,2.09) {\(Q_3\)};

\fill[blue!75] (3.455,2.48) rectangle (3.895,2.69);
\fill[blue!75] (3.455,2.76) rectangle (3.895,2.97);
\fill[blue!75] (3.455,3.04) rectangle (3.895,3.25);
\fill[blue!75] (3.455,3.32) rectangle (3.895,3.53);

\fill[red!75] (6.205,2.58) rectangle (6.645,2.79);
\fill[red!75] (6.205,2.86) rectangle (6.645,3.07);

\fill[green!60!black] (10.005,2.48) rectangle (10.445,2.69);
\fill[green!60!black] (10.005,2.76) rectangle (10.445,2.97);
\fill[green!60!black] (10.005,3.04) rectangle (10.445,3.25);

\draw[->] (4.95,5.00) -- (4.95,4.72) -- (3.68,4.72) -- (3.68,3.67);
\draw[->] (6.05,5.00) -- (6.05,4.48) -- (6.42,4.48) -- (6.42,3.67);

    \draw[->] (1.6,3.00) -- (3.0,3.00);
    \node at (1.65,3.30) {\(w_{1}\)};

    \draw[->] (8.15,3.00) -- (7.10,3.00);
    \node at (8.28,3.30) {\(w_{2}\)};

    \draw[->] (11.85,3.00) -- (10.90,3.00);
    \node at (12.00,3.30) {\(w_{3}\)};

    \draw (3.68,1.00) circle (0.34);
    \draw (6.42,1.00) circle (0.34);
    \draw (10.22,1.00) circle (0.34);

    \node at (3.68,1.00) {\(S_1\)};
    \node at (6.42,1.00) {\(S_2\)};
    \node at (10.22,1.00) {\(S_3\)};

    \draw[->] (3.68,1.77) -- (3.68,1.34);
    \draw[->] (6.42,1.77) -- (6.42,1.34);
    \draw[->] (10.22,1.77) -- (10.22,1.34);


\draw[->]
    (4.02,1.00) -- (4.70,1.00) -- (4.70,1.45) -- (5.95,4.05) -- (6.30,4.05) -- (6.30,3.72);

\draw[->]
    (6.08,1.00) -- (5.40,1.00) -- (5.40,1.45) -- (4.15,4.05) -- (4.00,4.05) -- (4.00,3.72);

\draw[->]
    (3.68,0.66) -- (3.68,-0.10) -- (9.25,-0.10) -- (9.25,4.05) -- (10.12,4.05) -- (10.12,3.67);

\draw[->]
    (6.42,0.66) -- (6.42,0.10) -- (8.95,0.10) -- (8.95,4.28) -- (10.34,4.28) -- (10.34,3.67);

    \draw[->] (3.34,1.00) -- (1.95,1.00);
    \node[left,font=\bfseries] at (1.95,1.00) {departures};

\draw[->] (6.74,0.88) -- (7.4,0.88);
\node[right,font=\bfseries] at (7.4,0.88) { dep};

    \draw[->] (10.56,1.00) -- (11.90,1.00);
\node[right,font=\bfseries] at (11.90,1.00) {departures};

\end{tikzpicture}
\caption{Three queue multiclass job queueing network. The scalar admission command \(u_t\) regulates the external workload entering queues \(Q_1\) and \(Q_2\). Queues \(Q_1\), \(Q_2\), and \(Q_3\) are processed by servers \(S_1\), \(S_2\), and \(S_3\), respectively. A portion of the workload served from \(Q_1\) may be routed to \(Q_2\) or directly to the shared downstream queue \(Q_3\), while another portion may leave the modeled network. Similarly, a portion of the workload served from \(Q_2\) may be routed to \(Q_1\) or \(Q_3\), while another portion may leave the modeled network. The disturbances \(w_{1,t}\), \(w_{2,t}\), and \(w_{3,t}\) represent signed net workload mismatch in the three queues. Depending on the unknown mode, the same signed command \(u_t\) may either increase or decrease the amount of work admitted into the network.}
\label{fig:mfqnet-four-mode}
\end{figure}

We model the control action by a scalar input \(u\in\mathbb R\), interpreted as
a signed deviation from a nominal admission level for the class-1 and class-2
queues.
\begin{itemize}
\item[\emph{(+)}] \(u>0\): the admission command increases the incoming workload
admitted into the class-1 and class-2 queues relative to the nominal level;
\item[\emph{(--)}] \(u<0\): the admission command decreases the incoming workload
admitted into the class-1 and class-2 queues relative to the nominal level, so
that a larger portion of the external arrivals remains outside the modeled
network.
\end{itemize}
We take
\[
\mathcal U(x_t)=\{\,u_t\in\mathbb R:\ |u_t|\le Ex_t\,\},
\qquad
E=\begin{bmatrix}1&1&0\end{bmatrix},
\]
so that the magnitude of the control action cannot exceed the total workload
currently present in queues \(Q_1\) and \(Q_2\).
We consider four operating modes:
\begin{itemize}
\item[\emph{(i)}] \textbf{Mode 1:} The dynamics are given by \((A_1,B_1)\), and
a positive input increases the workload admitted into queues \(Q_1\) and \(Q_2\).

\item[\emph{(ii)}] \textbf{Mode 2:} The dynamics are given by \((A_2,B_2)\). As
in Mode~1, a positive input increases the workload admitted into queues \(Q_1\)
and \(Q_2\), but the baseline propagation through the network differs because
\(A_2 \neq A_1\).

\item[\emph{(iii)}] \textbf{Mode 3:} The dynamics are given by \((A_3,B_3)\),
and the sign of the input effect on the upstream queues is reversed: a positive
input decreases the workload admitted into queues \(Q_1\) and \(Q_2\).

\item[\emph{(iv)}] \textbf{Mode 4:} The dynamics are given by \((A_4,B_4)\). As
in Mode~3, a positive input decreases the workload admitted into queues \(Q_1\)
and \(Q_2\), but the baseline propagation through the network differs because
\(A_4 \neq A_3\).
\end{itemize}
The dynamics are
\[
x_{t+1}=A_i x_t + B_i u_t + w_t,
\qquad i\in\{1,2,3,4\},
\]
where
\[
\begingroup
\setlength{\arraycolsep}{2pt}
\renewcommand{\arraystretch}{0.9}
\begin{gathered}
A_1=\begin{bmatrix}
0.80 & 0.08 & 0\\
0.08 & 0.70 & 0\\
0.12 & 0.08 & 0.84
\end{bmatrix},\,
A_2=\begin{bmatrix}
0.68 & 0.08 & 0\\
0.08 & 0.78 & 0\\
0.12 & 0.10 & 0.86
\end{bmatrix},\,
A_3=\begin{bmatrix}
0.84 & 0.08 & 0\\
0.08 & 0.84 & 0\\
0.08 & 0.08 & 0.84
\end{bmatrix},\,
A_4=\begin{bmatrix}
0.82 & 0.08 & 0\\
0.08 & 0.84 & 0\\
0.06 & 0.08 & 0.84
\end{bmatrix},\\[6pt]
B_1=B_2=
\begin{bmatrix}
0.08 & 0.08 & 0
\end{bmatrix}^{\top},\quad
B_3=
\begin{bmatrix}
-0.08 & -0.08 & 0
\end{bmatrix}^{\top},\quad
B_4=
\begin{bmatrix}
-0.08 & -0.08 & 0
\end{bmatrix}^{\top}.
\end{gathered}
\endgroup
\]

These matrices have the following interpretation for each mode \(i\):
\begin{itemize}
\item the \(j\)-th column of \(A_i\) describes how the workload initially present in queue \(j\) is distributed after one sampling period;

\item the diagonal entry \((A_i)_{jj}\) represents the fraction of that workload that remains in queue \(j\), corresponding to work that is still awaiting service;

\item each off-diagonal entry \((A_i)_{\ell j}\), with \(\ell\neq j\), represents the fraction of the workload initially in queue \(j\) that is processed during the sampling period and then routed to queue \(\ell\).
\end{itemize}
Moreover, the column sums of \(A_i\) are at most one. The qunatity \(1-\sum_{\ell}(A_i)_{\ell j}\) represents the fraction of workload initially in queue \(j\) that is served during the sampling period and exits the modeled network. Thus, \(A_i\) captures the net one step effect of retention, service, routing, and departures in mode \(i\).

The vector \(B_i\) determines the \emph{effective direction} of the scalar
control in mode \(i\). In modes \(1\) and \(2\), \(B_i\geq 0\), so a positive
input \(u_t\) increases the workload admitted into queues \(Q_1\) and \(Q_2\).
In modes \(3\) and \(4\), this effect is reversed in the upstream queues, so
the same positive input decreases the workload admitted into queues \(Q_1\) and
\(Q_2\). In all four modes, the input has no direct effect on queue \(Q_3\),
since the third component of each \(B_i\) is zero. Thus, the controller does
not know in advance whether the unknown mode belongs to the release family
\(\{1,2\}\) or the diversion family \(\{3,4\}\).

We take \(s=[\,1\,\,1\,\,1\,]^\top\), \(r=\tfrac{1}{5}\), and
\(\gamma=[\,10\,\,8\,\,8\,]^\top\). The disturbance
\(w_t\in\mathbb{R}^3\) represents the part of the queue evolution not captured
by the nominal model \(A_i x_t+B_i u_t\), such as unmodeled arrivals,
departures, rerouting, or other unexpected changes over one sampling interval.
In particular:
\begin{itemize}
\item \(w_{j,t}>0\) means that queue \(Q_j\) contains more workload than
predicted by the nominal dynamics, for example due to extra arrivals, slower
service, or additional rerouting into that queue;
\item \(w_{j,t}<0\) means that queue \(Q_j\) contains less workload than
predicted by the nominal dynamics, for example due to cancellations,
abandonment, faster service, or rerouting out of that queue.
\end{itemize}
The only requirement is that the next state remain nonnegative, namely
$
w_t \ge -(A_i x_t + B_i u_t).
$
Assumption~\ref{ass:input-set} holds since direct computation gives
\[
\begingroup
\setlength{\arraycolsep}{2pt}
\renewcommand{\arraystretch}{0.9}
|B_1|E=|B_2|E=|B_3|E=|B_4|E=
\begin{bmatrix}
0.08 & 0.08 & 0\\
0.08 & 0.08 & 0\\
0 & 0 & 0
\end{bmatrix}.
\endgroup
\]
Each of these matrices is componentwise dominated by the corresponding
\(A_i\). Moreover, \(E^\top|r|=[\tfrac15\;\tfrac15\;0]^\top\) and
\(s=[1\;1\;1]^\top\), so \(s>E^\top|r|\) componentwise.
The model-based optimal controllers do not coincide across the four modes. Applying Algorithm~\ref{alg:recipe-ineq-cond-jk} to this four mode example yields the diagonal certificates
\[
\begin{aligned}
p_{11}&=\begin{bmatrix}5.54 & 3.42 & 6.25\end{bmatrix}^{\top},\qquad
p_{22}=\begin{bmatrix}4.14 & 5.05 & 7.14\end{bmatrix}^{\top},\\
p_{33}&=\begin{bmatrix}7.08 & 7.08 & 6.25\end{bmatrix}^{\top},\qquad
p_{44}=\begin{bmatrix}6.06 & 7.08 & 6.25\end{bmatrix}^{\top}.
\end{aligned}
\]
and the corresponding optimal known-mode gains
\[
K_1^\star=K_2^\star=-E,
\qquad
K_3^\star=K_4^\star=E.
\]
since \(r+B_1^\top p_{11}>0\), \(r+B_2^\top p_{22}>0\),
\(r+B_3^\top p_{33}<0\), and \(r+B_4^\top p_{44}<0\). Hence modes \(1\) and \(2\)
favor one sign of the scalar input, while modes \(3\) and \(4\) favor the opposite sign. In particular, applying the controller designed for modes \(1\) and \(2\) in modes
\(3\) and \(4\) yields
\[
\rho(A_3-B_3E)=1.0800,
\qquad
\rho(A_4-B_4E)=1.0703,
\]
so the corresponding closed-loop systems are not asymptotically stable. 
The symmetric mixed cost vectors are then chosen as
\[
p_{12}=p_{21}=
\begin{bmatrix}5.5357 & 5.0476 & 7.1429\end{bmatrix}^{\top},\qquad
p_{13}=p_{31}=
\begin{bmatrix}8.9500 & 8.9500 & 6.2500\end{bmatrix}^{\top},
\]
\[
p_{14}=p_{41}=
\begin{bmatrix}7.7603 & 8.7859 & 6.2500\end{bmatrix}^{\top},\qquad
p_{23}=p_{32}=
\begin{bmatrix}8.9500 & 8.9500 & 7.1429\end{bmatrix}^{\top},
\]
\[
p_{24}=p_{42}=
\begin{bmatrix}7.7603 & 8.7859 & 7.1429\end{bmatrix}^{\top},\qquad
p_{34}=p_{43}=
\begin{bmatrix}7.0833 & 7.0833 & 6.2500\end{bmatrix}^{\top}.
\]
The remaining gain bound \(\gamma\) is chosen sufficiently large so that
\(0\leq p_{ij}\leq \gamma\) for all \(i,j\in\{1,\ldots,4\}\) and the remaining
inequalities in~\eqref{ineq-cond-jk} hold; one convenient feasible choice is
$\setlength{\arraycolsep}{2pt}
\gamma=\begin{bmatrix}15&15&15\end{bmatrix}^{\top}.
$
Simulations
of the control law~\eqref{eq:cert-unknown-sign} are shown in
Figure~\ref{fig:queue_workloads_and_admission}. As noted earlier, the adaptive
policy initially explores in order to learn the true dynamics. At time \(t=30\),
the true dynamics switch; in response, the controller again explores sufficiently
to identify the new mode and then collapses back to the corresponding model-based
optimal control law. For more details on the implementation, refer to the
code\footnote{Code available at:
\href{https://github.com/Fethi-Bencherki/minimax-adaptive-control-of-positive-systems}{https://github.com/Fethi-Bencherki/minimax-adaptive-control-of-positive-systems}}.
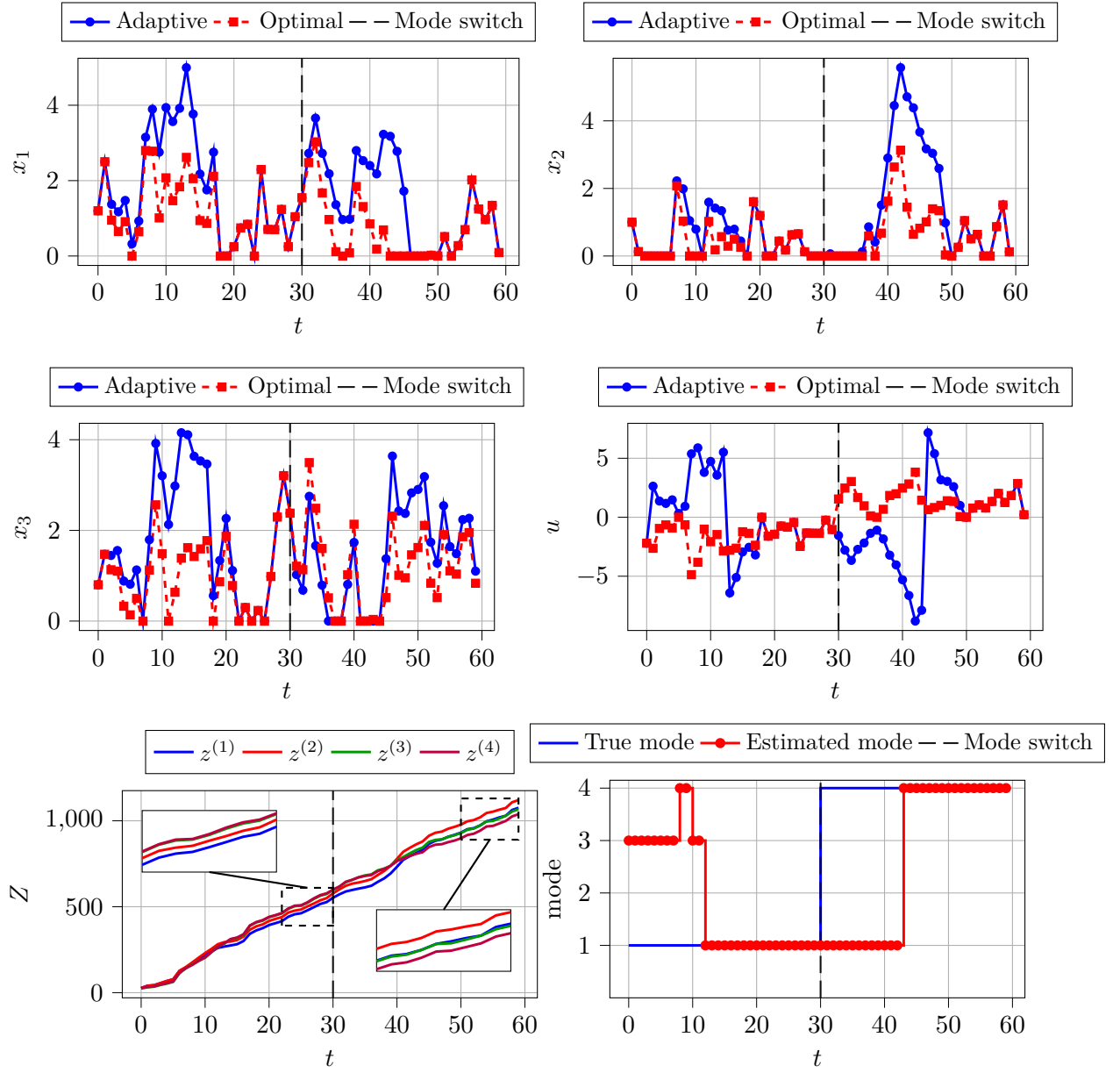
\begin{figure}[!htbp]
   \centering
    \captionsetup[subfigure]{skip=0pt}
    \begin{subfigure}[t]{0.49\textwidth}
        \centering
        \input{plots/queue/queue_1_workload.tex}
        \label{fig:queue1_workload}
    \end{subfigure}\hspace{0.005\textwidth}
        \vspace{-0.35em}
    \begin{subfigure}[t]{0.49\textwidth}
        \centering
        \input{plots/queue/queue_2_workload.tex}
        \label{fig:queue2_workload}
    \end{subfigure}
        \vspace{-0.35em}
    \begin{subfigure}[t]{0.49\textwidth}
        \centering
        \input{plots/queue/queue_3_workload.tex}
        \label{fig:queue3_workload}
    \end{subfigure}\hspace{0.005\textwidth}
        \vspace{-0.35em}
    \begin{subfigure}[t]{0.49\textwidth}
        \centering
        \input{plots/queue/admission_command.tex}
        \label{fig:admission_command}
    \end{subfigure}
        \vspace{-0.35em}
    \begin{subfigure}[t]{0.49\textwidth}
        \centering
        \input{plots/queue/history_variables.tex}
        \label{fig:queue_history_variables}
    \end{subfigure}
        \vspace{-0.35em}
    \begin{subfigure}[t]{0.49\textwidth}
        \centering
        \input{plots/queue/mode_estimate.tex}
        \label{fig:queue_mode_estimate}
    \end{subfigure}\hspace{0.005\textwidth}
        \vspace{-0.35em}
\caption{Closed loop trajectories under the adaptive certainty equivalence controller and the optimal controller with known mode. The dashed vertical line indicates the mode switching time. The figures show the workloads in queues \(Q_1\)--\(Q_3\), the admission command, the history variables, and the true and estimated modes.}
    \label{fig:queue_workloads_and_admission}
\end{figure}
 \end{example}
\subsection{Adversary policy in the learning setting}
Consider a two model uncertainty set
$\mathcal{M}=\{(a_1,b_1),\, (a_2,b_2)\}$.
For a prespecified input constraint set \(\mathcal{U}(x)\), let \(K_1\) and \(K_2\) denote the optimal
state feedback gains for the plants \((a_1,b_1)\) and \((a_2,b_2)\), respectively, where
\(K_1,\,K_2 \in \mathcal{K}(E)\).
Fix indices \(i,k\in\{1,2\}\). Suppose the true plant is \((a_i,b_i)\) while the controller applies the
state feedback law \(u = K_k x\).
As shown in the proof of Theorem~\ref{thm1} (cf.\ \eqref{advers}), when \(\gamma\) is above a suitable threshold,
the adversary selects
\[
v^\ast \;=\; \max\{(a_1+b_1K_k)x,\ (a_2+b_2K_k)x\}=\max_{\ell\in\{1,2\}}a_\ell+b_\ell K_k.
\]
Relative to the true plant \((a_i,b_i)\), the induced disturbance is
\begin{align*}
w^\ast(i,k)
\;=\; v^\ast - (a_i+b_iK_k)x 
\;=\; \Big(\max_{\ell\in\{1,2\}}(a_\ell+b_\ell K_k) - (a_i+b_iK_k)\Big)x.
\end{align*}
Since each gain \(K_i\) is chosen to optimally stabilize its own plant, it minimizes the resulting closed loop factor for that plant over the admissible class \(\mathcal{K}(E)\). In particular, for all \(\ell \neq i\),
\[
a_i + b_i K_i \leq a_i + b_i K_\ell,
\qquad \forall\, K_i, K_\ell \in \mathcal{K}(E).
\]
Let \(j\neq i\) denote the other index in \(\{1,2\}\). Two cases arise.
\begin{itemize}
\item[\textit{(a)}] \textbf{Matched gain (\(k=i\neq j\)).}
In this case, the adversary picks
\begin{multline*}
w^\ast(i,i)
= \Big(\max\{a_i+b_iK_i,\ a_j+b_jK_i\} - (a_i+b_iK_i)\Big)x 
= \big((a_j+b_jK_i)-(a_i+b_iK_i)\big)x\ge 0,
\end{multline*}
i.e., the adversary switches to the other model whenever doing so increases the next state.

\item[\textit{(b)}] \textbf{Mismatched gain (\(k=j\neq i\)).}
Since \(j=\arg\max_{\ell\in\{1,2\}}(a_\ell+b_\ell K_j)\), we have
\begin{multline*} 
w^\ast(i,j) = \Big(\max\{a_i+b_i K_j,\ a_j+b_j K_j\} - (a_i+b_i K_j)\Big)x  = \big((a_i+b_i K_j) - (a_i+b_i K_j)\big)x \;=\; 0, 
\end{multline*}
the adversary reinforces the controller’s incorrect gain choice by setting the disturbance to zero, since the resulting closed loop factor is already the larger one.
\end{itemize} 
As a result, the adversary may deviate from the model-based case in~\cite{gurpegui2025minimax}, where the choice \(w^\ast=0\) is independent of the controller’s gain selection. This aspect is specific to our \emph{learning} setup: because the controller must infer which model is active, the adversary can exploit mismatches between the true plant and the applied gain to make the observed dynamics appear more consistent with the \emph{wrong} model. In this sense, it actively complicates identification by steering the next state toward the model that maximizes \(\bigl(a_\ell+b_\ell K_k\bigr)x\).

%% file: plots/water_tank/state_x1.tex
 \begin{tikzpicture}
\begin{axis}[
    width=\linewidth,
    height=0.6\linewidth,
    xlabel={$t$},
    ylabel={$x_{1}$},
    grid=both,
    tick align=outside,
    legend style={
        at={(0.5,1.08)},
        anchor=south,
        legend columns=2,
        draw=black,
        fill=white,
       font=\small
    },
    legend cell align={left},
]
        \addplot+[mark=*] coordinates {(0,1) (1,0.49474876) (2,0.68733967) (3,0) (4,0) (5,1.3439971) (6,2.3440191) (7,2.8190352) (8,1.1210428) (9,1.9641767) (10,6.2143308) (11,6.9460107) (12,5.5666558) (13,4.4092257) (14,1.8676429) (15,1.7223847) (16,2.2614061) (17,1.0705342) (18,0.0024454486) (19,0.51558941) (20,2.5058489) (21,2.8551521) (22,2.2465665) (23,0) (24,1.1130661) (25,0) (26,0.48821182) (27,0) (28,0) (29,0.3761409) (30,0.46682014) (31,1.401062) (32,0.94777825) (33,1.0264468) (34,4.4531649) (35,3.6498707) (36,6.2721451) (37,6.5322779) (38,5.7162134) (39,4.0275574)};
        \addlegendentry{Adaptive}
        \addplot+[mark=square*, dashed] coordinates {(0,1) (1,1.2147488) (2,1.2554969) (3,0.65101132) (4,0.43488418) (5,1.648416) (6,3.0947112) (7,3.6004549) (8,1.3832449) (9,0.47586182) (10,1.0807763) (11,2.1250023) (12,1.8236938) (13,1.7891523) (14,0.033591592) (15,0.43854871) (16,1.3627209) (17,0.44145456) (18,0) (19,0.51387759) (20,2.5046507) (21,2.8543133) (22,2.2459794) (23,0) (24,1.1130661) (25,0) (26,0.48821182) (27,0) (28,0) (29,0.3761409) (30,0.46682014) (31,1.401062) (32,0.94777825) (33,1.0264468) (34,4.4531649) (35,3.6498707) (36,6.2721451) (37,6.5322779) (38,5.7162134) (39,4.0275574)};
        \addlegendentry{Optimal}
    \end{axis}
    \end{tikzpicture}

%% file: plots/water_tank/state_x2.tex
 \begin{tikzpicture}
\begin{axis}[
    width=\linewidth,
    height=0.6\linewidth,
    xlabel={$t$},
    ylabel={$x_{2}$},
    grid=both,
    tick align=outside,
    legend style={
        at={(0.5,1.08)},
        anchor=south,
        legend columns=2,
        draw=black,
        fill=white,
           font=\small
    },
    legend cell align={left},
]
        \addplot+[mark=*] coordinates {(0,0.8) (1,0.19106037) (2,0) (3,2.279428) (4,0) (5,0) (6,1.2526127) (7,2.0177224) (8,3.0289621) (9,3.7197583) (10,1.6998065) (11,0.35191025) (12,0) (13,0) (14,0.0074828504) (15,1.3366933) (16,0) (17,0.058162867) (18,0) (19,0.71051223) (20,0) (21,0) (22,0) (23,0) (24,0) (25,0.20148542) (26,0) (27,0) (28,0) (29,0) (30,0) (31,0) (32,0.38203838) (33,1.7925752) (34,0) (35,2.4153538) (36,1.1476701) (37,1.2566245) (38,1.0826131) (39,0)};
        \addlegendentry{Adaptive}
        \addplot+[mark=square*, dashed] coordinates {(0,0.8) (1,0) (2,0) (3,1.7295562) (4,0) (5,0) (6,0.17741499) (7,0) (8,0) (9,0) (10,0.58387897) (11,0.017132007) (12,0) (13,0) (14,0.0074828504) (15,1.3366933) (16,0) (17,0.058162867) (18,0) (19,0.71051223) (20,0) (21,0) (22,0) (23,0) (24,0) (25,0.20148542) (26,0) (27,0) (28,0) (29,0) (30,0) (31,0) (32,0.38203838) (33,1.7925752) (34,0) (35,2.4153538) (36,1.1476701) (37,1.2566245) (38,1.0826131) (39,0)};
        \addlegendentry{Optimal}
    \end{axis}
    \end{tikzpicture}

%% file: plots/water_tank/control_input.tex
 \begin{tikzpicture}
\begin{axis}[
    width=\linewidth,
    height=0.6\linewidth,
    xlabel={$t$},
    ylabel={$u_{t}$},
    grid=both,
    tick align=outside,
    legend style={
        at={(0.5,1.08)},
        anchor=south,
        legend columns=2,
        draw=black,
        fill=white,
           font=\small
    },
    legend cell align={left},
]
        \addplot+[mark=*] coordinates {(0,1.8) (1,0.68580913) (2,0.68733967) (3,2.279428) (4,0) (5,1.3439971) (6,3.5966318) (7,4.8367576) (8,4.1500048) (9,-5.683935) (10,-7.9141373) (11,-7.297921) (12,-5.5666558) (13,-4.4092257) (14,-1.8751258) (15,-3.059078) (16,-2.2614061) (17,-1.128697) (18,-0.0024454486) (19,-1.2261016) (20,-2.5058489) (21,-2.8551521) (22,-2.2465665) (23,0) (24,-1.1130661) (25,-0.20148542) (26,-0.48821182) (27,0) (28,0) (29,-0.3761409) (30,-0.46682014) (31,-1.401062) (32,-1.3298166) (33,-2.8190219) (34,-4.4531649) (35,-6.0652245) (36,-7.4198152) (37,-7.7889024) (38,-6.7988265) (39,-4.0275574)};
        \addlegendentry{Adaptive}
        \addplot+[mark=square*, dashed] coordinates {(0,-1.8) (1,-1.2147488) (2,-1.2554969) (3,-2.3805675) (4,-0.43488418) (5,-1.648416) (6,-3.2721262) (7,-3.6004549) (8,-1.3832449) (9,-0.47586182) (10,-1.6646553) (11,-2.1421343) (12,-1.8236938) (13,-1.7891523) (14,-0.041074443) (15,-1.775242) (16,-1.3627209) (17,-0.49961743) (18,0) (19,-1.2243898) (20,-2.5046507) (21,-2.8543133) (22,-2.2459794) (23,0) (24,-1.1130661) (25,-0.20148542) (26,-0.48821182) (27,0) (28,0) (29,-0.3761409) (30,-0.46682014) (31,-1.401062) (32,-1.3298166) (33,-2.8190219) (34,-4.4531649) (35,-6.0652245) (36,-7.4198152) (37,-7.7889024) (38,-6.7988265) (39,-4.0275574)};
        \addlegendentry{Optimal}
    \end{axis}
    \end{tikzpicture}

%% file: plots/water_tank/disturbance.tex
  \begin{tikzpicture}
\begin{axis}[
    width=\linewidth,
    height=0.62\linewidth,
    xlabel={$t$},
    ylabel={$w_{t}$},
    grid=both,
    tick align=outside,
    legend style={
        at={(0.5,1.08)},
        anchor=south,
        legend columns=2,
        draw=black,
        fill=white,
           font=\small
    },
    legend cell align={left},
]
        \addplot+[mark=*] coordinates {(0,-0.36525124) (1,0.40517279) (2,-0.2062019) (3,-1.5955996) (4,1.3439971) (5,1.9408199) (6,1.2390006) (7,-1.1370735) (8,-0.49240961) (9,0.74767302) (10,0.72619206) (11,0.317347) (12,0.5125666) (13,-1.218815) (14,0.40680346) (15,-0.41462584) (16,-0.51245007) (17,-0.81090763) (18,0.51387759) (19,1.3633729) (20,1.1010579) (21,0.24796003) (22,-1.5725966) (23,1.1130661) (24,-0.77914628) (25,0.26657785) (26,-0.34174827) (27,-0) (28,0.3761409) (29,0.20352151) (30,1.0742879) (31,-0.032965154) (32,-0.057240238) (33,1.7628194) (34,0.53265527) (35,1.0603464) (36,0.87933927) (37,-0.23866814) (38,-1.1646664) (39,0.28455163)};
        \addlegendentry{$w_{t,1}$}
        \addplot+[mark=square*] coordinates {(0,-1.4889396) (1,-0.60596541) (2,1.7295562) (3,-2.5073707) (4,-0) (5,0.17741499) (6,-1.2353668) (7,-1.4457608) (8,-0.50893417) (9,0.58387897) (10,-0.15803168) (11,-0.10557308) (12,-0) (13,0.0074828504) (14,1.3344485) (15,-0.40100799) (16,0.058162867) (17,-0.01744886) (18,0.71051223) (19,-0.21315367) (20,-0) (21,-0) (22,-0) (23,-0) (24,0.20148542) (25,-0.060445627) (26,-0) (27,-0) (28,-0) (29,-0) (30,-0) (31,0.38203838) (32,1.6779637) (33,-0.53777256) (34,2.4153538) (35,0.42306395) (36,0.91232346) (37,0.70562578) (38,-0.32478394) (39,-0)};
        \addlegendentry{$w_{t,2}$}
    \end{axis}
    \end{tikzpicture}

%% file: plots/water_tank/history_variables.tex
\begin{tikzpicture}
    \begin{axis}[
    width=\linewidth,
    height=0.6\linewidth,
    xlabel={$t$},
    ylabel={$z_t^{(\sigma)}$},
    grid=both,
    tick align=outside,
    legend style={
        at={(0.5,1.08)},
        anchor=south,
        legend columns=2,
        draw=black,
        fill=white,
           font=\small
    },
    legend cell align={left},
]
        \addplot+[mark=*] coordinates {(0,10.339601) (1,12.871687) (2,73.898943) (3,112.05657) (4,123.34614) (5,166.44852) (6,209.17406) (7,295.56812) (8,383.92512) (9,508.38532) (10,703.31037) (11,879.09484) (12,1013.4375) (13,1106.0119) (14,1119.8708) (15,1197.8732) (16,1244.9409) (17,1270.9703) (18,1293.009) (19,1338.4319) (20,1406.2174) (21,1474.9966) (22,1525.5893) (23,1534.9391) (24,1554.9682) (25,1563.4253) (26,1574.4198) (27,1574.4198) (28,1577.5794) (29,1588.0756) (30,1608.0046) (31,1630.9055) (32,1650.2456) (33,1744.35) (34,1792.4664) (35,1932.4803) (36,2090.3856) (37,2252.6889) (38,2409.8459) (39,2506.3198)};
        \addlegendentry{$z^{(+)}$}
        \addplot+[mark=square*] coordinates {(0,40.291601) (1,58.844188) (2,103.81519) (3,179.90249) (4,191.19207) (5,211.93033) (6,253.22211) (7,298.91754) (8,315.77714) (9,336.65457) (10,346.70537) (11,352.01041) (12,356.31597) (13,366.74109) (14,403.51945) (15,417.02751) (16,422.78616) (17,430.034) (18,452.11338) (19,468.89456) (20,478.14344) (21,480.22631) (22,493.43612) (23,502.78587) (24,514.36784) (25,518.11823) (26,520.98892) (27,520.98892) (28,524.1485) (29,525.85808) (30,534.8821) (31,544.70997) (32,587.13988) (33,615.39187) (34,680.25002) (35,699.73353) (36,729.92807) (37,749.57353) (38,767.47632) (39,769.86656)};
        \addlegendentry{$z^{(-)}$}
    \end{axis}
    \end{tikzpicture}

%% file: plots/water_tank/sign_estimate.tex
 \begin{tikzpicture}
\begin{axis}[
    width=\linewidth,
    height=0.6\linewidth,
    xlabel={$t$},
    ylabel={sign},
    grid=both,
    tick align=outside,
    legend style={
        at={(0.5,1.08)},
        anchor=south,
        legend columns=2,
        draw=black,
        fill=white,
           font=\small
    },
    legend cell align={left},
    ytick={-1,0,1},
    yticklabels={$-1$,$0$,$+1$},
]
        \addplot+[const plot, mark=none] coordinates {(0,-1) (1,-1) (2,-1) (3,-1) (4,-1) (5,-1) (6,-1) (7,-1) (8,-1) (9,-1) (10,-1) (11,-1) (12,-1) (13,-1) (14,-1) (15,-1) (16,-1) (17,-1) (18,-1) (19,-1) (20,-1) (21,-1) (22,-1) (23,-1) (24,-1) (25,-1) (26,-1) (27,-1) (28,-1) (29,-1) (30,-1) (31,-1) (32,-1) (33,-1) (34,-1) (35,-1) (36,-1) (37,-1) (38,-1) (39,-1)};
        \addlegendentry{Estimated sign}
        \addplot+[const plot, mark=*] coordinates {(0,1) (1,1) (2,1) (3,1) (4,1) (5,1) (6,1) (7,1) (8,-1) (9,-1) (10,-1) (11,-1) (12,-1) (13,-1) (14,-1) (15,-1) (16,-1) (17,-1) (18,-1) (19,-1) (20,-1) (21,-1) (22,-1) (23,-1) (24,-1) (25,-1) (26,-1) (27,-1) (28,-1) (29,-1) (30,-1) (31,-1) (32,-1) (33,-1) (34,-1) (35,-1) (36,-1) (37,-1) (38,-1) (39,-1)};
        \addlegendentry{True sign}
    \end{axis}
    \end{tikzpicture}

%% file: plots/queue/queue_1_workload.tex
\begin{tikzpicture}

\definecolor{darkgrey176}{RGB}{176,176,176}

\begin{axis}[
    width=1.05\linewidth,
    height=0.6\linewidth,
    grid=both,
    tick align=outside,
    legend style={
        at={(0.5,1.08)},
        anchor=south,
        legend columns=3,
        draw=black,
        fill=white,
        font=\small
    },
    legend cell align={left},
    tick pos=left,
    x grid style={darkgrey176},
    xlabel={$t$},
    xmajorgrids,
    xmin=-2.95, xmax=61.95,
    xtick style={color=black},
    y grid style={darkgrey176},
    ylabel={$x_1$},
    ymajorgrids,
    ymin=-0.249830222383749, ymax=5.24643467005874,
    ytick style={color=black}
]

\addplot+[
    blue,
    mark=*,
    mark size=1.4pt,
    mark options={solid, fill=blue, draw=blue},
    line width=1.12pt
]
table {%
0 1.2
1 2.50060451532162
2 1.37166531362362
3 1.17312708098341
4 1.47454946777324
5 0.32232903395898
6 0.927869066273487
7 3.15053975617251
8 3.89533325720889
9 2.7547886679496
10 3.93787430354505
11 3.56717070805575
12 3.91814098883289
13 4.99660444767499
14 3.76595685612096
15 2.17456925271618
16 1.75399373085322
17 2.75439707205308
18 0
19 0
20 0.247348439918604
21 0.745593339508123
22 0.840521376947025
23 0
24 2.29355258349946
25 0.706431382629372
26 0.699617000641067
27 1.23581343454705
28 0.247931359158294
29 1.04166612849017
30 1.54677829007982
31 2.72355804575763
32 3.65784062897248
33 2.72368844675516
34 2.18178966886806
35 1.36159571938876
36 0.966834812340677
37 0.975118111026548
38 2.79759635905081
39 2.53030055591141
40 2.40065811593073
41 2.17884901282976
42 3.22837277826182
43 3.17626582117595
44 2.7780702754516
45 1.72138952174391
46 0
47 0
48 0
49 0.0201275621938371
50 0
51 0.514573468706945
52 0
53 0.274967174520659
54 0.700042190927272
55 2.01893859324751
56 1.2482442121966
57 0.967469468960278
58 1.34346970811795
59 0.0918443665152762
};
\addlegendentry{Adaptive}

\addplot+[
    red,
    dashed,
    mark=square*,
    mark size=1.4pt,
    line width=1.12pt,
    mark options={solid, fill=red, draw=red}
]
table {%
0 1.2
1 2.50060451532162
2 0.950596448176717
3 0.65049104768186
4 0.910551190838781
5 0
6 0.644219516389584
7 2.7976166541524
8 2.78132166056277
9 1.01147652569715
10 2.07420976298193
11 1.46848896928311
12 1.83634282362766
13 2.61590479368032
14 2.0518531052448
15 0.940414552085347
16 0.865402346399023
17 2.11461127524606
18 0
19 0
20 0.247348439918604
21 0.745593339508123
22 0.840521376947025
23 0
24 2.29355258349946
25 0.706431382629372
26 0.699617000641067
27 1.23581343454705
28 0.247931359158294
29 1.04166612849017
30 1.54677829007982
31 2.47607351934486
32 3.02842418496904
33 1.67266577755702
34 0.968242742180616
35 0.114484646621162
36 0
37 0.0845760821654239
38 1.84449610654716
39 1.31085585828575
40 0.8521415431814
41 0.1849255894672
42 0.69189553632919
43 0
44 0
45 0
46 0
47 0
48 0
49 0.0201275621938371
50 0
51 0.514573468706945
52 0
53 0.274967174520659
54 0.700042190927272
55 2.01893859324751
56 1.2482442121966
57 0.967469468960278
58 1.34346970811795
59 0.0918443665152762
};
\addlegendentry{Optimal}

\addplot [
    black,
    thick,
    opacity=0.9,
    dash pattern=on 7.4pt off 3.2pt
]
table {%
30 -0.249830222383749
30 5.24643467005873
};
\addlegendentry{Mode switch}
\end{axis}
\end{tikzpicture}

%% file: plots/queue/queue_2_workload.tex
\begin{tikzpicture}

\definecolor{darkgrey176}{RGB}{176,176,176}

\begin{axis}[
    width=\linewidth,
    height=0.6\linewidth,
    grid=both,
    tick align=outside,
    legend style={
        at={(0.5,1.08)},
        anchor=south,
        legend columns=3,
        draw=black,
        fill=white,
        font=\small
    },
    legend cell align={left},
    tick pos=left,
    title={Queue 2 workload},
    x grid style={darkgrey176},
    xlabel={$t$},
    xmajorgrids,
    xmin=-2.95, xmax=61.95,
    xtick style={color=black},
    y grid style={darkgrey176},
    ylabel={$x_2$},
    ymajorgrids,
    ymin=-0.278690491791891, ymax=5.85250032762972,
    ytick style={color=black}
]

\addplot+[
    blue,
    mark=*,
     mark size=1.4pt,
    mark options={solid, fill=blue, draw=blue},
    line width=1.12pt
]
table {%
0 1
1 0.131075893721506
2 0
3 0
4 0
5 0
6 0.00147734687462651
7 2.22387876377527
8 1.98731544696138
9 1.04821007043384
10 0.792396131249973
11 0
12 1.59314010521025
13 1.41880992818503
14 1.3397531682206
15 0.766933639422664
16 0.793673071067765
17 0.443875250490444
18 0
19 1.60622152798083
20 1.19845934120301
21 0
22 0
23 0.441910335092597
24 0.176543035767935
25 0.627425519438257
26 0.65645603765558
27 0.125768745296336
28 0
29 0
30 0
31 0.0656787946047844
32 0
33 0
34 0
35 0
36 0.127441860965718
37 0.863001583638548
38 0.415844333280201
39 1.50799629725321
40 2.89947413111961
41 4.45223791619738
42 5.57380983583783
43 4.71147689506871
44 4.38575276784696
45 3.66770068506931
46 3.17282204279138
47 3.0393957877812
48 2.59204366961619
49 0.980606257565638
50 0
51 0.259331648532868
52 1.04687575425939
53 0.509840540081564
54 0.638459474015135
55 0
56 0
57 0.868520215776381
58 1.51091391740857
59 0.123708870383545
};
\addlegendentry{Adaptive}

\addplot+[
    red,
    dashed,
    mark=square*,
    mark size=1.4pt,
    line width=1.12pt,
    mark options={solid, fill=red, draw=red}
]
table {%
0 1
1 0.131075893721506
2 0
3 0
4 0
5 0
6 0
7 2.07426738260931
8 1.03464942744684
9 0
10 0
11 0
12 1.02239279192133
13 0.183141618898995
14 0.573638816463255
15 0.291942741333113
16 0.499178714252243
17 0.26128874926482
18 0
19 1.60622152798083
20 1.19845934120301
21 0
22 0
23 0.441910335092597
24 0.176543035767935
25 0.627425519438257
26 0.65645603765558
27 0.125768745296336
28 0
29 0
30 0
31 0
32 0
33 0
34 0
35 0
36 0
37 0.591061501575579
38 0
39 0.677804093187298
40 1.62240055952318
41 2.63364084225603
42 3.13070215099806
43 1.44636583633455
44 0.642229528609849
45 0.822623023249108
46 1.01056301980803
47 1.39607893031385
48 1.343122857941
49 0.0314264406924977
50 0
51 0.259331648532868
52 1.04687575425939
53 0.509840540081564
54 0.638459474015135
55 0
56 0
57 0.868520215776381
58 1.51091391740857
59 0.123708870383545
};
\addlegendentry{Optimal}

\addplot [
    black,
    thick,
    opacity=0.9,
    dash pattern=on 7.4pt off 3.2pt
]
table {%
30 -0.278690491791891
30 5.85250032762972
};
\addlegendentry{Mode switch}

\end{axis}

\end{tikzpicture}

%% file: plots/queue/queue_3_workload.tex
\begin{tikzpicture}

\definecolor{darkgrey176}{RGB}{176,176,176}

\begin{axis}[
    width=\linewidth,
    height=0.6\linewidth,
    grid=both,
    tick align=outside,
    legend style={
        at={(0.5,1.08)},
        anchor=south,
        legend columns=3,
        draw=black,
        fill=white,
        font=\small
    },
    legend cell align={left},
    tick pos=left,
    x grid style={darkgrey176},
    xlabel={$t$},
    xmajorgrids,
    xmin=-2.95, xmax=61.95,
    xtick style={color=black},
    y grid style={darkgrey176},
    ylabel={$x_3$},
    ymajorgrids,
    ymin=-0.207722903541419, ymax=4.36218097436981,
    ytick style={color=black}
]

\addplot+[
    blue,
    mark=*,
    mark size=1.4pt,
    mark options={solid, fill=blue, draw=blue},
    line width=1.12pt
]
table {%
0 0.8
1 1.47318310315731
2 1.44680033802691
3 1.55637347273908
4 0.882497202927281
5 0.812680996922404
6 1.12709823035552
7 0
8 1.79383076913293
9 3.91800564391244
10 3.20461499884044
11 2.13025786923209
12 2.97931923755585
13 4.15445807082839
14 4.10983526320791
15 3.63529897931443
16 3.53194251526454
17 3.46386106823347
18 0.559824255868807
19 1.33674744156912
20 2.26292905322859
21 1.1144790405704
22 0.000604112985301386
23 0.299430381529189
24 0
25 0.232396501820151
26 0
27 0.983339198488164
28 2.29409608598148
29 3.20717664518768
30 2.38039326042513
31 1.02481892277005
32 0.680745247876593
33 2.75190993258916
34 1.66307498277322
35 0.791791930332765
36 0
37 0
38 0
39 0.808789612767101
40 1.73402504824756
41 0
42 0
43 0
44 0
45 1.37413987159161
46 3.63941417368617
47 2.4266700363511
48 2.37428955170358
49 2.82771289459249
50 2.90339036076443
51 3.18529654705152
52 1.73912008897235
53 1.27629858694968
54 2.54396062212139
55 1.64303523598666
56 1.48637780081352
57 2.23761966195545
58 2.26782500199034
59 1.1007906271349
};
\addlegendentry{Adaptive}

\addplot+[
    red,
    dashed,
    mark=square*,
    mark size=1.4pt,
    line width=1.12pt,
    mark options={solid, fill=red, draw=red}
]
table {%
0 0.8
1 1.47318310315731
2 1.13099868894173
3 1.10123611794588
4 0.328048413184895
5 0.136158180789943
6 0.500799838691641
7 0
8 1.12473293299467
9 2.56415160086681
10 1.48545482413227
11 0
12 0.635921238107843
13 1.38832718382879
14 1.61872997270285
15 1.42460202320247
16 1.59140797213084
17 1.77460668179761
18 0
19 0.866495066639322
20 1.86791705828756
21 0.782668964819934
22 0
23 0.298922926621536
24 0
25 0.232396501820151
26 0
27 0.983339198488164
28 2.29409608598148
29 3.20717664518768
30 2.38039326042513
31 1.21043231757962
32 1.13247574594593
33 3.49477445316369
34 2.48780107336266
35 1.60075097548957
36 0.511396695762628
37 0
38 0
39 1.02181225888716
40 2.13499942108385
41 0
42 0
43 0.0354316287015767
44 0
45 0.516678185044222
46 2.31426874172227
47 1.01083161028376
48 0.954920913761582
49 1.46059432508669
50 1.62212558801731
51 2.10903413794394
52 0.835059665321985
53 0.516887831083376
54 1.9060555871937
55 1.1071950066474
56 1.03627200816853
57 1.85953079613366
58 1.95023035470004
59 0.834011123411043
};
\addlegendentry{Optimal}

\addplot [
    black,
    thick,
    opacity=0.9,
    dash pattern=on 7.4pt off 3.2pt
]
table {%
30 -0.207722903541419
30 4.36218097436981
};
\addlegendentry{Mode switch}

\end{axis}

\end{tikzpicture}

%% file: plots/queue/admission_command.tex
\begin{tikzpicture}

\definecolor{darkgrey176}{RGB}{176,176,176}

\begin{axis}[
    width=\linewidth,
    height=0.6\linewidth,
    grid=both,
    tick align=outside,
    legend style={
        at={(0.5,1.08)},
        anchor=south,
        legend columns=3,
        draw=black,
        fill=white,
        font=\small
    },
    legend cell align={left},
    tick pos=left,
    title={Admission command},
    x grid style={darkgrey176},
    xlabel={$t$},
    ylabel={$u$},
    xmajorgrids,
    xmin=-2.95, xmax=61.95,
    xtick style={color=black},
    y grid style={darkgrey176},
    ymajorgrids,
    ymin=-9.60048289696956, ymax=7.96212332616846,
    ytick style={color=black}
]

\addplot+[
    blue,
    mark=*,
    mark size=1.4pt,
    mark options={solid, fill=blue, draw=blue},
    line width=1.12pt
]
table {%
0 -2.2
1 2.63168040904313
2 1.37166531362362
3 1.17312708098341
4 1.47454946777324
5 0.32232903395898
6 0.929346413148113
7 5.37441851994778
8 5.88264870417027
9 3.80299873838345
10 4.73027043479502
11 3.56717070805575
12 5.51128109404314
13 -6.41541437586002
14 -5.10571002434155
15 -2.94150289213885
16 -2.54766680192099
17 -3.19827232254353
18 0
19 -1.60622152798083
20 -1.44580778112161
21 -0.745593339508123
22 -0.840521376947025
23 -0.441910335092597
24 -2.47009561926739
25 -1.33385690206763
26 -1.35607303829665
27 -1.36158217984339
28 -0.247931359158294
29 -1.04166612849017
30 -1.54677829007982
31 -2.78923684036241
32 -3.65784062897248
33 -2.72368844675516
34 -2.18178966886806
35 -1.36159571938876
36 -1.09427667330639
37 -1.8381196946651
38 -3.21344069233101
39 -4.03829685316463
40 -5.30013224705034
41 -6.63108692902714
42 -8.80218261409965
43 -7.88774271624466
44 7.16382304329855
45 5.38909020681322
46 3.17282204279138
47 3.0393957877812
48 2.59204366961619
49 1.00073381975947
50 0
51 0.773905117239813
52 1.04687575425939
53 0.784807714602223
54 1.33850166494241
55 2.01893859324751
56 1.2482442121966
57 1.83598968473666
58 2.85438362552652
59 0.215553236898821
};
\addlegendentry{Adaptive}

\addplot+[
    red,
    dashed,
    mark=square*,
    mark size=1.4pt,
    line width=1.12pt,
    mark options={solid, fill=red, draw=red}
]
table {%
0 -2.2
1 -2.63168040904313
2 -0.950596448176717
3 -0.65049104768186
4 -0.910551190838781
5 0
6 -0.644219516389584
7 -4.87188403676171
8 -3.81597108800961
9 -1.01147652569715
10 -2.07420976298193
11 -1.46848896928311
12 -2.85873561554899
13 -2.79904641257932
14 -2.62549192170805
15 -1.23235729341846
16 -1.36458106065127
17 -2.37590002451088
18 0
19 -1.60622152798083
20 -1.44580778112161
21 -0.745593339508123
22 -0.840521376947025
23 -0.441910335092597
24 -2.47009561926739
25 -1.33385690206763
26 -1.35607303829665
27 -1.36158217984339
28 -0.247931359158294
29 -1.04166612849017
30 1.54677829007982
31 2.47607351934486
32 3.02842418496904
33 1.67266577755702
34 0.968242742180616
35 0.114484646621162
36 0
37 0.675637583741003
38 1.84449610654716
39 1.98865995147305
40 2.47454210270458
41 2.81856643172323
42 3.82259768732725
43 1.44636583633455
44 0.642229528609849
45 0.822623023249108
46 1.01056301980803
47 1.39607893031385
48 1.343122857941
49 0.0515540028863348
50 0
51 0.773905117239813
52 1.04687575425939
53 0.784807714602223
54 1.33850166494241
55 2.01893859324751
56 1.2482442121966
57 1.83598968473666
58 2.85438362552652
59 0.215553236898821
};
\addlegendentry{Optimal}

\addplot [
    black,
    thick,
    opacity=0.9,
    dash pattern=on 7.4pt off 3.2pt
]
table {%
30 -9.60048289696956
30 7.96212332616847
};
\addlegendentry{Mode switch}

\end{axis}

\end{tikzpicture}

%% file: plots/queue/history_variables.tex
\begin{tikzpicture}

\definecolor{darkgrey176}{RGB}{176,176,176}

\begin{axis}[
    name=main,
    width=\linewidth,
    height=0.6\linewidth,
    grid=both,
    tick align=outside,
    legend style={
        at={(0.5,1.08)},
        anchor=south,
        legend columns=4,
        draw=black,
        fill=white,
        font=\small
    },
    legend cell align={left},
    tick pos=left,
    x grid style={darkgrey176},
    xlabel={$t$},
    xmajorgrids,
    xmin=-2.95, xmax=61.95,
    xtick style={color=black},
    y grid style={darkgrey176},
    ylabel={$Z$},
    ymajorgrids,
    ymin=-29.9266756793897, ymax=1175.44405149864,
    ytick style={color=black}
]

\addplot [line width=1.12pt, blue]
table {%
0 25.9509028287027
1 40.5410414657181
2 43.3894484759639
3 54.4011826147055
4 67.5936353721556
5 77.5280086591466
6 126.371819935463
7 143.376492118107
8 173.742866403732
9 196.763093230116
10 218.048517300748
11 238.400462493164
12 260.731976745083
13 268.432065531302
14 274.868689899299
15 281.956950662793
16 300.265113110863
17 342.49007623832
18 362.271808995281
19 375.229603519314
20 393.450518138234
21 404.32981091496
22 415.904780661758
23 441.702715290159
24 456.025750651917
25 462.076412067231
26 479.073040454059
27 497.246812394066
28 516.000448304338
29 526.977315793943
30 552.141667826917
31 572.194860655529
32 588.780200487555
33 597.483145698854
34 605.464861603447
35 612.593945264508
36 622.140288172317
37 644.655271498715
38 664.877901314086
39 693.3637603992
40 732.372903812871
41 773.234502918229
42 794.240024008597
43 813.147153491256
44 831.059557545829
45 865.663319686758
46 883.853655354744
47 891.204237769334
48 907.243122915983
49 918.231591338633
50 931.423568366041
51 950.988537361302
52 957.165987432687
53 973.313250517141
54 996.875039991868
55 1006.80094484344
56 1019.57623516438
57 1029.27685005663
58 1062.8167541622
59 1075.62449131725
};
\addlegendentry{$z^{(1)}$}

\addplot [line width=1.12pt, red]
table {%
0 27.7429028287027
1 39.6728860588145
2 43.2570354912168
3 55.9255418827768
4 67.4897348313673
5 77.6808739996016
6 127.817754745803
7 149.670536442466
8 176.029409310921
9 203.820831498217
10 230.47951883664
11 254.771227619646
12 282.092527230656
13 293.988780120205
14 297.635632437547
15 306.138181556071
16 326.366988752835
17 366.195993365138
18 385.88815424116
19 397.347091952157
20 417.185660820459
21 429.137982251086
22 439.704229687469
23 465.90360144499
24 477.333139277934
25 483.967537543364
26 502.118805229036
27 518.712635773571
28 537.396733941076
29 550.136749048099
30 577.538097950837
31 601.076074497685
32 621.941903844817
33 634.353580781437
34 641.033679441991
35 649.653082562787
36 660.298455152141
37 684.673981478681
38 707.92105145829
39 738.107467834652
40 778.883097412222
41 820.222241133201
42 842.426380836281
43 862.883520394719
44 879.567401447316
45 913.060639152197
46 930.310326507434
47 936.939653475716
48 953.842833255508
49 964.825405544022
50 977.552840113707
51 996.885489202752
52 1004.17869909178
53 1020.04384257968
54 1045.36348396227
55 1053.1295481397
56 1064.54520231169
57 1074.35394524816
58 1107.85326883824
59 1120.65447299055
};
\addlegendentry{$z^{(2)}$}

\addplot [line width=1.12pt, green!60!black]
table {%
0 24.8629028287027
1 31.359191097395
2 34.7358956835221
3 44.527577658041
4 52.6475802463511
5 62.88924063322
6 113.145067518771
7 142.302563899375
8 164.522238333446
9 185.516053001509
10 202.841163500177
11 235.464175928304
12 267.885760840837
13 284.756726133026
14 310.967249883602
15 318.212267851633
16 334.823264510986
17 389.02363826011
18 408.805371017071
19 418.971469175069
20 438.819225869354
21 449.400281310277
22 461.311459607853
23 487.745745118788
24 502.965788504145
25 506.371893754498
26 522.745953426685
27 545.000316117147
28 563.496103413894
29 574.056304452103
30 597.551084455369
31 616.609431529909
32 642.221115745793
33 656.314659621324
34 672.150818333842
35 682.142576809251
36 688.398884019687
37 711.766542382113
38 722.188771353285
39 737.938401679228
40 769.972810285296
41 789.425496914254
42 810.787452478516
43 829.73302053996
44 843.863171979947
45 877.323018675434
46 888.913884494414
47 892.398891252848
48 905.937238733404
49 915.796834253045
50 928.988811280453
51 947.561545794304
52 954.576496469096
53 972.893263930097
54 996.237777105472
55 999.218533196268
56 1013.66807485171
57 1028.78123323505
58 1054.07340814292
59 1067.2192827909
};
\addlegendentry{$z^{(3)}$}

\addplot [line width=1.12pt, purple]
table {%
0 26.7029028287027
1 33.1604376243519
2 36.1530759226644
3 46.742484312252
4 55.2753607515386
5 65.426769008899
6 116.314492361516
7 143.166555201575
8 161.465521134688
9 185.003446542036
10 205.513445091114
11 237.137649720986
12 267.442545489311
13 287.87629654216
14 311.96329812664
15 321.003895175531
16 339.876836331221
17 391.038732847978
18 410.820465604939
19 422.014545540844
20 440.946879192809
21 451.080578630027
22 463.496069753771
23 489.647532650246
24 506.12288616689
25 508.421007702173
26 525.9148162391
27 548.99115898728
28 567.734877643185
29 577.6700790043
30 600.236792033518
31 617.618969852056
32 645.425358445324
33 656.7952138741
34 670.44958291775
35 679.07974567377
36 684.674389205783
37 706.904655688065
38 719.271582847967
39 737.504511137699
40 765.28259818392
41 780.578543138814
42 795.144887629877
43 807.898844657301
44 817.333026538559
45 847.274999920816
46 860.896471847182
47 863.870324327368
48 875.749763859369
49 884.968084631877
50 898.160061659285
51 917.042848999434
52 924.727800156953
53 942.641278863436
54 966.870434792011
55 970.416493688916
56 984.516526964943
57 998.802940958875
58 1025.43827229216
59 1038.47925684047
};
\addlegendentry{$z^{(4)}$}

\addplot [
    black,
    thick,
    opacity=0.9,
    dash pattern=on 7.4pt off 3.2pt
]
table {%
30 -29.9266756793897
30 1175.44405149864
};

\draw[black, dashed, thick]
    (axis cs:22,390) rectangle (axis cs:30,610);
\coordinate (beforeboxtop) at (axis cs:26,610);

\draw[black, dashed, thick]
    (axis cs:50,890) rectangle (axis cs:59,1130);
\coordinate (afterboxbottom) at (axis cs:54.5,890);

\end{axis}

\begin{axis}[
    name=beforezoom,
    at={(main.north west)},
    anchor=north west,
    xshift=0.30cm,
    yshift=-0.30cm,
    width=0.46\linewidth,
    height=0.32\linewidth,
    grid=both,
    xmin=22, xmax=30,
    ymin=390, ymax=610,
    xtick=\empty,
    ytick=\empty,
    xlabel={},
    ylabel={},
    title={},
    x grid style={darkgrey176},
    y grid style={darkgrey176},
    axis background/.style={fill=white}
]

\addplot [line width=1.12pt, blue]
table {%
22 415.904780661758
23 441.702715290159
24 456.025750651917
25 462.076412067231
26 479.073040454059
27 497.246812394066
28 516.000448304338
29 526.977315793943
30 552.141667826917
};

\addplot [line width=1.12pt, red]
table {%
22 439.704229687469
23 465.90360144499
24 477.333139277934
25 483.967537543364
26 502.118805229036
27 518.712635773571
28 537.396733941076
29 550.136749048099
30 577.538097950837
};

\addplot [line width=1.12pt, green!60!black]
table {%
22 461.311459607853
23 487.745745118788
24 502.965788504145
25 506.371893754498
26 522.745953426685
27 545.000316117147
28 563.496103413894
29 574.056304452103
30 597.551084455369
};

\addplot [line width=1.12pt, purple]
table {%
22 463.496069753771
23 489.647532650246
24 506.12288616689
25 508.421007702173
26 525.9148162391
27 548.99115898728
28 567.734877643185
29 577.6700790043
30 600.236792033518
};

\end{axis}

\begin{axis}[
    name=afterzoom,
    at={(main.south east)},
    anchor=south east,
    xshift=-0.40cm,
    yshift=0.40cm,
    width=0.46\linewidth,
    height=0.32\linewidth,
    grid=both,
    xmin=50, xmax=59,
    ymin=890, ymax=1130,
    xtick=\empty,
    ytick=\empty,
    xlabel={},
    ylabel={},
    title={},
    x grid style={darkgrey176},
    y grid style={darkgrey176},
    axis background/.style={fill=white}
]

\addplot [line width=1.12pt, blue]
table {%
50 931.423568366041
51 950.988537361302
52 957.165987432687
53 973.313250517141
54 996.875039991868
55 1006.80094484344
56 1019.57623516438
57 1029.27685005663
58 1062.8167541622
59 1075.62449131725
};

\addplot [line width=1.12pt, red]
table {%
50 977.552840113707
51 996.885489202752
52 1004.17869909178
53 1020.04384257968
54 1045.36348396227
55 1053.1295481397
56 1064.54520231169
57 1074.35394524816
58 1107.85326883824
59 1120.65447299055
};

\addplot [line width=1.12pt, green!60!black]
table {%
50 928.988811280453
51 947.561545794304
52 954.576496469096
53 972.893263930097
54 996.237777105472
55 999.218533196268
56 1013.66807485171
57 1028.78123323505
58 1054.07340814292
59 1067.2192827909
};

\addplot [line width=1.12pt, purple]
table {%
50 898.160061659285
51 917.042848999434
52 924.727800156953
53 942.641278863436
54 966.870434792011
55 970.416493688916
56 984.516526964943
57 998.802940958875
58 1025.43827229216
59 1038.47925684047
};

\end{axis}

\draw[black, thick]
    (beforeboxtop) -- (beforezoom.south);

\draw[black, thick]
    (afterboxbottom) -- (afterzoom.north);

\end{tikzpicture}

%% file: plots/queue/mode_estimate.tex
\begin{tikzpicture}

\definecolor{darkgrey176}{RGB}{176,176,176}

\begin{axis}[
    width=\linewidth,
    height=0.62\linewidth,
    grid=both,
    tick align=outside,
    legend style={
        at={(0.5,1.08)},
        anchor=south,
        legend columns=3,
        draw=black,
        fill=white,
        font=\small
    },
    legend cell align={left},
    tick pos=left,
    title={True mode vs estimated mode},
    x grid style={darkgrey176},
    xlabel={$t$},
    xmajorgrids,
    xmin=-2.95, xmax=61.95,
    xtick style={color=black},
    y grid style={darkgrey176},
    ylabel={mode},
    ymajorgrids,
    ymin=0, ymax=4.15,
    ytick={1,2,3,4},
    ytick style={color=black}
]

\addplot [
    line width=1.12pt,
    blue,
    const plot mark left
]
table {%
0 1
1 1
2 1
3 1
4 1
5 1
6 1
7 1
8 1
9 1
10 1
11 1
12 1
13 1
14 1
15 1
16 1
17 1
18 1
19 1
20 1
21 1
22 1
23 1
24 1
25 1
26 1
27 1
28 1
29 1
30 4
31 4
32 4
33 4
34 4
35 4
36 4
37 4
38 4
39 4
40 4
41 4
42 4
43 4
44 4
45 4
46 4
47 4
48 4
49 4
50 4
51 4
52 4
53 4
54 4
55 4
56 4
57 4
58 4
59 4
};
\addlegendentry{True mode}

\addplot [
    line width=1.12pt,
    red,
    const plot mark left,
    mark=*,
    mark size=1.6pt,
    mark options={solid, fill=red, draw=red}
]
table {%
0 3
1 3
2 3
3 3
4 3
5 3
6 3
7 3
8 4
9 4
10 3
11 3
12 1
13 1
14 1
15 1
16 1
17 1
18 1
19 1
20 1
21 1
22 1
23 1
24 1
25 1
26 1
27 1
28 1
29 1
30 1
31 1
32 1
33 1
34 1
35 1
36 1
37 1
38 1
39 1
40 1
41 1
42 1
43 4
44 4
45 4
46 4
47 4
48 4
49 4
50 4
51 4
52 4
53 4
54 4
55 4
56 4
57 4
58 4
59 4
};
\addlegendentry{Estimated mode}

\addplot [
    black,
    thick,
    opacity=0.9,
    dash pattern=on 7.4pt off 3.2pt
]
table {%
30 0
30 4.15
};
\addlegendentry{Mode switch}

\end{axis}

\end{tikzpicture}

%% file: tex_input/conclusions.tex
\section{Conclusion and future directions}\label{conclus}
This paper proposes a minimax adaptive control framework for positive linear systems. The problem is formulated as a minimax dynamic game between the controller and an adversary that selects both the disturbance sequence and the true dynamics. This formulation naturally compels the controller to explore and learn the true dynamics in order to stabilize the system, thereby giving rise to the explore--exploit tradeoff commonly discussed in dual control and reinforcement learning.

We showed how past data can be compressed into history variables, on which the
stabilizing adaptive control policy depends. These history variables contain
the information needed for the controller to balance learning and control. By
reformulating the problem in terms of the history variables, we showed that the Bellman inequality
can be solved when the adversary is free to choose from a finite set of possible positive 
LTI plants.

While the present work takes a step in this direction, a complete theory of
dual adaptive control for positive systems remains an open problem. Several important avenues for future research remain open:
\begin{enumerate}
	\item[\emph{(1)}] One natural extension is to allow uncertainty not only in the dynamics but also in the stage cost. In that setting, the unknown model is described by a quadruple \((A,B,s,r)\in\mathcal M\times\mathcal S\), where \(\mathcal M\) denotes the set of admissible dynamical models and \(\mathcal S\) the set of admissible cost parameters. The controller must then learn both the true plant dynamics and the cost online in order to stabilize the system.
	
	\item[\emph{(2)}] Throughout the manuscript, we relied on model dependent history variables \(z^{(A,B)}\) for each \((A,B)\in\mathcal M\). This parameterization becomes prohibitive when one moves from a finite model set \(\mathcal M\) to a continuum of models. In that case, one must instead seek history variables whose definition is independent of \((A,B)\), in the same spirit as the covariance-matrix parameterization in~\cite{rantzer2021minimax}. A natural direction for future work is therefore to extend the framework to richer model classes, including continua of dynamical models. 
	
	\item[\emph{(3)}] Another important direction for future work is to solve the exact Bellman equation and thereby characterize the truly optimal dual controller, including its optimal exploration strategy.

	\item[\emph{(4)}] Another possible follow-up work would be to devise better approaches for solving~\eqref{ineq-cond-jk}, which is needed to synthesize the adaptive policy. This should lead to tighter certified cost and gain bounds, similar to the work in~\cite{cederberg2022synthesis}, which addressed the quadratic setting instead.
	
	\item[\emph{(5)}] Another interesting direction is to pursue application driven
	extensions supported by the positive system classes considered here, and to
	assess the practical implications of the proposed dual controllers in
	representative case studies. For instance, one may develop a minimax adaptive
	dual control framework for stochastic shortest-path problems modeled by the
	class in~\cite{ohlin2024heuristic}; see Section~\ref{sec31}. Another promising
	direction is the study of compartmental models, as in~\cite{blanchini2023optimal}.
\end{enumerate}

%% file: tex_input/acknowledgments.tex
\section*{Acknowledgments}
The authors would like to thank Tomas Meijer for his constructive feedback, which led to a significant improvement in the clarity and exposition of the paper.
The authors are affiliated with the ELLIIT Strategic Research Area at Lund University. This project received funding from the European Research Council (ERC) under grant agreement No.~834142 (ScalableControl), and was partially supported by the Wallenberg AI, Autonomous Systems and Software Program (WASP), funded by the Knut and Alice Wallenberg Foundation.

%% file: tex_input/append_prelim.tex
\section{Auxiliary Lemmata and their proofs}\label{appenA}
\begin{lemma}\label{lem:sup-ptwise}
Let $c,c_1,c_2$ be nonnegative real numbers. Then, the following holds:
\begin{enumerate}
 \item[(L1)] 
If $\lambda,\gamma\ge0$, then
\[
\max_{v\ge0}\,\bigl(\lambda v-\gamma|c-v|\bigr)=
\begin{cases}
\lambda c, & \gamma\ge \lambda,\\[2pt]
+\infty, & \gamma<\lambda.
\end{cases}
\tag{L1}
\]

 \item[(L2)] 
If $c_1,c_2\ge0$, $\alpha>0$, and $\lambda\ge0$, then
\[
\max_{v\ge0}\,\bigl(\lambda v-\alpha(|c_1-v|+|c_2-v|)\bigr)
=
\begin{cases}
\displaystyle \lambda\max\{c_1,c_2\}
\;-\;\alpha|c_1-c_2|, & 0\le \lambda \le 2\alpha,\\[6pt]
\displaystyle +\infty, & \lambda>2\alpha.
\end{cases}
\tag{L2}
\]
\end{enumerate}
\end{lemma}

\begin{proof}[Proof of (L1).]
Consider the scalar problem
\[
\max_{v\ge0}\,\bigl(\lambda v - \gamma|c-v|\bigr).
\]
If $v\le c$, then $|c-v|=c-v$ and
\[
\lambda v - \gamma|c-v|
= \lambda v - \gamma(c-v)
= (\lambda+\gamma)v - \gamma c,
\]
which is affine in $v$ with slope $\lambda+\gamma\ge0$, so on $[0,c]$ it is
maximized at $v=c$. If $v\ge c$, then $|c-v|=v-c$ and
\[
\lambda v - \gamma|c-v|
= \lambda v - \gamma(v-c)
= (\lambda-\gamma)v + \gamma c.
\]
We therefore distinguish three cases depending on $(\gamma,\,\lambda)$
\begin{itemize}
\item[\textit{(1)}] If $\gamma>\lambda$, then $\lambda-\gamma<0$, so on $[c,\infty)$ the
      expression is decreasing in $v$ and is again maximized at $v=c$.
\item[\textit{(2)}] If $\gamma=\lambda$, then the expression is constant on $[c,\infty)$,
      equal to $\lambda c$.
\item[\textit{(3)}] If $\gamma<\lambda$, then $\lambda-\gamma>0$, so on $[c,\infty)$ the
      expression grows linearly with $v$ and the supremum is $+\infty$.
\end{itemize}
Hence
\[
\max_{v\ge0}\,\bigl(\lambda v - \gamma|c-v|\bigr)
=
\begin{cases}
\lambda c, & \gamma\ge \lambda,\\[2pt]
+\infty, & \gamma<\lambda.
\end{cases}
\]
This proves \textup{(L1)}.

\medskip
\noindent\emph{Proof of \textup{(L2)}.}
Without loss of generality, assume $c_1\le c_2$; the case $c_2<c_1$ is
symmetric. Consider
\[
\max_{v\ge0}\,\Bigl(\lambda v-\alpha\bigl(|c_1-v|+|c_2-v|\bigr)\Bigr).
\]

We analyze three regions.
\begin{itemize}
    \item[\textit{(1)}] \emph{Region 1: $0\le v\le c_1$.}
Then $|c_1-v|=c_1-v$ and $|c_2-v|=c_2-v$, so
\[
\lambda v-\alpha(|c_1-v|+|c_2-v|)
= \lambda v-\alpha(c_1+c_2-2v)
= (\lambda+2\alpha)v-\alpha(c_1+c_2).
\]
This is affine in $v$ with slope $\lambda+2\alpha>0$, so on $[0,c_1]$ it is
maximized at $v=c_1$.
       \item[\textit{(2)}] \emph{Region 2: $c_1\le v\le c_2$.}
Then $|c_1-v|=v-c_1$ and $|c_2-v|=c_2-v$, so
\[
\lambda v-\alpha(|c_1-v|+|c_2-v|)
= \lambda v-\alpha(c_2-c_1).
\]
This is affine in $v$ with slope $\lambda\ge0$, so on $[c_1,c_2]$ it is
maximized at $v=c_2$.
          \item[\textit{(3)}] \emph{Region 3: $v\ge c_2$.}
Then $|c_1-v|=v-c_1$ and $|c_2-v|=v-c_2$, so
\[
\lambda v-\alpha(|c_1-v|+|c_2-v|)
= \lambda v-\alpha(2v-(c_1+c_2))
= (\lambda-2\alpha)v+\alpha(c_1+c_2).
\]
If $\lambda<2\alpha$, this is decreasing on $[c_2,\infty)$, so it is maximized
there at $v=c_2$. If $\lambda=2\alpha$, it is constant on $[c_2,\infty)$.
If $\lambda>2\alpha$, it is increasing on $[c_2,\infty)$, so the supremum is
$+\infty$.
\end{itemize}
Therefore two cases arise:
\begin{itemize}
\item[\textit{(1)}] If $0\le\lambda\le 2\alpha$, the maximum is attained at $v=c_2$, and since
      $c_2=\max\{c_1,c_2\}$ and $c_2-c_1=|c_1-c_2|$,
      \[
      \max_{v\ge0}\,\bigl(\lambda v-\alpha(|c_1-v|+|c_2-v|)\bigr)
      = \lambda c_2-\alpha(c_2-c_1)
      = \lambda\max\{c_1,c_2\}-\alpha|c_1-c_2|.
      \]
\item[\textit{(2)}] If $\lambda>2\alpha$, the maximum is $+\infty$.
\end{itemize}
Hence
\[
\max_{v\ge0}\,\bigl(\lambda v-\alpha(|c_1-v|+|c_2-v|)\bigr)
=
\begin{cases}
\displaystyle \lambda\max\{c_1,c_2\}
\;-\;\alpha|c_1-c_2|, & 0\le \lambda \le 2\alpha,\\[6pt]
\displaystyle +\infty, & \lambda>2\alpha.
\end{cases}
\]
This proves \textup{(L2)} and completes the proof of the lemma.
\end{proof}
\begin{lemma}[Vector version of Lemma~\ref{lem:sup-ptwise}]\label{lem:sup-ptwise-vec}
Let $c,c_1,c_2$ be vectors in $\mathbb R^m_+$. The maximum is taken over all
nonnegative vectors $v\in\mathbb R^m_+$. The following holds.
\begin{enumerate}
\item[(VL1)] 
Let $\lambda,\gamma\in\mathbb R^m_+$. Then,
\[
\max_{v\ge 0}\ \Big(\lambda^\top v-\gamma^\top|c-v|\Big)
=
\begin{cases}
\lambda^\top c, & \gamma\ge \lambda,\\[2pt]
+\infty, & \text{otherwise}.
\end{cases}
\tag{VL1}
\]
\item[(VL2)] 
Let $c_1,c_2\in\mathbb R^m_+$, $\alpha\in\mathbb R^m_{+}$, and
$\lambda\in\mathbb R^m_+$. Then
\begin{equation}\tag{VL2}\label{eq:VL2-ptwise}
\begin{aligned}
\max_{v\ge 0}\ \Big(\lambda^\top v
-\alpha^\top\big(|c_1-v|+|c_2-v|\big)\Big)
&=
\begin{cases}
\lambda^\top \max\{c_1,c_2\}
-\alpha^\top |c_1-c_2|,
& 0\le \lambda \le 2\alpha,\\[4pt]
+\infty, & \text{otherwise}.
\end{cases}
\end{aligned}
\end{equation}
\end{enumerate}
\end{lemma}

\medskip
\begin{proof}
The objectives separate across coordinates. Indeed,
\[
\lambda^\top v-\gamma^\top|c-v|
=
\sum_{i=1}^m\bigl(\lambda_i v_i-\gamma_i|c_i-v_i|\bigr),
\]
and
\[
\lambda^\top v-\alpha^\top\bigl(|c_1-v|+|c_2-v|\bigr)
=
\sum_{i=1}^m\bigl(\lambda_i v_i-\alpha_i(|c_{1,i}-v_i|+|c_{2,i}-v_i|)\bigr).
\]
Therefore the maximization over $v\in\mathbb R^m_+$ reduces to independent
scalar maximization problems in the coordinates $v_1,\dots,v_m$. Applying
Lemma~\ref{lem:sup-ptwise} coordinatewise yields the stated formulas. The
finiteness conditions are also coordinatewise, namely $\gamma_i\ge\lambda_i$
for all $i$ in \textup{(VL1)} and $0\le\lambda_i\le2\alpha_i$ for all $i$ in
\textup{(VL2)}, which is exactly $\gamma\ge\lambda$ and
$0\le\lambda\le2\alpha$ componentwise.
\end{proof}

\subsection{Minimax dynamic programming}\label{standarddp}
We define a general discrete-time, infinite-horizon, minimax optimal control
problem with continuous cost function and constraints as
\begin{equation}\label{eq:minimax-problem}
\begin{aligned}
\inf_{\mu}\;
\max_{w}
&\ \sum_{t=0}^{\infty} g(x(t),u(t),w(t)) \\[1mm]
\text{s.t.}\quad
& x(t+1)=f(x(t),u(t),w(t)),\quad x(0)=x_0,\\
& x(t)\in \mathcal{X};\ x(0)=x_0;\ u(t)=\mu(x(t)),\\
& u(t)\in \mathcal{U}(x(t));\; w(t)\in \mathcal{W}(x(t)),
\end{aligned}
\end{equation}
where \(f:\mathcal{X}\times\mathcal{U}\times\mathcal{W}\to\mathcal{X}\) represents the vector of
\(n\)-dimensional state variables, \(u\) the \(m\)-dimensional control
variable and \(w\) the \(q\)-dimensional disturbance.

\begin{lemma}\label{lemma}
    Suppose
\[
\max_{w\in \mathcal{W}(x)} g(x,u,w)\geq 0,\qquad \text{for all}\ x\in \mathcal{X},\ \text{for all}\ u\in \mathcal{U}(x).
\]
Then, the following statements are equivalent.

\begin{enumerate}
\item[\textit{(i)}] The general optimal control problem in~\eqref{eq:minimax-problem} has a finite value
for every \(x_0\in \mathbb{R}_+^n\).

\item[\textit{(ii)}] The recursive sequence \(\{J_k\}_{k=0}^{\infty}\) with \(J_0=0\) and
\begin{align}\label{jk}
    J_k(x)=\min_{u}\max_{w}\bigl[g(x,u,w)+J_{k-1}(f(x,u,w))\bigr].
\end{align}
has a finite limit for all \( x\in \mathcal{X}\).
\item[\textit{(iii)}] The Bellman equation
\begin{align}\label{fixed_point}
    J^*(x)=\min_{u}\max_{w}\bigl[g(x,u,w)+J^*(f(x,u,w))\bigr]
\end{align}
has nonnegative solution \(J^*(x)\), for all \( x\in \mathcal{X}\).
\end{enumerate}
\end{lemma}
\begin{proof}
    The proof is presented in~\cite[Appendix]{gurpegui2023minimax}. 
\end{proof}

%% file: tex_input/append_main.tex
\section{Model-based minimax dynamic games for positive linear systems}\label{model-based}
Understanding the model-based solution is an important step toward the model-free setting. 
Accordingly, this section shows that the model-based solution of Problem~\ref{prob:l1-robust-control} coincides with the solution derived in~\cite{gurpegui2025minimax}, even when disturbances are allowed to be negative, provided that state positivity is preserved.
When the system matrices \((A,B)\) are known, Problem~\ref{prob:l1-robust-control} reduces to an \(\mathcal{H}_\infty\)-type optimal control problem for positive systems. Equivalently, it can be formulated as a zero-sum dynamic game in which the controller is the minimizing player and the disturbance is the maximizing player~\cite{bacsar2008h}. We establish this equivalence in the next section.
\subsection{Model-based solution via dynamic programming}\label{soldp}
When the parameters \((A,B)\) are known and under Assumption~\ref{ass:input-set}, Problem~\ref{prob:l1-robust-control} reduces to
\begin{equation}\label{eq:minimax-problem}
\begin{aligned}
\inf_{\mu}\;
\sup_{ w \,\ge\, -(Ax+Bu)}
&\ \sum_{t=0}^\infty \bigl( s^\top x_t + r^\top u_t - \gamma^\top |w_t| \bigr) \\[1mm]
\text{s.t.}\quad
& x_{t+1} = A x_t + B u_t + w_t,\quad x_0 \in \mathbb{R}_+^n,\\
& u_t=\mu_t(x_t),\\
& \abs{u_t}\leq Ex_t.
\end{aligned}
\end{equation}
The following theorem characterizes the solution to~\ref{eq:minimax-problem}.
\begin{theorem}[Solution to problem~\ref{eq:minimax-problem}]\label{thm:model-based-dp}
Let \(A\in\mathbb{R}_+^{n\times n}\), \(B=[B_1\ \cdots\ B_m]\in\mathbb{R}^{n\times m}\),
\(E=[E_1^\top\ \cdots\ E_m^\top]^\top\in\mathbb{R}_+^{m\times n}\),
\(s\in\mathbb{R}_+^n\), \(r\in\mathbb{R}^m\), and \(\gamma\in\mathbb{R}_+^n\). Suppose that
$A\ge |B|E$ and 
   $s\ge E^\top |r|$. 
Then, the following statements are equivalent:
\begin{enumerate}
\item[\textit{(i)}] Problem~\eqref{eq:minimax-problem} has finite value for every \(x_0\in\mathbb{R}^n_+\).

\item[\textit{(ii)}] The recursive sequence \(\{p_k\}_{k=0}^{\infty}\), with \(p_0=0\) and
\[
    p_k
    =
    s+A^\top p_{k-1}
    -E^\top\bigl|r+B^\top p_{k-1}\bigr|,
    \qquad k\ge 1,
\]
has a finite limit \(p\in\mathbb{R}^n_+\) satisfying \(p \leq \gamma \).

\item[\textit{(iii)}] There exists \(p\in\mathbb{R}^n_+\) such that
\begin{align*}
        p
    =
    s+A^\top p
    -E^\top\bigl|r+B^\top p\bigr|,
    \qquad
    \gamma\ge p.
\end{align*}
\end{enumerate}
If these conditions hold, then problem~\eqref{eq:minimax-problem} has the optimal value \(p^\top x_0\), where \(p\) is the limit of the sequence in \textit{(ii)} that solves the algebraic equation in \textit{(iii)}. Moreover, an optimal control law is \(u_t=-Kx_t\), where
\[
K=
\begin{bmatrix}
\operatorname{sign}(r_1+p^\top B_1)E_1\\
\vdots\\
\operatorname{sign}(r_m+p^\top B_m)E_m
\end{bmatrix}.
\]
\end{theorem}
For the sake of brevity, we omit the detailed proof, as it can be readily adapted from~\cite{gurpegui2025minimax}. We provide only a brief proof sketch below.
\begin{proofsketch}
The proof is based on Lemma~\ref{lemma}. First, we reduce the
general problem setup in~Lemma~\ref{lemma} to our setting by letting
\begin{align*}
    f(x,u,w) &\coloneqq Ax+Bu+w,\\
    g(x,u,w) &\coloneqq s^\top x+r^\top u-\gamma^\top |w|.
\end{align*}
The assumptions \(s\geq E^\top |r|\) and \(A\geq |B|E\), together with the
constraint \(|u|\leq Ex\), guarantee that
\[
\sup_{w\geq -(Ax+Bu)} g(x,u,w) \geq 0 .
\]

Next, we use induction on \(J_k(x)\coloneqq p_k^\top x\) and
\(J^*(x)\coloneqq p^\top x\) for all \(x\), to prove the equivalence between
the recursive sequence \(\{p_k\}_{k=0}^{\infty}\) and the
recursion~\ref{jk} in Lemma~\ref{lemma} and between the Bellman equation for Problem~(B.1)
and the fixed point equation~\ref{fixed_point} in Lemma~\ref{lemma}. We use these equivalences to
deduce the bound for the disturbance penalty and prove the implication
``\(\Rightarrow\)''. The implication ``\(\Leftarrow\)'' is direct by applying
the disturbance penalty condition to \(J^*(x)\).
Finally, the expression for the optimal control policy is given by
\(u(t)=\mu(x(t))\), where
\[
\mu(x)
=
\arg\min_{|u|\leq Ex}
\sup_{w\geq -(Ax+Bu)}
\left\{
s^\top x+r^\top u-\gamma^\top |w|
+p^\top(Ax+Bu+w)
\right\}.
\]
Equivalently, after removing the terms independent of \(u\),
\[
\mu(x)
=
\arg \min_{|u|\leq Ex}
\sum_{i=1}^{m}
\left(r_i+p^\top B_i\right)u_i .
\]
Since, for all \(i=1,\ldots,m\), the inequality \(|u|\leq Ex\)
restricts \(u_i\) to the interval \([-E_i x,E_i x]\), the minimum is
attained when \((r_i+p^\top B_i)\) and \(u_i\) have opposite signs.
Thus,
\[
u_i
=
-\operatorname{sign}\left(r_i+p^\top B_i\right)E_i x,
\qquad i=1,\ldots,m .
\]
\end{proofsketch}
\subsection{Related minimax dynamic game and \(\ell_1\)-gain}
In this section, we interpret the problem from a robustness perspective. Specifically, we show that a solution to the minimax problem certifies robustness guarantees formulated in terms of an \(\ell_1\)-gain minimization problem. This highlights that, whenever such a solution is found, it enjoys robustness guarantees.

For mathematical convenience, in this subsection we restrict the disturbance penalty vector to be a scalar multiple of the all-ones vector; that is, instead of a general nonnegative vector \(\bar{\gamma}\in\mathbb{R}^n_+\), we consider \(\bar{\gamma}=\gamma\mathbf{1}\), where \(\gamma\ge 0\) and \(\mathbf{1}\in\mathbb{R}^n\) denotes the all-ones vector. 
Define the functional
\[
L_{\gamma}(\mu,w)
\coloneqq
\sum_{t=0}^\infty \bigl( s^\top x_t + r^\top u_t - \gamma\, \mathbf{1}^\top |w_t| \bigr),
\]
where \(u_t=\mu_t(x_t)\), and consider the minimax two-player dynamic game
\[
\begin{aligned}
\inf_{\mu}\ \sup_{ w \ge -(Ax+Bu)}
&\; L_{\gamma}(\mu,w) \\
\text{s.t.}\qquad
& x_{t+1}=Ax_t+Bu_t+w_t,\\
& u_t=\mu_t(x_t),\qquad |u_t|\le Ex_t,
\qquad \forall t\in\mathbb N .
\end{aligned}
\]
From Theorem~\ref{eq:minimax-problem}, if feasible, the game value is
\[
\inf_{\mu}\ \sup_{ w \ge -(Ax+Bu)}
L_{\gamma}(\mu,w)
=
p^\top x_0.
\]
Let \(u^\star\) and \(w^\star\) denote optimal policies of the minimizing and maximizing players, respectively. For any disturbance sequence \(w\) satisfying
$
w_t \ge -(A x_t + B u_t^\star),
\ \forall t\in\mathbb N,
$
we have
\[
\sum_{t=0}^\infty \bigl( s^\top x_t + r^\top u_t^\star \bigr)
-
\gamma \sum_{t=0}^\infty \mathbf{1}^\top |w_t|
=
L_{\gamma}(u^\star,w)
\le
L_{\gamma}(u^\star,w^\star)
=
p^\top x_0.
\]
It follows that, for every positivity preserving disturbance sequence \(w\),
\[
\sum_{t=0}^\infty \bigl( s^\top x_t + r^\top u_t^\star \bigr)
\le
p^\top x_0 + \gamma \sum_{t=0}^\infty \mathbf{1}^\top |w_t|
=
p^\top x_0 + \gamma \|w\|_{1}.
\]

In particular, if \(x_0=0\), then, by positivity of the stage cost, defining
$
    y_t \coloneqq s^\top x_t + r^\top u_t^\star \ge 0,
$
we obtain
\[
\|y\|_{1}
=
\sum_{t=0}^\infty |y_t|
=
\sum_{t=0}^\infty \bigl(s^\top x_t + r^\top u_t^\star\bigr)
\le
\gamma \|w\|_{1}.
\]
Equivalently, defining the induced \(\ell_1\)-gain from \(w\) to \(y\) under the positivity-preserving disturbance constraint by
\[
\|G_\gamma\|_{\ell_1\to\ell_1}
\coloneqq
\sup_{\,w\ge -(Ax+Bu^\star)}
\frac{\sum_{t=0}^\infty |y_t|}{\sum_{t=0}^\infty \|w_t\|_1},
\]
we obtain
$
\|G_\gamma\|_{\ell_1\to\ell_1}\le \gamma.
$
That is, the induced \(\ell_1\)-gain from \(w\) to \(y\) is upper bounded by \(\gamma\).

We define \(\gamma_\star\) as the smallest \(\gamma\ge0\) for which there exist \(p\in\mathbb{R}^n_+\) and a static state feedback law \(u=Kx\), with \(K\in\mathcal K(E)\), such that
\[
\begin{aligned}
	\gamma_\star
	=
	\inf_{\gamma,\;p,\;K\in\mathcal K(E)}\quad & \gamma\\
	\text{subject to}\quad
	& \gamma \ge 0,\\
	& p = s + (A+BK)^\top p + K^\top r,\quad 0 \le p \le \gamma\,\mathbf 1.
\end{aligned}
\]
The value \(\gamma_\star\) is the tightest upper bound on the induced \(\ell_1\)-gain from the disturbance sequence \(w\) to the performance output \(y_t=s^\top x_t+r^\top u_t\) when \(x_0=0\). 
\section{Proofs of Main Results}\label{appenB}
\subsection{Proof of Theorem~\ref{thm:l1-dp}}\label{AppendB}
Recall that, for each \((A,B)\in\mathcal M\), the model-specific history update is
\begin{align}\label{hist-var}
    z^{(A,B)}_+
    =
    z^{(A,B)}
    + \gamma^\top \lvert v - A x - B u\rvert,
    \qquad
    x_+ = v,\quad v \ge 0,
\end{align}
We collect all model-specific histories in the data collection \(Z\). Recall also the dynamic programming operators
\[
\mathcal F_u V(x,Z)= \max_{v\ge 0}\Bigl\{s^\top x+r^\top u+V(v,Z_+)\Bigr\},
\qquad
\mathcal F V(x,Z)= \min_{u\in\mathcal U(x)} \mathcal F_u V(x,Z),
\]
with initialization and update
\begin{align}\label{v0vk}
    V_0(x,Z)\coloneqq -\min_{(A,B)\in\mathcal M} z^{(A,B)},
\qquad
V_{k+1}=\mathcal F V_k,\quad k\ge 0.
\end{align}
\noindent\emph{Proof of (i)-(ii).}
We start by showing that the sequence \(V_0,V_1,V_2,\ldots\) is monotonically non-decreasing.
By definition,
\begin{align*}
V_1(x,Z)
&= \min_{u\in\mathcal U(x)}\ \max_{v\ge 0}
\Bigl\{
s^\top x+r^\top u+V_0(v,Z_+)
\Bigr\}\\
&= \min_{u\in\mathcal U(x)}\ \max_{v\ge 0}
\left\{
s^\top x+r^\top u
-\min_{(A,B)\in\mathcal M}
\bigl(z^{(A,B)}+\gamma^\top \lvert v-Ax-Bu\rvert\bigr)
\right\}.

\end{align*}
Positivity of the stage cost follows from
Assumption~\ref{ass:stage cost-pos}, i.e. \(s^\top x+r^\top u\ge 0\).
Therefore,
\[
s^\top x+r^\top u
-\min_{(A,B)\in\mathcal M}
\bigl(z^{(A,B)}+\gamma^\top \lvert v-Ax-Bu\rvert\bigr)
\ge
-\min_{(A,B)\in\mathcal M}
\bigl(z^{(A,B)}+\gamma^\top \lvert v-Ax-Bu\rvert\bigr).
\]
It follows that, for any \((x,Z)\),
\begin{align*}
V_1(x,Z)
\ge
-\max_{u\in\mathcal U(x)}
\min_{(A,B)\in\mathcal M}
\min_{v\ge 0}
\bigl(z^{(A,B)}+\gamma^\top \lvert v-Ax-Bu\rvert\bigr)
=
-\min_{(A,B)\in\mathcal M} z^{(A,B)}
=V_0(x,Z),
\end{align*}
and so $V_1\ge V_0$.
Since \(\mathcal F\) is order-preserving in \(V\), it follows by induction that
\[
V_{k+1}=\mathcal F V_k\ \ge\ \mathcal F V_{k-1}=V_k,
\qquad \text{for all}\ k\ge 0.
\]
We conclude that the sequence \(V_0,V_1,V_2,\ldots\) is monotonically non-decreasing.

Fix \(N\ge 0\). Consider the \(N\)-stage truncated original minimax problem
\begin{equation}
\label{eq:FH-lb}
\begin{aligned}
\inf_{\mu}\ \sup_{\substack{(A,B)\in\mathcal M\\[1pt]
w \,\ge\, -(Ax+Bu)}}\
&\sum_{t=0}^{N}\Bigl(s^\top x_t+r^\top u_t-\gamma^\top|w_t|\Bigr)\\
\text{s.t.}\quad
&x_{t+1}=A x_t+B u_t+w_t,\quad x_0 \in \mathbb{R}_+^n.
\end{aligned}
\end{equation}
The value of Problem~\ref{prob:l1-robust-control} is bounded below by the
expression~\eqref{eq:FH-lb}. The value of~\eqref{eq:FH-lb} grows
monotonically with \(N\), and the value of Problem~\ref{prob:l1-robust-control}
is obtained in the limit as \(N\to\infty\).
Introducing the change of variables \(v_t\coloneqq x_{t+1}\) and using the
model-specific histories from~\eqref{hist-var} renders
\eqref{eq:FH-lb} equivalent to
\begin{equation}
\label{eq:FH-change}
\begin{aligned}
\inf_{\mu}\ \sup_{v \,\ge\, 0}\ &
\sum_{t=0}^{N}\bigl(s^\top x_t+r^\top u_t\bigr)
-\min_{(A,B)\in\mathcal M} z_{N+1}^{(A,B)},
\\
\text{s.t.}\quad
&x_{t+1}=v_t,\quad x_0\in \mathbb{R}_+^n,\\
& z_{t+1}^{(A,B)}
= z_t^{(A,B)}
+ \gamma^\top \lvert v_t - A x_t - B u_t\rvert,
\qquad z_0^{(A,B)}=0,\quad (A,B)\in\mathcal M.
\\
&
Z_t \coloneqq \{z_t^{(A,B)}\}_{(A,B)\in\mathcal M},
\qquad
Z_0 \coloneqq \{0\}_{(A,B)\in\mathcal M},
\\
&
u_t = \eta(x_t,Z_t) \in \mathcal{U}(x_t),
\quad
\forall\,t=0,\ldots,N+1.
\end{aligned}
\end{equation}
Standard minimax dynamic programming argument with terminal cost and recursion given in~\eqref{v0vk}, shows that the value of \eqref{eq:FH-change} is \(V_{N+1}(x_0,0)\). Thus, Problem~\ref{prob:l1-robust-control} has finite value \(J_*(x_0)\) if and only if the nondecreasing sequence \(\{V_k(x_0,0)\}_{k\ge 0}\) is upper bounded. The limit $V_\ast(x_0, 0)$ is equal to
the value of Problem~\ref{prob:l1-robust-control}.
It holds also that \(V_k(x,Z)\leq V_k(x,0)\) for every nonnegative history collection \(Z\). This shows that the limit exists for every such \(Z\) as well.
If problem~\ref{prob:robust-rewritten-min-history} is finite, then \(V_k(x_0,0)\) is bounded above by the value of~\eqref{prob:robust-rewritten-min-history}, so also
$
V_\ast(x_0,0)=\lim_{k\to\infty} V_k(x_0,0)
$
is finite.
This concludes proving items~\textit{(i)}-\textit{(ii)}. Next, we prove item~\textit{(iii)} of the theorem. 

\noindent\emph{Proof of (iii).}
Suppose we could find a value function $\overbar{V}(x,Z)$ satisfying \(V_0 \le \overbar V \le \infty\) and \(\mathcal F_{\overbar\eta (x,\,Z)}\,\overbar V \le \overbar V\), we may then define the sequence \(W_0,W_1,W_2,\ldots\) recursively by
\[
W_0\coloneqq V_0,\qquad
W_{k+1}(x,Z)\coloneqq\mathcal F_{\overbar\eta (x,\, Z)} W_k(x,Z).
\]
By dynamic programming,
\[
W_N(x,0)
=\max_{v\ge 0}
\sum_{t=0}^{N}\bigl(s^\top x_t+r^\top u_t\bigr)
-\min_{(A,B)\in\mathcal M} z_{N+1}^{(A,B)},
\]
where \(x_{t+1}=v_t\), \(u_t=\overbar\eta(x_t,Z_t)\), and \(Z\) is generated by the model-specific historical data update in~\eqref{hist-var}. Therefore, Problem~\eqref{prob:robust-rewritten-min-history} is bounded above by \(\lim_{k\to\infty} W_k(x_0,0)\). The definitions of \(V_k\) and \(W_k\) give by induction \(\overbar V \ge W_k \ge V_k\) for all \(k\), so
\[
V_\ast(x,Z)\le \lim_{k\to\infty} W_k(x,Z)\le \mathcal F_{\overbar\eta (x,\,Z)}\overbar V(x,Z)\leq \overbar V(x,Z).
\]
This proves that \(V_\ast(x_0,0)\le J_{\overbar\mu}(x_0)\le \overbar V(x_0,0)\) when the control law \(\overbar\mu\) is defined by \(u_t=\overbar\eta(x_t,Z_t)\). In particular, the control law is optimal if the equality \(\overbar V(x,Z)=V_\ast(x,Z)\) is attained. \qed

\subsection{Proof of Theorem~\ref{thm:l1-unknown-sign}}\label{thm:l1-unknown-sign-append}
This is a specialization of the more general Theorem~\ref{thm1} to the two model family
$
\mathcal M \coloneqq \{(A,B),(A,-B)\}.
$
Theorem~\ref{thm1} is proven in appendix~\ref{appenBsth}. We verify that the conditions in~\eqref{ineq-cond-jk} of Theorem~\ref{thm1} reduce to conditions
\emph{(i)}--\emph{(iii)} in the statement of Theorem~\ref{thm:l1-unknown-sign}.
Set
\[
p_{++}=p_+,\qquad p_{--}=p_-,
\qquad p_{+-}=p_{-+}=h.
\]
 The conditions in~\eqref{ineq-cond-jk} are covered by considering the following three cases.
\begin{enumerate}
\item[\textit{(i)}] \emph{Case \(i=j=\sigma\).}
Here \(A_i-A_j=0\) and \(B_i-B_j=0\), so the absolute-value term in
\eqref{ineq-cond-jk} vanishes, while
$
\max\!\bigl\{A_ix+B_iK_\sigma x,\ A_jx+B_jK_\sigma x\bigr\}
=
(A+\sigma BK_\sigma)x.
$
Since
\(p_{jk}=p_{ij}=p_\sigma\), condition \eqref{ineq-cond-jk} becomes
\[
p_\sigma^\top x
\ \ge\
s^\top x+r^\top K_\sigma x+p_\sigma^\top(A+\sigma BK_\sigma)x.
\]
Since this must hold for all \(x\in\mathbb R_+^n\), it is equivalent to having
\emph{(i)}.

\item[\textit{(ii)}] \emph{Case \(i=j=\sigma\neq k\).}
Again \(A_i-A_j=0\) and \(B_i-B_j=0\), so the absolute value term vanishes, and
$
\max\!\bigl\{A_ix+B_iK_kx,\ A_jx+B_jK_kx\bigr\}
=
(A+\sigma BK_k)x.
$
Since \(p_{jk}=h\) and \(p_{ij}=p_\sigma\), condition \eqref{ineq-cond-jk} becomes
\[
h^\top x
\ \ge\
s^\top x+r^\top K_kx+p_\sigma^\top(A+\sigma BK_k)x.
\]
Since this holds for all \(x\in\mathbb R_+^n\), it is equivalent to
\emph{(ii)}. Such an \(h\) also satisfies \(h\geq \max\{p_+,p_-\}\), since
\begin{align*}
h^\top x
\geq
s^\top x+r^\top K_kx+p_\sigma^\top(A+\sigma BK_k)x  
\geq
s^\top x+r^\top K_\sigma x+p_\sigma^\top(A+\sigma BK_\sigma)x
=
p_\sigma^\top x .
\end{align*}
i.e., applying the wrong controller leads to an increase in the certified cost.
\item[\textit{(iii)}] \emph{Case \(i\neq j\).}
Now \(A_i=A_j=A\) and \(B_i-B_j=\pm 2B\). Therefore
\[
\max\!\bigl\{A_ix+B_iK_kx,\ A_jx+B_jK_kx\bigr\}
=
\max\!\bigl\{(A+BK_k)x,\ (A-BK_k)x\bigr\}
=
Ax+|BK_kx|,
\]
where the last identity holds componentwise. Moreover,
\[
\frac12\,\gamma^\top\bigl|\,(A_i-A_j)x+(B_i-B_j)K_kx\,\bigr|
=
\gamma^\top|BK_kx|.
\]
Since \(p_{jk}=p_{ij}=h\), substituting into \eqref{ineq-cond-jk} yields \emph{(iii)}, i.e.
\[
h^\top x
\ \ge\
s^\top x+r^\top K_kx+h^\top\!\bigl(Ax+|BK_kx|\bigr)-\gamma^\top|BK_kx|.
\]
\end{enumerate}
Assuming that \emph{(i)}--\emph{(iii)} is true,
from Theorem~\ref{thm1} it follows that the Bellman equality holds 
$
\mathcal F\,\overline V(x,Z)\le \overline V (x,Z)
$
for the choice 
\[
V^{ij}(x,z^{(+)},z^{(-)})
\coloneqq
p_{ij}^\top x-\tfrac12\!\left(z^{(i)}+z^{(j)}\right),
\qquad
\overbar V(x,z^{(+)},z^{(-)})
\coloneqq
\max_{i,j\in\{+,-\}}V^{ij}(x,z^{(+)},z^{(-)}). 
\]
The stated performance bound follows directly from Theorem~\ref{thm1}, i.e. $
J_{\mu}(x_0)\le \max_{i,j} p_{ij}^\top x_0=h^\top x_0.
$
Finally, assume in addition that \(m=1\) hence $K_k$ can take two possible values \(K_k=\pm E\). Then \(K_kx=\pm Ex\)
is a scalar. Since \(E\in\mathbb R_+^{1\times n}\) and \(x\in\mathbb R_+^n\), we have
\(Ex\ge 0\), and hence
$
|BK_kx|
=
|B|\,|K_kx|
=
|B|\,Ex
=
(|B|E)x.
$
Substituting this identity into \emph{(iii)} gives
\[
h^\top x
\ \ge\
s^\top x+r^\top K_kx+h^\top Ax+(h-\gamma)^\top(|B|E)x,
\qquad \forall x\in\mathbb R_+^n.
\]
Since this inequality must hold for all \(x\in\mathbb R_+^n\), it is equivalent to 
 \emph{(iii')}. \qed
\subsection{Proof of Corollary~\ref{thm:l1-scalar}}\label{thm:l1-scalar-appe}
To prove Corollary~\ref{thm:l1-scalar}, we seek the smallest value of
\(\gamma\) for which the conditions in
Theorem~\ref{thm:l1-unknown-sign} hold, since this yields the best gain and cost
guarantee achieved by the CE policy proposed in
Theorem~\ref{thm:l1-unknown-sign}. First, 
observe that conditions
\textit{(i)}--\textit{(iii)} in Theorem~\ref{thm:l1-unknown-sign}, under the
choices \(n=m=1\), \(A=a>0\), \(B=b>0\), \(s=1\), \(r=0\), and $E=\tfrac{a}{b}$ reduce to the simpler
feasibility conditions
\[
p\geq 1,\qquad
h\geq 1+2ap,\qquad
h\geq 1+2ah-\gamma a,
\]
with \(p,h\in[0,\gamma]\). Next, we find the smallest $\gamma$ for which these three scalar inequalities hold.
Since \(p\geq 1\), we necessarily have
$
h\geq 1+2a,
$
and hence every feasible \(\gamma\) satisfies
$
\gamma\geq h\geq 1+2a.
$
The last inequality can be equivalently written as
$
h(1-2a)\geq 1-\gamma a.
$
We now distinguish two cases.
\begin{itemize}
    \item[\textit{(1)}] Case \(0\leq a\le\tfrac12\).
    Choose \(p=1\) and \(h=\gamma=1+2a\). Then the first two inequalities hold with equality. Moreover,
    \[
    h(1-2a)-(1-\gamma a)
    =
    (1+2a)(1-2a)-\bigl(1-a(1+2a)\bigr)
    =
    a(1-2a)\ge0.
    \]
    Thus \(\gamma=1+2a\) is feasible. Since no feasible \(\gamma\) can be smaller than \(1+2a\), this value is minimal.

    \item[\textit{(2)}] Case \(a\ge\tfrac12\).
    Since \(1-2a\le0\), the inequalities \(h\ge1+2a\) and \(h(1-2a)\ge1-\gamma a\) imply
    \[
    1-\gamma a
    \le h(1-2a)
    \le (1+2a)(1-2a)
    =
    1-4a^2,
    \]
    and therefore \(\gamma\ge4a\). Conversely, choose \(p=1\), \(h=1+2a\), and \(\gamma=4a\). Then \(h\le\gamma\) because \(a\ge\tfrac12\), and the last inequality holds with equality:
    \(h(1-2a)=1-4a^2=1-\gamma a\). Hence \(\gamma=4a\) is feasible and minimal.
\end{itemize}
Therefore, \(\gamma^{\mathrm{CE}}(a)=1+2a\) for \(0\leq a\le \tfrac12\), and \(\gamma^{\mathrm{CE}}(a)=4a\) for \(a\ge \tfrac12\), with minimizing values \(p=1\) and \(h=1+2a\). Moreover, Theorem~\ref{thm:l1-unknown-sign} implies that the Bellman inequality holds and certifies the cost bound
$
J_\mu(x_0)\le h x_0=(1+2a)x_0.
$
This proves the claim of the corollary.
\qed
\subsection{Proof of Theorem~\ref{thm1} }\label{appenBsth}
For $\mathcal M = \{(A_1,B_1),\ldots,(A_M,B_M)\}$, let
$z^{(i)}$ denote the scalar cumulative history associated with model $(A_i,B_i)$, evolving according to
\[ z_{t+1}^{(i)} = z_t^{(i)} + \gamma^\top\!\bigl|\,v_t-A_i x_t-B_i u_t\,\bigr|, \qquad z_0^{(i)}=0. \]
Collect these into
$
Z \coloneqq \bigl(z^{(1)},\ldots,z^{(M)}\bigr).
$ For \(i,j \in \{1,\ldots,M\}\), define
\[
V^{ij}(x,Z)
\;\coloneqq\;
p_{ij}^\top x \;-\; \tfrac{1}{2}\,\bigl(z^{(i)} + z^{(j)}\bigr).
\]
Next, define
$
\overbar V\coloneqq\max_{i,j}V^{ij}.
$
In proving this Theorem, we rely on item~\textit{(iii)} of theorem~\ref{thm:l1-dp}, that is we show that $\overbar V$ fulfills the sufficient conditions proposed there. 
First, note that \(\overbar V \ge V_0=-\min_i z^{(i)}\), indeed, it holds that
\begin{align*}
    \max_{i,j} \Bigl(p_{ij}^\top x - \tfrac{1}{2}\bigl(z^{(i)} + z^{(j)}\bigr)\Bigr)
    \ge \max_{i,j} \Bigl(- \tfrac{1}{2}\bigl(z^{(i)} + z^{(j)}\bigr)\Bigr) 
    = -\min_{i,j} \tfrac{1}{2}\bigl(z^{(i)} + z^{(j)}\bigr) \ge -\min_i z^{(i)}.
\end{align*}
The first inequality follows from \(p_{ij}^\top x \ge 0\), while the last follows by choosing \(j=i\), which gives
$
-\tfrac{1}{2}\bigl(z^{(i)}+z^{(i)}\bigr)=-z^{(i)}.
$
Hence, $\overbar V \ge V_0$ as required by  item~\textit{(iii)} in Theorem~\ref{thm:l1-dp}.
It remains to show that 
\[
\mathcal F_{\overbar\eta}\,\overbar V(x,Z)\le \overbar V(x,Z)
\qquad
\text{for all } (x,Z) \in \mathbb R_+^n \times \mathbb R^{\abs{\mathcal{M}}},
\]
where
$
\mathcal F_{\overbar\eta}V(x,Z)\coloneqq \mathcal F_uV(x,Z)\big|_{u=\overbar\eta(x,Z)}.
$
We let $\bar{\eta}(x,Z)=K_kx$, where $k\coloneqq\arg\min_{i\in\{1,\ldots,M\}} z^{(i)}$.
and recall
$
  \mathcal F_u V(x,Z)\coloneqq  \max_{v\ge 0}
\bigl\{
s^\top x+r^\top u+V(v,Z_+)
\bigr\}.$
We distinguish two cases:
\begin{enumerate}
    \item[\textit{(i)}] If \(i=j=k\), then
\begin{multline*}
\mathcal{F}_{K_ix}\, V^{ii}(x,Z)
=
\max_{v\geq 0}
\left\{
s^\top x+r^\top K_i x
+
V^{ii}\left(
v,\,
\begin{bmatrix}
z^{(1)}\\
\vdots\\
z^{(M)}
\end{bmatrix}
+
\begin{bmatrix}
\gamma^\top\!\lvert v-A_1x-B_1K_i x\rvert\\
\vdots\\
\gamma^\top\!\lvert v-A_Mx-B_MK_i x\rvert
\end{bmatrix}
\right)
\right\}\\
= \max_{v \ge 0}
\left\{
p_{ii}^\top v
- z^{(i)}
+s^\top x+r^\top K_i x
-\gamma^\top\!\lvert v - A_i x - B_i K_i x\rvert
\right\} \\[4pt]
= -z^{(i)}+s^\top x+r^\top K_i x
+\max_{v \ge 0} \Bigl\{ p_{ii}^\top v-\gamma^\top\!\lvert v - (A_i x + B_i K_i x)\rvert \Bigr\}.
\end{multline*}
Since \(0\le p_{ii}\le \gamma\), from Lemma~\ref{lem:sup-ptwise-vec} (VL1), it holds that 
\begin{align*}
    -z^{(i)}+s^\top x+r^\top K_i x
+\max_{v \ge 0} \Big\{  p_{ii}^\top v-\gamma^\top\!\lvert v - (A_i x + B_i K_i x)\rvert \Big\}\\
= -z^{(i)} + s^\top x +  p_{ii}^\top A_i x + \big(r^\top + p_{ii}^\top B_i\big)K_i x.
\end{align*}
Since \(p_{ii}^\top x \ge s^\top x + p_{ii}^\top(A_i+B_iK_i)x + r^\top K_i x\), it follows that
\[
\mathcal{F}_{K_i}V^{ii}(x,Z)\le -z^{(i)} + p_{ii}^\top x = V^{ii}(x,Z).
\]
\item[\textit{(ii)}] We now consider all remaining index configurations,
that is, all triples \((i,j,k)\) for which at least two of the indices are
distinct. With control \(K_{k}x\) used, we obtain
\begin{multline*}
\mathcal{F}_{K_kx}V^{ij}(x,Z)
=
\max_{v\geq 0}
\left\{
s^\top x+r^\top K_kx
+
V^{ij}\left(
v,\,
\begin{bmatrix}
z^{(1)}\\
\vdots\\
z^{(M)}
\end{bmatrix}+
\begin{bmatrix}
\gamma^\top \lvert v-A_1x-B_1K_kx\rvert\\
\vdots\\
\gamma^\top \lvert v-A_Mx-B_MK_kx\rvert
\end{bmatrix}
\right)
\right\}\\
=\max_{v\ge0}\Bigl\{
p_{ij}^\top v
-\tfrac12\bigl(z^{(i)}+z^{(j)}\bigr)
+s^\top x+r^\top K_kx
-\tfrac12\gamma^\top\lvert v-A_ix-B_iK_kx\rvert
-\tfrac12\gamma^\top\lvert v-A_jx-B_jK_kx\rvert
\Bigr\}\\[-2pt]
=-\tfrac12\bigl(z^{(i)}+z^{(j)}\bigr)
+s^\top x+r^\top K_kx+\max_{v\ge0}\Bigl\{
p_{ij}^\top v
-\tfrac12\gamma^\top\lvert v-A_ix-B_iK_kx\rvert
-\tfrac12\gamma^\top\lvert v-A_jx-B_jK_kx\rvert
\Bigr\}.
\end{multline*}
Let
$
y_i \coloneqq A_i x + B_i K_{k} x $, $y_j \coloneqq A_j x + B_j K_{k} x,
$
and recall that \(y_i\ge 0\), \(y_j\ge 0\) and \(0\le p_{ij}\le \gamma\).
We are then solving the optimization problem
\[
\max_{v\ge 0}\ \Big\{\, p_{ij}^\top v
-\tfrac{1}{2}\,\gamma^\top\!\big|v-y_i\big|
-\tfrac{1}{2}\,\gamma^\top\!\big|v-y_j\big| \,\Big\}.
\]
It follows from Lemma~\ref{lem:sup-ptwise-vec} (VL2) that 
\begin{align}\label{advers}
    v^\ast=\max\{y_i,y_j\}=\max\{A_i x + B_i K_{k} x,A_j x + B_j K_{k} x\},
\end{align}
and the optimal value equals
$
p_{ij}^\top \max\{y_i,y_j\}
\;-\;\tfrac{1}{2}\,\gamma^\top \big|\,y_i-y_j\,\big|.
$
Plugging this in gives
\begin{multline*}
\mathcal{F}_{K_k x}V^{ij}(x,Z)
=-\tfrac12\bigl(z^{(i)}+z^{(j)}\bigr)
+s^\top x+r^\top K_kx
\\
+p_{ij}^\top\max\{A_ix+B_iK_kx,\ A_jx+B_jK_kx\}
-\tfrac12\gamma^\top\bigl|A_ix+B_iK_kx-A_jx-B_jK_kx\bigr|.
\end{multline*}
The assumption in \eqref{ineq-cond-jk} can be written succinctly as
\begin{multline*}
  \max \{p_{ik}^\top x,\, p_{jk}^\top x\}  \;\ge\; 
  s^\top x \;+\; r^\top K_{k} x 
 \\ \;+\; p_{ij}^\top \max\{A_i x + B_i K_{k} x,\ A_j x + B_j K_{k} x\}
\;-\;\tfrac{1}{2}\,\gamma^\top \big|\,A_i x + B_i K_{k} x \;-\; A_j x - B_j K_{k} x\,\big|,
\end{multline*}
where \(0 \leq p_{ij} = p_{ji} \leq \gamma\). Here, the \(\max\) on the left excludes the case \(i \neq j = k\), consistent with the assumption in the theorem.
It then follows that
\begin{multline*}
\mathcal{F}_{K_k}V^{ij}(x,Z)
\le -\tfrac12\bigl(z^{(i)}+z^{(j)}\bigr)+\max\{p_{ik}^\top x,p_{jk}^\top x\}
\\\le\max\Bigl\{p_{ik}^\top x-\tfrac12\bigl(z^{(i)}+z^{(k)}\bigr),
p_{jk}^\top x-\tfrac12\bigl(z^{(j)}+z^{(k)}\bigr)\Bigr\},
\end{multline*}
where the last step uses \(k=\arg\min_r z^{(r)}\), so \(z^{(k)}\le z^{(i)}\) and \(z^{(k)}\le z^{(j)}\).
\end{enumerate}
Since by definition
$
V^{ij}(x,Z)=p_{ij}^\top x-\tfrac12\bigl(z^{(i)}+z^{(j)}\bigr),
$
we obtain
\[
\mathcal{F}_{K_k}V^{ij}(x,Z)\le\max\{V^{ik}(x,Z),V^{jk}(x,Z)\}.
\]
Consequently,
\begin{multline*}
\mathcal{F}\overbar V(x,Z)
\le\mathcal{F}_{K_k}\overbar V(x,Z)
=\max_{i,j}\mathcal{F}_{K_k}V^{ij}(x,Z)
\le\max_{i,j}\max\{V^{ik}(x,Z),V^{jk}(x,Z)\}\\
\le\max_{i,j}V^{ij}(x,Z)
=\overbar V(x,Z).
\end{multline*}
This establishes the Bellman inequality
$
\mathcal{F}\overbar V(x,Z)\le\overbar V(x,Z).
$
Moreover, the corresponding certainty--equivalence control law certifying the Bellman inequality is
\[
u_t=K_{k_t}x_t,\qquad
k_t=\arg\min_{i\in\{1,\ldots,M\}}z_t^{(i)}.
\]
This concludes the proof.
\qed